\documentclass{amsart}

\usepackage{amsmath,amsthm,amssymb}

\usepackage{graphicx}

\usepackage{hyperref}

\newcommand{\one}{\mathit{1}}
\newcommand{\two}{\mathit{2}}
\newcommand{\three}{\mathit{3}}
\newcommand{\four}{\mathit{4}}

\newcommand{\onep}{\mathbf{1}}
\newcommand{\twop}{\mathbf{2}}

\newcommand{\unit}{\varepsilon}

\newcommand{\C}{\mathbb{C}}
\newcommand{\D}{\mathcal{D}}
\newcommand{\G}{\Gamma}
\newcommand{\M}{\mathcal{M}}
\newcommand{\R}{\mathbb{R}}
\newcommand{\Z}{\mathbb{Z}}
\newcommand{\T}{\mathcal{T}}
\newcommand{\X}{\mathcal{X}}

\newcommand{\bim}{\mathfrak{M}}
\newcommand{\nuke}{\mathcal{N}}
\newcommand{\img}[1]{\mathop{\mathrm{IMG}}(#1)}
\newcommand{\arr}{\rightarrow}
\newcommand{\alb}{\mathsf{X}}
\newcommand{\xs}{\alb^*}

\newcommand{\xmo}{\alb^{-\omega}}
\newcommand{\symm}[1][\alb]{\mathfrak{S}\left(#1\right)}

\newcommand{\CP}{\mathbb{P}\C^2}
\newcommand{\CS}{\widehat{\C}}
\newcommand{\wt}[1]{\widetilde{#1}}
\newcommand{\wh}[1]{\widehat{#1}}
\newcommand{\lims}[1][G]{\mathcal{J}_{#1}}
\newcommand{\limg}[1][G]{\mathcal{X}_{#1}}
\newcommand{\si}{\mathsf{s}}

\newcommand{\til}{\mathcal{T}}

\newcommand{\mapdown}[1]%
{\Big\downarrow\rlap{$\vcenter{\hbox{$\scriptstyle#1$}}$}}

\title{The Julia set of a post-critically finite endomorphism of
$\CP$}
\author{Volodymyr Nekrashevych}\thanks{This paper is based upon work supported by the
National Science Foundation under Grants DMS0605019 and
DMS0800085.}

\newtheorem{theorem}{Theorem}[section]
\newtheorem{proposition}[theorem]{Proposition}
\newtheorem{corollary}[theorem]{Corollary}
\newtheorem{lemma}[theorem]{Lemma}
\newtheorem*{conjecture}{Conjecture}

\theoremstyle{definition}
\newtheorem{defi}{Definition}[section]

\begin{document}
\maketitle
\begin{abstract}
We construct a combinatorial model of the Julia set of the
endomorphism $f(z, w)=((1-2z/w)^2, (1-2/w)^2)$ of $\CP$.
\end{abstract}
\tableofcontents

\section{Introduction}

J.~E. Forn{\ae}ss and N.~Sibony studied in~\cite{fornsibon:crfin}
two post-critically finite endomorphisms of $\CP$. The Julia set
of one of them is $\CP$, while the Julia set of the other has no
interior.

The latter map appeared independently in~\cite{bartnek:rabbit} as
a natural skew product map in the study of Thurston equivalence of
topological polynomials with the post-critical dynamics of
$z^2+i$. It is written in affine coordinates as
\[f\left(z, w\right)=
\left(\left(1-\frac{2z}{w}\right)^2,\quad
\left(1-\frac{2}{w}\right)^2\right).\]

As a development of~\cite{bartnek:rabbit}, group theoretic aspects
of the map $f$ were studied in~\cite{nek:ssfamilies}
and~\cite{nek:nonuniform}. The iterated monodromy group of $f$ was
used to construct an uncountable family of three-generated groups
with interesting properties, and later a new group of non-uniform
exponential growth was found in this family (the first examples of
groups of non-uniform exponential growth were found
in~\cite{wilson:nonuniform}).

In the current paper we apply our group theoretic knowledge to
description of the Julia set of $f$. We construct a combinatorial
model of the Julia set in the spirit of Hubbard trees. Of course,
here the Hubbard trees become ``Hubbard complexes'' (actually even
complexes of groups); but due to a particular skew product
structure of the map, the Hubbard complex of $f$ is a bundle of
``Hubbard tripods'', and the combinatorial model of the Julia set
can be described in terms of a ``folding map'' on the bundle of
tripods. The Julia set of $f$ is represented then as the
projective limit of a sequences of three-dimensional Hubbard
complexes, which are homeomorphic to subsets of the Julia set.

A general method of constructing similar combinatorial models of
expanding dynamical systems is described in~\cite{nek:models}.
Finding simple and elegant models is still of interest, since
construction from~\cite{nek:models} depends on some choices, which
can be made in different ways, leading to models of different
complexity.

Interesting examples of post-critically finite multi-dimensional
maps come from the study of correspondences on the moduli space of
a punctured sphere,
see~\cite{bartnek:rabbit,koch:french,nek:models}.

Unfortunately, an important ingredient of our analysis is missing.
I was not able to prove that $f$ is sub-hyperbolic, i.e., to
construct a singular metric on a neighborhood of the Julia set of
$f$, such that $f$ is uniformly expanding with respect to it.

\subsection{Overview of the paper}

The second section of the paper collects elementary and previously
known properties of the function $f$. We describe its action on
the projective plane $\CP$, the structure of the post-critical set
of $f$, recall the results of J.~E. Forn{\ae}ss and N.~Sibony and
discuss the skew product structure of the map.

Section ``Techniques'' is an overview of the theory of
self-similar groups, iterated monodromy groups and their limit
spaces. There are no proofs in it, which can be found either
in~\cite{nek:book}, or in~\cite{nek:models}. In particular, the
general notion of combinatorial models of expanding dynamical
systems is described in this section. More on combinatorial models
of hyperbolic dynamical systems, see~\cite{ishiismillie}
and~\cite{nek:models}.

We compute the iterated monodromy group of $f$ in
Section~\ref{s:imgcomputation}. We use an interpretation of the
map $f$ given in~\cite{bartnek:rabbit}, which makes it possible to
compute the iterated monodromy group $\img{f}$ in a relatively
easy way.

The combinatorial model of $f$ is constructed in
Section~\ref{s:polyhedralmodel}. It is convenient to pass to an
index 2 extension $\G$ of $\img{f}$. This extension can be defined
as the iterated monodromy group of the quotient of the dynamical
system $(f, \CP)$ by the group of order two generated by the
transformation $(z, w)\mapsto (\overline z, \overline w)$. The
group $\G$ is generated by a relatively small automaton (of 12
states). It has appeared for the first time
in~\cite{nek:ssfamilies}, where it was used to study a Cantor set
of groups associated with $f$. We continue the study of the group
$\G$ in Section~\ref{s:polyhedralmodel} of our paper. In
particular, we describe its nucleus, which happens to be a union
of six finite groups. We use the poset of subgroups of the nucleus
to construct a simplicial complex, serving as the first
approximation of the limit space of $\G$.

This complex consists of three tetrahedra with one common face.
The corresponding approximation of the Julia set of $f$ is
obtained by pasting together two copies of this complex. We
construct the combinatorial model of the Julia set (in
Subsection~\ref{ss:iota}); give an inductive cut-and-paste rule
for constructing a sequences of polyhedra approximating the Julia
set (Theorem~\ref{th:pastingTn} and
Proposition~\ref{pr:whMnpasting}); and prove that Julia set is the
inverse limit of the constructed polyhedra
(Theorem~\ref{th:JuliasetMn}).

The cut-and-paste rule works as follows. The $n$th level
approximation $\M_n$ of the Julia set is obtained by pasting
together two copies of a complex $\T_n$ along their ``boundary''.
The boundary of $\T_n$ is decomposed into a union of 11 regions,
which are domains of involutive maps $\kappa_{g, n}$ (``pasting
rules''). The next complex $\T_{n+1}$ is obtained by pasting four
copies of $\T_n$ using five of the maps $\kappa_{g, n}$. The 11
pieces of the boundary of $\T_{n+1}$ and the corresponding maps
$\kappa_{g, n+1}$ are defined then as unions of the pieces of the
boundaries of the copies of $\T_n$ according to rules described by
the finite automaton generating the group $\G$.

The rules are not very complicated, but since the complexes are
three dimensional and can not be embedded into $\R^3$ without
self-intersections, it is hard to visualize them.

In order to understand better the constructed polyhedral model of
the Julia set, we use the skew product structure of $f$ (and of
the model), and study it ``fiberwise'' in Section~\ref{s:tripods}.
We show that the fibers of the complexes $\M_n$ are trees that can
be constructed using natural ``folding'' and ``unfolding''
transformation (Theorem~\ref{th:fibersrhon} and
Proposition~\ref{pr:fiberinvlim}). As a limit of iterations of the
folding and unfolding procedures we get dendrites homeomorphic to
the intersections of the Julia set of $f$ with the lines $w=w_0$.

Our models are very similar to the usual Hubbard trees, since the
approximating complexes $\M_n$ are homeomorphic to the subsets
$f^{-n}(\M)$, where $\M$ is a ``span'' of the post-critical set of
$f$ inside the Julia set of $f$ (Proposition~\ref{pr:MninJulia}).

The last section of our paper contains additional results deduced
from the model of the Julia set and from properties of the
iterated monodromy group of $f$. We construct a family of length
metrics on the slices of the Julia set of $f$; define a family of
natural surjections of the slices onto an isosceles right triangle
(this includes, for instance, the Sierpi\'nski plane-filling curve
as a particular case); describe when the slices are finite trees;
and describe the action of $f$ on the manifold of ``external
rays'' of the Julia set of $f$. It is shown that the manifold of
external rays is an orbispace with the universal covering
identified with the real Heisenberg group, so that the action of
$f$ is induced by an expanding automorphism of the Heisenberg
group. Note that in the classical case of polynomials of degree
$d$ the space of external rays is $\R/\Z$ with the action of the
polynomial induced by the automorphism $x\mapsto d\cdot x$ of
$\R$.

\section{The rational function}
\label{s:thefunction}

Consider the transformation of $\C^2$
\[f\left(z, w\right)=
\left(\left(1-\frac{2z}{w}\right)^2,\quad
\left(1-\frac{2}{w}\right)^2\right).\] It can be extended to a map
$f:\mathbb{PC}^2\arr\mathbb{PC}^2$ as
\[f:[z:w:u]\mapsto[(w-2z)^2:(w-2u)^2:w^2].\]
The Jacobian of the map $f$ is then
\[\left|\begin{array}{ccc}-4(w-2z) & 0 & 0\\
2(w-2z) & 2(w-2u) & 2w\\
0 & -4(w-2u) & 0\end{array}\right|=-32(w-2z)(w-2u)w,\] Hence, the
critical locus is the union of the lines $w=2z$, $w=2u$, and
$w=0$. Their orbits are
\[\{w=2z\}\mapsto\{z=0\}\mapsto\{z=u\}\mapsto\{z=w\}\mapsto\{z=u\}\]
and
\[\{w=2u\}\mapsto\{w=0\}\mapsto\{u=0\}\mapsto\{w=u\}\mapsto\{w=u\}.\]

Consequently, the post-critical set is the union of six lines
\[z=0,\quad z=u,\quad z=w,\quad w=0,\quad w=u,\quad u=0,\] or, in affine coordinates: $z=0$,
$z=1$, $z=w$, $w=0$, $w=1$, and the line at infinity.

The function $f$ appeared in~\cite{bartnek:rabbit}, where it was
used to answer a question of J.~Hubbard and A.~Duady
from~\cite{DH:Thurston} on combinatorial equivalence of some
branched coverings of the plane. Groups associated with it were
studied in~\cite{nek:ssfamilies} and~\cite{nek:nonuniform}.

This function is conjugate to the function
\[\wt f([z:w:t])=[(z-2w)^2:z^2:(z-2t)^2],\]
considered by J.~E. Forn{\ae}ss and N.~Sibony
in~\cite{fornsibon:crfin}. The conjugating map is
\[z\mapsto w,\quad w\mapsto u,\quad t\mapsto z,\]
where the variables on the left-hand side are
from~\cite{fornsibon:crfin}, while the variables from the
right-hand side are the ones used in our paper.

The following properties of the map $f$ are proved
in~\cite{fornsibon:crfin}. Denote by $V$ the post-critical set of
$f$ and by $W$ its full preimage $f^{-1}(V)$.

\begin{theorem}
\label{th:fornsibon} The sets $\CP\setminus V$ and $\CP\setminus
W$ are Kobayashi hyperbolic and the map $f:\CP\setminus
W\arr\CP\setminus V$ is noncontracting in the infinitesimal
Kobayashi metric on $\CP\setminus V$.

The point $[1:0:0]$ is a superattracting fixed point. Let $U$ be
its basin of attraction and let $J=\CP\setminus U$ be its
complement. Then $U$ is a topological cell and a Kobayashi
hyperbolic domain of holomorphy.

The set $J$ has no interior and is the Julia set of $f$. The map
$f$ is topologically transitive on $J$ and repelling periodic
points are dense in $J$.
\end{theorem}

The aim of our paper is to describe the combinatorics and topology
of the Julia set $J$ in the spirit of Hubbard trees
(see~\cite{DH:orsayI,DH:orsayII,poirer:classification2}).

Theorem~\ref{th:fornsibon} is not quite what we need to be able to
apply the techniques of the iterated monodromy groups to the study
of the Julia set of $f$. The results of our paper are therefore
true only modulo the following conjecture.

\begin{conjecture}
There exists an orbispace metric on $\CP\setminus\{[1:0:0]\}$ such
that $f$ is uniformly expanding with respect to this metric on a
neighborhood of the Julia set.
\end{conjecture}

In fact, some weaker results would be sufficient, but they are
probably equivalent to the above conjecture. It would be very nice
to have a general statement about sub-hyperbolicity of
post-critically finite endomorphisms of complex projective spaces.

Note that the map $f$ has a skew product structure: the second
coordinate is a rational function depending only on the second
coordinate. On the first coordinates of iterations of $f$ we get
compositions of quadratic polynomials $f_w(z)=\left(1-\frac 2w
z\right)^2$, i.e., non-autonomous iteration of quadratic
polynomials. The critical point $z=w/2$ of the polynomial $f_w$ is
mapped to $0$, $f_w(0)=1$, and $f_w(1)$ is equal to the next value
$\left(1-\frac 2w\right)^2$ of the second coordinate in the
iteration of $f$. We see that the non-autonomous iteration on the
first coordinate is post-critically finite: the set of critical
values of the composition $f_{w_n}\circ
f_{w_{n-1}}\circ\cdots\circ f_{w_1}$ belongs to the set $\{0, 1,
w_{n+1}\}$, where $w_{i+1}=\left(1-\frac 2{w_i}\right)^2$ for
$i=1, \ldots, n$. For more on post-critically finite
non-autonomous iterations of polynomials, see~\cite{nek:polynom}.

It follows from the skew-product structure of the map $f$ that it
agrees with the projection $P:[z:w:u]\mapsto [w:u]$ (defined on
the complement of the point $[1:0:0]$), which is written in affine
coordinates as $P:(z, w)\mapsto w$. Namely, the fibers of the
projection are mapped by $f$ to fibers. In particular, by
Theorem~\ref{th:fornsibon}, the fiber $P^{-1}(w)\cap J$ of the
projection $P:J\arr\widehat\C$ of the Julia set of $f$ onto the
sphere (which is the Julia set of $\left(1-\frac 2w\right)^2$) is
the Julia set of the non-autonomous iteration
\[\C\xrightarrow{f_{w_0}}\C\xrightarrow{f_{w_1}}\C\xrightarrow{f_{w_2}}\cdots,\]
where $w_0=w$ and $w_{i+1}=\left(1-\frac 2w\right)^2$.

This makes it possible to draw the slices $P^{-1}(w)\cap J$ of the
Julia set of $f$ in $z$-planes. See some of such slices on
Figure~\ref{fig:limspaces}.

\begin{figure}
\centering \includegraphics[width=4in]{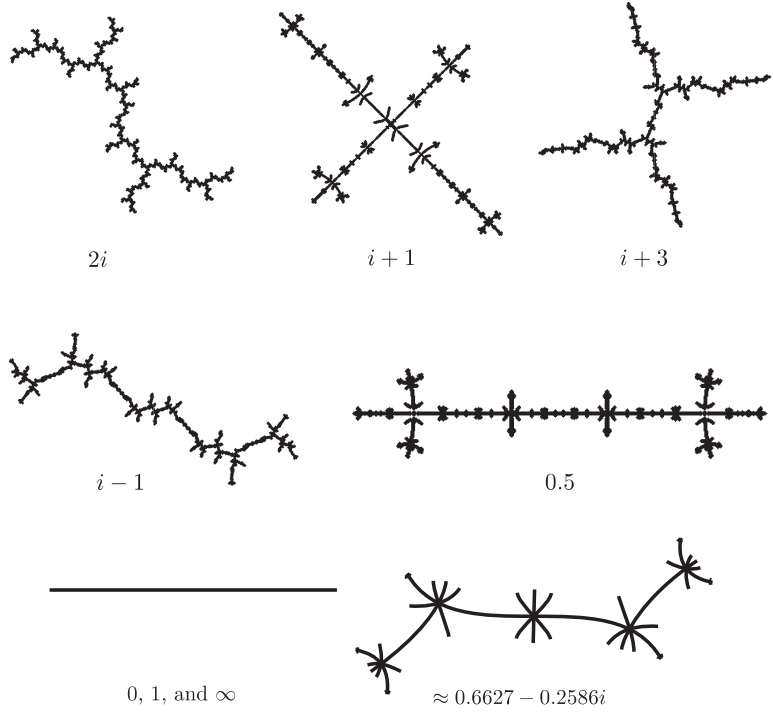}\\
 \caption{Intersections of the Julia set of $f$ with the $z$-planes}\label{fig:limspaces}
\end{figure}

The rational function appearing on the second coordinate of $f$ is
a Latt\`es example. Namely, it is semiconjugate to the map
$z\mapsto (i-1)z$ on $\C$, where the semiconjugacy is the map
\[\C\arr\mathbb{PC}^1:z\mapsto (\wp(z))^2,\]
where $\wp$ is the Weierstrass' function associated with the
lattice of Gaussian integers $\Z[i]$. See a proof of this fact
in~\cite{milnor:dragons} and~\cite{bartnek:rabbit}. In other
words, the rational function $\left(1-\frac 2w\right)^2$ is
conjugate to the map induced by $z\mapsto (i-1)z$ on the orbifold
of the action of the group of orientation-preserving isometries of
the lattice $\Z[i]$ of Gaussian integers.

\section{Techniques}

\subsection{Self-similar groups}

\begin{defi}
A \emph{wreath recursion} is a homomorphism $\psi:G\arr\symm\wr G$
from a group $G$ to the wreath product of $G$ with the symmetric
group $\symm$.
\end{defi}

Here symmetric group $\symm$ acts on $\alb$ from the left. We will
denote the identity element of a group by $\unit$.

The wreath product $\symm\wr G=\symm\ltimes G^\alb$ is the set of
pairs $(\sigma, (g_x)_{x\in\alb})$, where $\sigma\in\symm$ and
$(g_x)_{x\in\alb}$ is an element of the direct product $G^\alb$.
The elements of the wreath product are multiplied by the rule
\[(\sigma, (g_x)_{x\in\alb})\cdot(\pi, (h_x)_{x\in\alb})=
(\sigma\pi, (g_{\pi(x)}h_x)_{x\in\alb}).\]

We will write the pair $(\sigma, (g_x)_{x\in\alb})$ as a product
$\sigma(g_x)_{x\in\alb}$, identifying $\sigma$ with $(\sigma,
(\unit)_{x\in\alb})$ and $(g_x)_{x\in\alb}$ with $(\unit,
(g_x)_{x\in\alb})$. It is easy to see that this identification
agrees with the multiplication rule. If $\alb=\{\one, \two,
\ldots, d\}$, then we write $\sigma(g_x)_{x\in\alb}$ as a sequence
$\sigma(g_{\one}, g_{\two}, \ldots, g_d)$.

Let $\psi:G\arr\symm\wr G$ be a wreath recursion. Then the
\emph{associated permutation $G$-bimodule} $\bim$ is the set
$\alb\times G$ together with two (left and right) actions of $G$
on it that are given by the rules
\[(x, h)\cdot g=(x, hg),\quad g\cdot (x, h)=(\sigma(x), g_xh),\]
for $\psi(g)=\sigma(g_x)_{x\in\alb}$. We will identify $x\in\alb$
with $(x, \unit)\in\bim$ and write $(x, g)=x\cdot g$.

If we compose a wreath recursion $\psi:G\arr\symm\wr G$ with an
inner automorphism of $\symm\wr G$, then we do not change the
isomorphism class of the associated bimodule
(see~\cite[Proposition~2.3.4]{nek:book} or~\cite[Proposition~2.22]{nek:filling}).

\begin{defi}
A \emph{self-similar group} $(G, \alb)$ is a group $G$ together
with a wreath recursion $\psi:G\arr\symm\wr G$. Self-similar
groups defined by wreath recursions that differ by an inner
automorphism of $\symm\wr G$ are called \emph{equivalent}.
\end{defi}

Equivalently, a self-similar group is a group $G$ together with a
\emph{covering bimodule} $\bim$, as it is defined below.

\begin{defi}
\label{def:bimodules} Let $G$ be a group. A \emph{permutational
$G$-bimodule} is a set $\bim$ with commuting left and right
actions of $G$ on it, i.e., maps $G\times\bim\arr\bim:(g,
x)\mapsto g\cdot x$ and $\bim\times G\arr\bim:(x, g)\mapsto x\cdot
g$ such that $\unit\cdot x=x\cdot\unit=x$ for all $x\in\bim$ and
\begin{gather*}g_1\cdot (g_2\cdot
x)=g_1g_2\cdot x,\qquad (x\cdot g_1)\cdot g_2=x\cdot g_1g_2,\\
(g_1\cdot x)\cdot g_1=g_1\cdot (x\cdot g_2),
\end{gather*}
for all $g_1, g_2\in G$ and $x\in\bim$.

A permutational $G$-bimodule is called a \emph{covering $d$-fold}
bimodule if the right action of $G$ on $\bim$ has $d$ orbits and
is free, i.e., if $x\cdot g=x$ implies $g=\unit$.
\end{defi}

It is easy to see that the bimodule associated with a wreath
recursion $\psi:G\arr\symm\wr G$ is a covering $d$-fold bimodule
for $d=|\alb|$. In the other direction, if $\bim$ is a covering
$G$-bimodule, then for a given right orbit transversal $\alb$
(i.e., such a subset $\alb\subset\bim$ that every right orbit
contains exactly one element of $\alb$), we get the associated
wreath recursion
\[g\mapsto\sigma(g|_x)_{x\in\alb},\]
where $\sigma\in\symm$ and $g|_x$ are given by the condition
\[g\cdot x=\sigma(x)\cdot g|_x\]
in $\bim$. The permutation $\sigma$ is the associated action of
$g$ on $\alb$. We get in this way an action of $G$ on $\alb$
defined by the wreath recursion. More formally, this action is
obtained by composing the wreath recursion with the projection of
$\symm\wr G$ onto $\symm$. We will usually denote
$g(x)=\sigma(x)$.

We say that a subset $\alb\subset\bim$ is a \emph{basis} of the
bimodule $\bim$, if it a right orbit transversal. We will usually
label the letters of $\alb$ by integers $\one, \two, \ldots, d$,
thus identifying $\symm$ with $\symm[d]$. It is not hard to prove
that changing the basis $\alb$ amounts to composing the associated
wreath recursion by an inner automorphism of $\symm[d]\wr G$
(see~\cite[Subsection~2.3.4]{nek:book}
and~\cite[Proposition~2.22]{nek:filling}).

The following result is proved in~\cite[Section~2.5]{nek:book}.

\begin{proposition}
\label{pr:virtwreath} Let $G$ be a self-similar group with the
associated wreath recursion $\psi:G\arr\symm\wr G$. Suppose that
the associated action on $\alb$ is transitive. Then the wreath
recursion, up to composition with an inner automorphism of
$\symm\wr G$, is uniquely determined for any $x\in\alb$ by the
homomorphism
\[g\mapsto g_x,\]
from the stabilizer of $x$ into $G$. Here $g_x$ is the coordinate
of $\psi(g)=\sigma(g_x)_{x\in\alb}$ corresponding to $x$.
\end{proposition}

The homomorphism $g\mapsto g_x$ defined on the stabilizer of $x$
is called the \emph{virtual endomorphism} associated with the
wreath recursion (or with self-similarity of the group).

The following proposition gives formulae for the wreath recursion
with a given associated virtual endomorphism.

\begin{proposition}
\label{prop:formulawreathrec} Let $\phi:G_1\arr G$ be a virtual
endomorphism of a group $G$, where $G_1$ is a subgroup of finite
index in $G$. Let $\{r_1, r_2, \ldots, r_d\}$ be a left coset
transversal of $G$ modulo $G_1$ (i.e., a set such that $G$ is a
disjoint union of the cosets $r_i\cdot G_1$). Let
$\{x_i\}_{i=1,\ldots, d}=\alb$ be an alphabet of size $d$. For
$g\in G$ put $\psi(g)=\sigma(g_x)_{x\in\alb}$, where
$\sigma(x_i)=x_j$ if $gr_i\in r_j G_1$, and
$g_{x_i}=\phi(r_j^{-1}gr_i)$. Then $\psi:G\arr\symm\wr G$ is a
wreath recursion such that $\phi$ is associated with it.
\end{proposition}

\subsection{Iteration of the wreath recursion}

Let $(G, \alb)$ be a self-similar group and let $\bim=\alb\times G$
be the associated permutational bimodule. For $x\in\alb$ and $g\in
G$ we denote
\[g\cdot x=g(x)\cdot g|_x.\]
If $\psi$ is the associated wreath recursion, then for
$\psi(g)=\sigma(g_x)_{x\in\alb}$ we have $g(x)=\sigma(x)$ and
$g|_x=g_x$.

We define then inductively, for a finite word $v=x_1x_2\ldots
x_n\in\xs$ and $g\in G$, a word $g(v)$ and an element $g|_v\in G$
by the rule
\[g(xv)=g(x)g|_x(v),\qquad g|_{xv}=g|_x|_v.\]
It is easy to see that for every $n\ge 1$ the map $\sigma_{n,
g}:v\mapsto g(v)$ is a permutation of the set $\alb^n$ and that
the map
\[\psi^{\otimes n}:G\arr\symm[\alb^n]\wr G:g\mapsto
\sigma_{n, g}(g|_v)_{v\in\alb^n}\] is a homomorphism. The wreath
recursion $\psi^{\otimes n}$ is called the \emph{$n$th iteration}
of the wreath recursion $\psi$.

Let $\xs=\bigsqcup_{n\ge 0}\alb^n$ be the rooted tree of finite
words over $\alb$, where every word $v\in\xs$ is connected to the
words of the form $vx$ for $x\in\alb$. The empty word is the root
of the tree $\xs$. It is easy to check that for every $g\in G$ the
permutation $v\mapsto g(v)$ of $\xs$ is an automorphism of the
rooted tree $\xs$, and that in this way we get an action of the
group $G$ on the tree $\xs$. It is called the \emph{action
associated with the bimodule (with the wreath recursion)}. This
action, up to conjugacy of actions, depends only on the associated
permutational bimodule (does not depend on the choice of the basis
$\alb$).

\begin{defi}
Let $\psi:G\arr\symm\wr G$ be a wreath recursion. A \emph{faithful
quotient of $G$} (with respect to $\psi$) is the quotient of $G$
by the kernel of the action on $\xs$ associated with the wreath recursion
$\psi$.
\end{defi}

The wreath recursion is interpreted then as a recurrent
description of the action of the group elements on the tree $\xs$.
For $g\in G$ such that $\psi(g)=\sigma(g|_x)_{x\in\alb}$, the
permutation $\sigma\in\symm$ describes the action of $g$ on the
first level $\alb$ of the tree $\xs$, while the coordinates $g|_x$
describe the action on the subtree $x\xs$, so that
\[g(xv)=\sigma(x)g|_x(v)\]
for all $v\in\xs$. In general, we have
\[g(vw)=g(v)g|_v(w)\]
for all $v, w\in\xs$ and $g\in G$. Note that if the action of $G$
on $\xs$ is faithful, then the above equality uniquely determines
$g|_v$. The elements $g|_v$ are called the \emph{sections} of the
element $g\in G$.

If the action is faithful, then we identify the elements of the
group $G$ with the corresponding automorphisms of the rooted tree
$\xs$. We will usually omit then the letter denoting the wreath
recursion and write $g=\sigma(g_1, \ldots, g_d)$ instead of
$\psi(g)=\sigma(g_1, \ldots, g_d)$, naturally identifying the
automorphism group $\mathrm{Aut}(\xs)$ with the wreath product
$\symm\wr\mathrm{Aut}(\xs)$.

Iterations of wreath recursions correspond to tensor
powers of the associated bimodule.

\begin{defi}
Let $\bim_1$ and $\bim_2$ be permutational $G$-bimodules (i.e.,
sets with commuting left and right actions of $G$). Then their
\emph{tensor product} $\bim_1\otimes\bim_2$ is the quotient of the
direct product $\bim_1\times\bim_2$ by the identifications
\[x_1\cdot g\otimes x_2=x_1\otimes g\cdot x_2\]
together with the actions \[g_1\cdot (x_1\otimes x_2)\cdot
g_2=(g_1\cdot x_1)\otimes (x_2\cdot g_2).\] If $\bim_2$ is a set
with a right (resp.\ left) action of $G$ and $\bim_1$ is a
$G$-bimodule, then the right $G$-space $\bim_1\otimes\bim_2$
(resp.\ left $G$-space $\bim_2\otimes\bim_1$) are defined in a
similar way.
\end{defi}

One can show that the bimodule associated with the $n$th iteration
$\psi^{\otimes n}$ of a wreath recursion $\psi$ is isomorphic to
the $n$th tensor power $\bim^{\otimes n}$ of the bimodule $\bim$
associated with $\psi$.

If the associated action on $\xs$ is \emph{level-transitive}
(i.e., transitive on the levels $\alb^n$ of $\xs$), then the
virtual endomorphism associated with the $n$th iterate
$\psi^{\otimes n}$ of the wreath recursion is conjugate (i.e., is
equal, up to inner automorphisms of the group) to the $n$th
iterate of the virtual endomorphism associated with $\psi$.

\begin{defi}
Let $(G, \alb)$ be a self-similar group. A subset $A\subset G$ is
said to be \emph{state-closed} (or \emph{self-similar}) if for
every $g\in A$ and $x\in\alb$ we have $g|_x\in A$.
\end{defi}

If $A$ is a state-closed subset of $A$, then it can be considered
as an \emph{automaton}, which being in a state $g\in A$ and
reading a letter $x\in\alb$ on input, gives on output the letter
$g(x)$ and changes its internal state to $g|_x$. It is easy to see
that if it processes a word $v\in\xs$ in this way, then it will
give on output the word $g(v)$.

\begin{defi}
Let $A\subset G$ be a state-closed subset of a self-similar group
$(G, \alb)$. Then its \emph{Moore diagram} is the oriented graph
with the set of vertices $A$, where for every $x\in\alb$ and $g\in
G$ there is an arrow starting in $g$, ending in $g|_x$, and
labeled by $x$. Every vertex $g$ of the Moore diagram is labeled
by the permutation $x\mapsto g(x)$ of $\alb$.
\end{defi}

\subsection{Contracting self-similar groups and their limit
spaces}

\begin{defi}
A self-similar group $(G, \alb)$ is called \emph{contracting} if
there exists a finite set $\nuke\subset G$ such that for every
$g\in G$ there exists $n_0$ such that for every $v\in\alb^n$ for
$n\ge n_0$ we have $g|_v\in\nuke$. The smallest set $\nuke$
satisfying this condition is called the \emph{nucleus} of the
group $(G, \alb)$.
\end{defi}

Let us fix some contracting self-similar group $(G, \alb)$. Denote
by $\xmo$ the space of sequences $\ldots x_2x_1$, $x_i\in\alb$,
with the direct product topology. By $\xmo\times G$ we denote the
direct product of the space $\xmo$ with the discrete group $G$. We
write the elements of $\xmo\times G$ in the form $\ldots
x_2x_1\cdot g$ for $x_i\in\alb$ and $g\in G$.

\begin{defi}
\label{def:limspace} Two sequences $\ldots x_2x_1, \ldots
y_2y_1\in\xmo$ are said to be \emph{asymptotically equivalent}
(with respect to $(G, \alb)$) if there exists a sequence $g_n\in
G$ taking values in a finite subset of $G$ such that
\[g_n(x_n\ldots x_2x_1)=y_n\ldots y_2y_1,\]
for all $n\ge 1$. Two sequences $\ldots x_2x_1\cdot g, \ldots
y_2y_1\cdot h\in\xmo\times G$ are \emph{asymptotically equivalent}
if there exists a sequence $g_n\in G$ taking values in a finite
subset of $G$ such that
\[g_n\cdot x_n\ldots x_2x_1\cdot g=y_n\ldots y_2y_1\cdot h,\] in
$\bim^{\otimes n}$ for all $n\ge 1$.
\end{defi}

Here $x_n\ldots x_2x_1$ denotes the element
$x_n\otimes\cdots\otimes x_1\otimes x_1$ of $\bim^{\otimes n}$.
Note that the last equality in the definition is equivalent to the
conditions
\[g_n(x_n\ldots x_2x_1)=y_n\ldots y_2y_1,\qquad g_n|_{x_n\ldots
x_2x_1}\cdot g=h,\] for the $n$th iteration of the associated
wreath recursion.

The following description of the asymptotic equivalence relations
is proved in~\cite[Proposition~3.2.6 and Theorem~3.6.3]{nek:book}.

\begin{proposition}
Sequences $\ldots x_2x_1, \ldots y_2y_1\in\xmo$ are asymptotically
equivalent if and only if there exists a sequence $g_n$, $n\ge 0$,
of elements of the nucleus of $G$ such that $g_n\cdot x_n=y_n\cdot
g_{n-1}$ for all $n\ge 1$.

Sequences $\ldots x_2x_1\cdot g, \ldots y_2y_1\cdot h\in\xmo\times
G$ are asymptotically equivalent if and only if there exists a
sequence $g_n$, $n\ge 0$, of elements of the nucleus of $G$ such
that $g_n\cdot x_n=y_n\cdot g_{n-1}$ for all $n\ge 1$, and
$g_0g=h$.
\end{proposition}

\begin{defi}
The quotient of the space $\xmo$ by the asymptotic equivalence
relation is called the \emph{limit space} of the group $(G, \alb)$
and is denoted $\lims$. The quotient of $\xmo\times G$ by the
asymptotic equivalence relation is called the \emph{limit
$G$-space} and is denoted $\limg$.
\end{defi}

The asymptotic equivalence relations on $\xmo$ and $\xmo\times G$
are invariant with respect to the shift $\ldots
x_2x_1\mapsto\ldots x_3x_2$ and the right $G$-action $g:\ldots
x_2x_1\cdot h\mapsto\ldots x_2x_1\cdot (hg)$, respectively. Hence
we get a continuous map $\si:\lims\arr\lims$ induced by the shift,
and a natural right action of $G$ on $\limg$. The space of orbits
$\limg/G$ of the action is naturally homeomorphic to $\lims$.

For every element $x\cdot g$ of the bimodule associated with $(G,
\alb)$ we have a continuous map $\xi\mapsto\xi\otimes x\cdot g$
mapping a point $\xi$ represented by a sequence $\ldots
x_2x_1\cdot h$ to the point represented by
\[\ldots x_2x_1h(x)\cdot h|_xg.\]
Recall that $h\cdot x\cdot g=h(x)\cdot h|_xg$ in the bimodule
$\alb\cdot G$ associated with $(G, \alb)$.

For more on contracting groups and their limit spaces, in
particular for examples, see~\cite[Section~2.11 and Chapter~6]{nek:book}.

\subsection{Approximation of $\limg$ by $G$-spaces}
For more on the subject of this subsection (in particular for
proofs) see~\cite{nek:models}.

Let $(G, \alb)$ be a self-similar contracting group with the
associated wreath recursion $\psi:G\arr\symm\wr G$ and the
permutational bimodule $\bim=\alb\cdot G$.

If $\X$ is a topological space on which $G$ acts from the right
side by homeomorphisms, then we denote by $\X\otimes\bim$ the
quotient of the direct product of the topological spaces
$\X\times\bim$ (where $\bim$ is discrete) by the identifications
\[\xi\cdot g\otimes x=\xi\otimes g\cdot x\]
for $\xi\in\X$, $g\in G$, and $x\in\bim$. It is a right $G$-space
with respect to the action
\[(\xi\otimes x)\cdot g=\xi\otimes (x\cdot g).\]

A map $I:\X\otimes\bim\arr\X$ is said to be \emph{equivariant} if
$I(\xi\otimes x\cdot g)=I(\xi\otimes x)\cdot g$ for all
$\xi\otimes x\in\X\otimes\bim$ and $g\in G$.

For example, consider the limit $G$-space $\limg$. Then there is a
canonical equivariant homeomorphism between $\limg\otimes\bim$ and
$\limg$ induced by the map
\[\xmo\times G\times\bim\arr\xmo\times G:
(\ldots x_2x_1\cdot h, x\cdot g)\mapsto \ldots x_2x_1h(x)\cdot
(h|_xg)\] already mentioned above (see
also~\cite[Section~3.4]{nek:book}).

If $I:\X\otimes\bim\arr\X$ is a $G$-equivariant map, then we
denote by $I^{(n)}$ the map from $\X\otimes\bim^{\otimes n}$ to
$\X$ given by
\[I^{(n)}(\xi\otimes x_1\otimes x_2\otimes\cdots\otimes x_n)=
I(\ldots I(I(\xi\otimes x_1)\otimes x_2)\ldots\otimes x_n),\] for
$x_i\in\bim$, and by $I_n:\X\otimes\bim^{\otimes
(n+1)}\arr\X\otimes\bim^{\otimes n}$ the map given by
\[I_n(\xi\otimes x\otimes v)=I(\xi\otimes x)\otimes v\]
for $v\in\bim^{\otimes n}$ and $x\in\bim$. It is not hard to see
that $I^{(n)}$ and $I_n$ are $G$-equivariant.

\begin{defi}
Suppose that $(G, \alb)$ is a self-similar group, and let $G$ act
on the metric space $(\X, d)$ by isometries properly and
co-compactly. An equivariant map $I:\X\otimes\bim\arr\X$ is
\emph{contracting} if there exists an integer $n\ge 1$ and a
number $0<\lambda<1$ such that
\[d(I^{(n)}(\xi_1\otimes v), I^{(n)}(\xi_2\otimes
v))\le\lambda d(\xi_1, \xi_2),\] for all $\xi_1, \xi_2\in\X$ and
$v\in\bim^{\otimes n}$.
\end{defi}

An action of $G$ on $\X$ is said to be \emph{proper} if for every
compact subset $C\subset\X$ the set of elements $g\in G$ such that
$C\cdot g\cap C\ne\emptyset$ is finite. It is called \emph{co-compact}
if there exists a compact subset $C\subset\X$ such that every $G$-orbit
contains a point in $C$.

If there exists a contracting equivariant map
$I:\X\otimes\bim\arr\X$, then $\X\otimes\bim^{\otimes n}$ are
approximations of the limit $G$-space $\limg$ in the following
sense.

\begin{theorem}
\label{th:approximationlimg} Let $(G, \alb)$ be a contracting
group and let $\bim$ be the associated permutational $G$-bimodule.
Suppose that $\X$ is a locally compact metric space with a
co-compact proper right $G$-action by isometries, and let
$I:\X\otimes\bim\arr\X$ be a contracting equivariant map. Then the
inverse limit of the $G$-spaces and the $G$-equivariant maps
\[\X\xleftarrow{I_1}\X\otimes\bim\xleftarrow{I_2}\X\otimes\bim^{\otimes 2}
\xleftarrow{I_3}\X\otimes\bim^{\otimes 3}\xleftarrow{I_4}\cdots,\]
is homeomorphic as a $G$-space to the limit $G$-space $\limg$
(i.e., there exists an equivariant homeomorphism between the
inverse limit and $\limg$).
\end{theorem}

It follows that, in the setting of the previous theorem, the
orbispaces $\M_n=\X\otimes\bim^{\otimes n}/G$ are approximations
of the limit space $\lims$. More precisely we have the following.

\begin{corollary}
\label{cor:approximatinglims} In conditions of
Theorem~\ref{th:approximationlimg} the limit space $\lims$ is
homeomorphic to the inverse limit of the quotients
$\M_n=\X\otimes\bim^{\otimes n}/G$ with respect to the maps
$\iota_n:\M_{n+1}\arr\M_n$ induced by $I_n$.

The shift map $\si:\lims\arr\lims$ is the limit of the maps
$p_n:\M_{n+1}\arr\M_n$ induced by the correspondence $\xi\otimes
x_1\otimes\cdots\otimes x_n\mapsto\xi\otimes
x_1\otimes\cdots\otimes x_{n-1}$.
\end{corollary}

\subsection{Polyhedral models of the limit space}
\label{ss:polyhedralmodels}

There is a standard procedure of constructing a $G$-space
satisfying conditions of Theorem~\ref{th:approximationlimg}.

Let $(G, \alb)$ be a contracting finitely generated group. Suppose
that it is also \emph{self-replicating} (or \emph{recurrent} in
terms of~\cite{nek:book}), i.e., that it is transitive on the
first level of the tree $\xs$ and the associated virtual
endomorphism is onto. Let $\nuke$ be the nucleus of $(G, \alb)$.
It is a generating set of $G$,
by~\cite[Proposition~2.11.3]{nek:book}.

Denote by $\overline\Xi$ Cayley-Rips complex of $G$ with respect
to the generating set $\nuke$, i.e., the simplicial complex with
the set of vertices $G$ in which a subset $A\subset G$ is a
simplex if $A\cdot g^{-1}\subset\nuke$ for all $g\in A$. The
action of $G$ on itself by right translations is simplicial on
$\overline\Xi$, and we get in this way a right proper co-compact
$G$-space $\overline\Xi$.

The map $I:\overline\Xi\otimes\bim\arr\overline\Xi$ given by the
rule
\[I(g\otimes x\cdot h)=g|_xh\]
is a well defined and $G$-equivariant continuous map. Its
iteration $I^{(n)}:\overline\Xi\otimes\bim^{\otimes
n}\arr\overline\Xi$ is defined by
\[I(g\otimes v\cdot h)=g|_vh\]
for $g, h\in G$ and $v\in\alb^n$.

It is proved in~\cite[Theorem~6.6]{nek:models} that there exists
$k$ such that the map $I^{(k)}$ is homotopic (through equivariant
maps) to a contracting map, hence the spaces
$\overline\Xi\otimes\bim^{\otimes n}$ converge to the limit
$G$-space $\limg$, by Theorem~\ref{th:approximationlimg}.

We can replace $\overline\Xi$ by any sub-complex $\Xi$ of
$\overline\Xi$ (or of the barycentric subdivision of
$\overline\Xi$) such that $\Xi\cdot g=\Xi$ for all $g\in G$, and
$I(\Xi\otimes x)\subseteq\Xi$ for all $x\in\alb$. For instance, it
is natural to consider the complex $\bigcap_{n\ge
1}I^{(n)}(\overline\Xi\otimes\bim^{\otimes n})$.

If the map $I$ is contracting, then the complexes
$J_n=\overline\Xi\otimes\bim^{\otimes n}/G$ approximate the limit
space $\lims$.

\subsection{Iterated monodromy groups}

\begin{defi}
A \emph{partial self-covering} is a covering map $f:\M_1\arr\M$,
where $\M$ is a topological space and $\M_1$ is a subset of $\M$.
\end{defi}

More generally, a \emph{topological automaton} is a covering of
orbispaces $f:\M_1\arr\M$ together with a morphism
$\iota:\M_1\arr\M$ (which is an embedding in the case of a partial
self-covering). For details on the definition of coverings and
morphisms of orbispaces, see~\cite{nek:book}
and~\cite{nek:models}.

The iterated monodromy group of a partial self-covering is defined
in the following way.

\begin{defi}
Let $f:\M_1\arr\M$ be a partial self-covering of a path-connected
and locally path connected topological space $\M$. Let $t\in\M$ be
a base-point. Denote by $K_n$ the kernel of the monodromy action
of $\pi_1(\M, t)$ on the fiber $f^{-n}(t)$ of the $n$th iteration
of $f$. Then the iterated monodromy group $\img{f}$ is the
quotient of the group $\pi_1(\M, t)$ by the intersection
$\bigcap_{n\ge 0}K_n$.
\end{defi}

The iterated monodromy group acts naturally by automorphisms on
the rooted tree of inverse images $\bigsqcup_{n\ge 0}f^{-n}(t)$ of
$t$ under the iterations of the partial self-covering $f$. The
action can be computed using the following permutational
$\pi_1(\M, t)$-bimodule.

Define $\bim_f$ as the set of homotopy classes in $\M$ of the
paths starting in $t$ and ending in a preimage $z\in f^{-1}(t)$.
Then the right action of $\pi_1(\M, t)$ on $\bim_f$ is given by
pre-pending the loops:
\[\ell\cdot\gamma=\ell\gamma,\]
for all $\ell\in\bim_f$ and $\gamma\in\pi_1(\M, t)$. We compose
paths as maps: in a product $\ell\gamma$ the path $\gamma$ is
passed before $\ell$. The left action is given by taking lifts of
loops by $f$:
\[\gamma\cdot\ell=f^{-1}(\gamma)_\ell\ell,\]
where $f^{-1}(\gamma)_\ell$ is the lift of $\gamma$ by $f$
starting at the end of $\ell$.

Let $\alb\subset\bim_f$ be a right orbit transversal, i.e., a
collection of paths $\{\ell_z\}_{z\in f^{-1}(t)}$ starting at $t$
and ending in each of the preimages of $t$. The transversal
defines a wreath recursion $\psi_f$ on $\pi_1(\M, t)$, as it is
described above (just after Definition~\ref{def:bimodules}). This
recursion is the main method of encoding the iterated monodromy
group.

It is sufficient, by Propositions~\ref{pr:virtwreath}
and~\ref{prop:formulawreathrec}, to know the virtual endomorphism
associated with the wreath recursion in order to be able to
reconstruct the wreath recursion. In many cases this is a
convenient way to compute the iterated monodromy group. One can
use the following proposition (see~\cite{nek:book}).

\begin{proposition}
\label{pr:fstar} Let $f:\M_1\arr\M$ be a partial self-covering and
suppose that $\M_1$ and $\M$ are path connected and locally path
connected. Then the virtual endomorphism of $\pi_1(\M)$ associated
with the permutational bimodule $\bim_f$ is equal to the
composition $\iota_*\circ f^{-1}_*$, where $f^{-1}_*$ is the
virtual homomorphism $\pi_1(\M)\arr\pi_1(\M_1)$ lifting loops by
$f$, and $\iota:\M_1\arr\M$ is the identical embedding. All
morphisms of the fundamental groups are defined here up to inner
automorphisms.
\end{proposition}

The associated self-similar action of $\pi_1(\M, t)$ on $\xs$ is
conjugate to the action of $\pi_1(\M, t)$ on the tree of preimages
of $t$, hence the iterated monodromy group $\img{f}$ coincides
with the faithful quotient of $\pi_1(\M, t)$ with respect to the
wreath recursion $\psi_f$.

The main application of the iterated monodromy groups is based on
the following theorem, proved in~\cite{nek:book} (which can also
be deduced from Theorem~\ref{th:approximationlimg} above).

\begin{theorem}
\label{th:main} Let $f:\M_1\arr\M$ be a partial self-covering of a
path-connected and locally simply connected orbispace $\M$ with a
complete length metric. Suppose that the fundamental group of $\M$
is finitely generated and $f$ is uniformly expanding on $\M$.

Then the iterated monodromy group $\img{f}$ is contracting and the
restriction of $f$ onto the set of the accumulation points of
$\bigcup_{n\ge 0}f^{-n}(t)$ is topologically conjugate with the
limit dynamical system $\si:\lims[\img{f}]\arr\lims[\img{f}]$.
\end{theorem}

\section{Computation of the iterated monodromy group}
\label{s:imgcomputation} Recall that the post-critical set $V$ of
$f$ is the union of the line at infinity and the lines $z=0$,
$z=1$, $w=0$, $w=1$, and $z=w$. It follows that the complement
$\CP\setminus V$ can be interpreted as the configuration space of
a pair of points $(z, w)$ in $\C$ that are different from $0$,
$1$, and from each other.

The rational map $\left(1-\frac 2w\right)^2$ appearing in the
second coordinate of $f$ has three fixed points: $w=1$, $w=2i$,
and $w=-2i$. The polynomial $f_w$ for $w=2i$ is conjugate to the
polynomial $z^2+i$.

The polynomial $f_{2i}$ has two fixed points $z_1\approx 0.3002 +
0.3752i$ and $z_2\approx -1.3002 + 1.6248i$. Let us take $(z,
w)=(z_1, 2i)$ as a base-point in the space $\CP\setminus V$.

Let $\alpha, \beta, \gamma$ be the loops in the configuration
space $\CP\setminus V$ obtained by moving $z$ around $0, 1$, and
$2i$, respectively; and let $s$ and $t$ be the loops obtained by
moving $w$ around $0$ and $1$, respectively. Then the fundamental
group of $\CP\setminus V$ is generated by the loops $\alpha,
\beta, \gamma, s$, and $t$ (see Figure~\ref{fig:generators}).

\begin{figure}
\centering
  \includegraphics{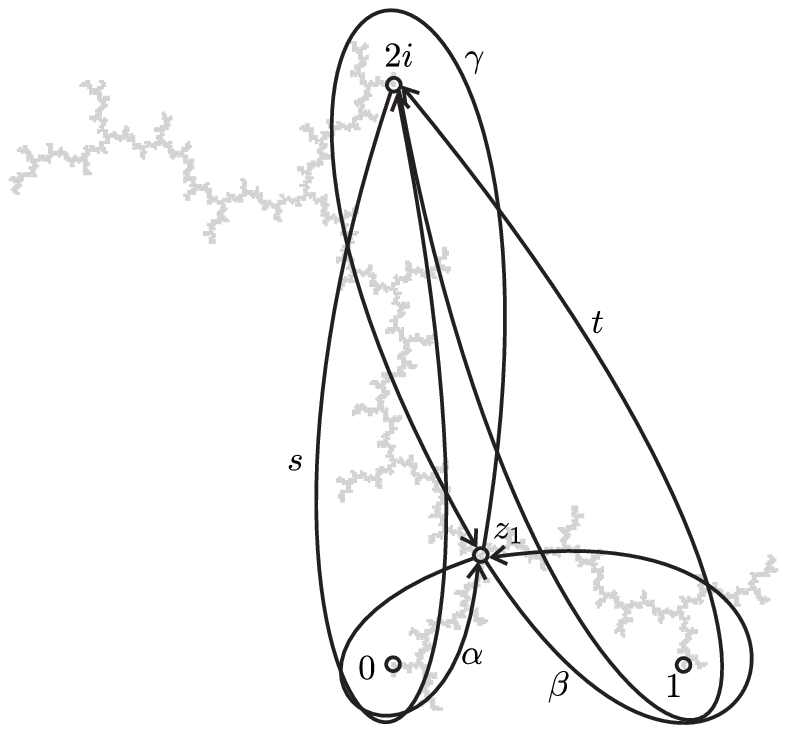}\\
  \caption{Generators of $\pi_1(\CP\setminus V)$}\label{fig:generators}
\end{figure}

We have the following relations between these loops
\begin{align}\label{eq:conj1}
t\alpha t^{-1}= & \alpha,
&\quad s\alpha s^{-1}= & \alpha\gamma\alpha\gamma^{-1}\alpha^{-1},\\
\label{eq:conj2} t\beta t^{-1}= & \gamma\beta\gamma^{-1},&\quad
s\beta s^{-1}= & \beta,\\
\label{eq:conj3} t\gamma t^{-1}= &
\gamma\beta\gamma\beta^{-1}\gamma^{-1},&\quad s\gamma s^{-1}= &
\alpha\gamma\alpha^{-1},
\end{align}
since $s$ and $t$ correspond to the Dehn twists around the curves
shown on the left-hand side part of
Figure~\ref{fig:imgcomputation}.

\begin{figure}
\centering\includegraphics{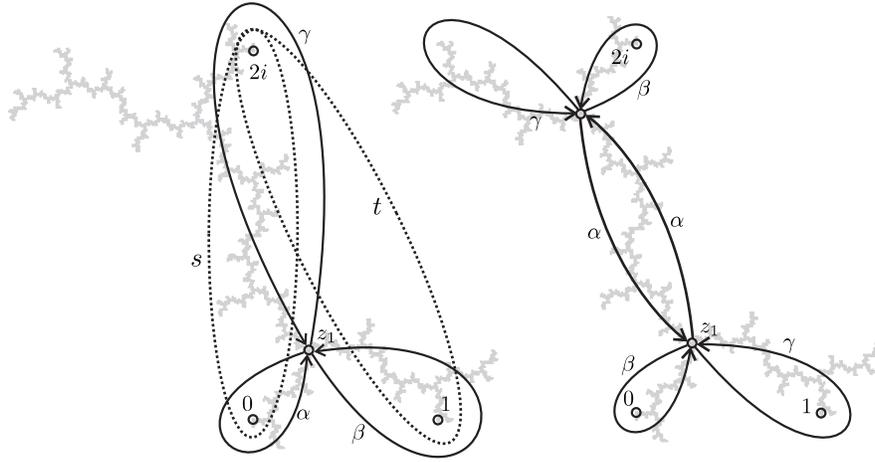}\\
  \caption{Computation of $\img{f}$}\label{fig:imgcomputation}
\end{figure}

\begin{lemma}
\label{l:normalsbgrp} The subgroup of $\pi_1(\CP\setminus V)$
generated by $\alpha, \beta$, and $\gamma$ is normal and has
trivial centralizer.
\end{lemma}

\begin{proof}
The group $G=\langle\alpha, \beta, \gamma\rangle$ is the
fundamental group of the configuration space of one point $z$ in
$\C\setminus\{0, 1, w\}$, where $w\in\C$ is an arbitrary point
different from $0$ and $1$. By~\cite[Theorem~1.4]{birman:book},
the subgroup $G<\pi_1(\CP\setminus V)$ is normal with the quotient
isomorphic to the configuration space of one point $p$ in
$\C\setminus\{0, 1\}$. It follows that $G$ is the fundamental
group of a three-punctured plane, and the quotient
$\pi_1(\CP\setminus V)/G$ is the fundamental group of a
two-punctured plane. Consequently $G$ and $\pi_1(\CP\setminus
V)/G$ are free of rank 3 and 2, respectively. Hence the group
$\langle s, t\rangle$ is \emph{a fortiori} free.

It is known (see, for
instance~\cite[Corollary~1.8.3]{birman:book}) that the braid group
$B_n$ acts faithfully on the free group $F_n$ by automorphisms in
the natural way. In particular, the action of $\langle s,
t\rangle$ on $G$ by conjugation is faithful. It follows that if
$g\in\pi_1(\CP\setminus V)$ acts trivially by conjugation on $G$,
then it belongs to $G$. But $G$ is free, hence only the trivial
element of $\pi_1(\CP\setminus V)$ centralizes $G$.
\end{proof}

In other words, if we know that two elements $g_1$ and $g_2$ of
$\pi_1(\CP\setminus V)$ define the same automorphism on
$\langle\alpha, \beta, \gamma\rangle$ by conjugation, then we know
that $g_1=g_2$. We will use this fact to identify the elements of
the fundamental group by their action on the free group
$\langle\alpha, \beta, \gamma\rangle$.

\begin{proposition}
\label{pr:imgvirtend} The values of the virtual endomorphism of
$\pi_1(\CP\setminus V)$ associated with the partial self-covering
$f$ on the generators of its domain are
\begin{gather*}
\phi(\alpha^2)=\unit,\quad\phi(\beta)=\alpha,\quad\phi(\gamma)=\beta,
\quad\phi(\alpha\beta\alpha^{-1})=\gamma,\quad\phi(\alpha\gamma\alpha^{-1})=\unit\\
\phi(s^2)=\unit,\quad\phi(t)=\beta\alpha\beta^{-1}\gamma\beta
t^{-1}s^{-1},\quad\phi(sts^{-1})=t.
\end{gather*}
\end{proposition}

Recall that $\unit$ denotes the identity element of the group.

\begin{proof}
The domain of $\phi$ is the subgroup of loops such that their
$f$-preimages starting in $(z_1, 2i)$ are again loops. It is a
subgroup of index 4, since the covering is $4$-fold.

The right-hand side of Figure~\ref{fig:imgcomputation} shows the
preimages of the loops $\alpha, \beta, \gamma$ under the action of
the polynomial $f_{2i}$ (the labels show the images of the
corresponding paths under the action of $f_{2i}$). We see that
$\alpha^2$, $\beta$, $\gamma$, $\alpha\beta\alpha^{-1}$ and
$\alpha\gamma\alpha^{-1}$ belong to the domain of $\phi$. It is
also clear that $s^2$, $t$ and $sts^{-1}$ belong to the domain of
$\phi$. These elements already generate a subgroup of index 4 in
the fundamental group of $\CP\setminus V$, due to
relations~\eqref{eq:conj1}--\eqref{eq:conj3} between $\alpha,
\beta, \gamma$, and $s, t$. Consequently, these elements generate
the domain of $\phi$.

We see from Figure~\ref{fig:imgcomputation} that
\[\phi(\alpha^2)=\unit,\quad\phi(\beta)=\alpha,\quad\phi(\gamma)=\beta,
\quad\phi(\alpha\beta\alpha^{-1})=\gamma,\quad\phi(\alpha\gamma\alpha^{-1})=\unit.\]

It is more convenient to find the action of $\phi$ on rest of the
generators of the domain using the action of $s$ and $t$ on
$\langle\alpha, \beta, \gamma\rangle$ (see
Lemma~\ref{l:normalsbgrp}).

We have
\[\phi(s^2)\alpha\phi(s^2)^{-1}=\phi(s^2\beta
s^{-2})=\phi(\beta)=\alpha,\]
\begin{multline*}\phi(s^2)\beta\phi(s^{-2})=\phi(s^2\gamma
s^{-2})=\phi(\alpha\gamma\alpha\gamma\alpha^{-1}\gamma^{-1}\alpha^{-1})=\\
\phi(\alpha\gamma\alpha^{-1}\cdot\alpha^2\cdot\gamma\cdot\alpha^{-2}\cdot
\alpha\gamma^{-1}\alpha^{-1})=\beta,\end{multline*} and
\begin{multline*}\phi(s^2)\gamma\phi(s^{-2})=\phi(s^2\alpha\beta\alpha^{-1}s^2)=\\
\phi((\alpha\gamma)^2\alpha(\alpha\gamma)^{-2}\beta(\alpha\gamma)^2\alpha^{-1}
(\alpha\gamma)^{-2})=\\
\phi(\alpha\gamma\alpha^{-1}\cdot\alpha^2\cdot\gamma\cdot\alpha\gamma^{-1}\alpha^{-1}\cdot
\gamma^{-1}\cdot\alpha^{-1}\beta\alpha\cdot\gamma\cdot\alpha\gamma\alpha^{-1}\cdot\gamma^{-1}
\cdot\alpha^{-2}\cdot\alpha\gamma^{-1}\alpha^{-1})=\\
\beta\beta^{-1}\gamma\beta\beta^{-1}=\gamma,
\end{multline*}
which implies that
\[\phi(s^2)=\unit.\]

We have
\[\phi(t)\alpha\phi(t)^{-1}=\phi(t\beta
t^{-1})=\phi(\gamma\beta\gamma^{-1})=\beta\alpha\beta^{-1},\]
\[\phi(t)\beta\phi(t)^{-1}=\phi(t\gamma
t^{-1})=\phi(\gamma\beta\gamma\beta^{-1}\gamma^{-1})=\beta\alpha\beta\alpha^{-1}\beta^{-1},\]
and
\[\phi(t)\gamma\phi(t)^{-1}=\phi(t\alpha\beta\alpha^{-1}t^{-1})=
\phi(\alpha\gamma\beta\gamma^{-1}\alpha^{-1})=\gamma.\] It follows
that
\[\phi(t)=r=\beta\alpha\beta^{-1}\gamma\beta t^{-1}s^{-1},\]
since direct computations show that
\[r\alpha r^{-1}=\beta\alpha\beta^{-1},\quad r\beta
r^{-1}=\beta\alpha\beta\alpha^{-1}\beta^{-1},\quad r\gamma
r^{-1}=\gamma.\]

It remains to compute $\phi(sts^{-1})$. We have
\begin{multline*}\phi(sts^{-1})\alpha\phi(st^{-1}s^{-1})=\phi(sts^{-1}\beta
st^{-1}s^{-1})=\phi(s\gamma\beta\gamma^{-1}s^{-1})=\\
\phi(\alpha\gamma\alpha^{-1}\beta\alpha\gamma^{-1}\alpha^{-1})=\alpha,
\end{multline*}
\begin{multline*}
\phi(sts^{-1})\beta\phi(st^{-1}s^{-1})=\phi(sts^{-1}\gamma
st^{-1}s^{-1})=\\ \phi(st\gamma^{-1}\alpha^{-1}\gamma\alpha\gamma
t^{-1}s^{-1})=\\
\phi(s\gamma\beta\gamma^{-1}\beta^{-1}\gamma^{-1}\alpha^{-1}
\gamma\beta\gamma\beta^{-1}\gamma^{-1}
\alpha\gamma\beta\gamma\beta^{-1}\gamma^{-1}s)=\\
\phi(\alpha\gamma\alpha^{-1}\beta\alpha\gamma^{-1}\alpha^{-1}
\beta^{-1}\alpha^{-1}\beta\alpha\gamma\alpha^{-1}\beta^{-1}\alpha\beta\alpha\gamma\alpha^{-1}
\beta^{-1}\alpha\gamma^{-1}\alpha^{-1})=\gamma\beta\gamma^{-1},
\end{multline*}
and
\begin{multline*}
\phi(sts^{-1})\gamma\phi(sts^{-1})=\phi(sts^{-1}\alpha\beta\alpha^{-1}st^{-1}s^{-1})=\\
\phi(\alpha\gamma\alpha^{-1}\beta\alpha\gamma^{-1}\alpha^{-1}\beta^{-1}
\alpha\beta\alpha\gamma\alpha^{-1}\beta\alpha\gamma^{-1}\alpha^{-1}
\beta^{-1}\alpha^{-1}\beta\alpha\gamma\alpha^{-1}\beta^{-1}\alpha\gamma^{-1}\alpha^{-1})=\\
\gamma\beta\gamma\beta^{-1}\gamma^{-1},
\end{multline*}
which implies that
\[\phi(sts^{-1})=t,\]
which finishes the proof of the proposition.
\end{proof}

\begin{theorem}
\label{th:imgrecursion} The iterated monodromy group $\img{f}$ is
given by the wreath recursion
\begin{gather*}\alpha=\sigma,\quad\beta=(\alpha, \gamma, \alpha,
\beta^{-1}\gamma\beta),\quad\gamma=(\beta, \unit, \unit, \beta),\\
t=(r, r, t, t),\quad s=\pi(\unit, \beta^{-1}, \unit,
\beta),\end{gather*} where $\sigma=(\one\two)(\three\four)$,
$\pi=(\one\four)(\two\three)$, and
\[r=\beta\alpha\beta^{-1}\gamma\beta t^{-1}s^{-1}.\]
\end{theorem}

\begin{proof}
We will use Proposition~\ref{prop:formulawreathrec} to find the
wreath recursion with the associated virtual endomorphism given in
Proposition~\ref{pr:imgvirtend}.

The set $\{\unit, \alpha, s, \alpha s\}$ is a left coset
transversal of $\pi_1(\CP\setminus V)$ modulo the domain of the
virtual endomorphism $\phi$. Let us take the transversal in the
given order. Using Propositions~\ref{prop:formulawreathrec}
and~\ref{pr:imgvirtend} and
relations~\eqref{eq:conj1}--\eqref{eq:conj3}, we get:
\begin{multline*}\alpha=\sigma(\phi(\alpha^{-1}\alpha),
\phi(\alpha^2), \phi((\alpha s)^{-1}(\alpha s)),
\phi(s^{-1}\alpha^2 s))=\\ \sigma(\phi(\unit), \phi(\alpha^2),
\phi(\unit), \phi(\gamma^{-1}\alpha^2\gamma))=\sigma,
\end{multline*}
\begin{multline*}\beta=(\phi(\beta), \phi(\alpha^{-1}\beta\alpha),
\phi(s^{-1}\beta s), \phi(s^{-1}\alpha^{-1}\beta\alpha s))=\\
(\phi(\beta), \phi(\alpha^{-1}\beta\alpha), \phi(\beta),
\phi(\gamma^{-1}\alpha^{-1}\gamma\beta\gamma^{-1}\alpha\gamma))=
(\alpha, \gamma, \alpha, \beta^{-1}\gamma\beta),
\end{multline*}
\begin{multline*}
\gamma=(\phi(\gamma), \phi(\alpha^{-1}\gamma\alpha),
\phi(s^{-1}\gamma s), \phi(s^{-1}\alpha^{-1}\gamma\alpha s))=\\
(\phi(\gamma), \phi(\alpha^{-1}\gamma\alpha),
\phi(\gamma^{-1}\alpha^{-1}\gamma\alpha\gamma),
\phi(\gamma^{-1}\alpha^{-2}\gamma\alpha^2\gamma))=(\beta, \unit,
\unit, \beta),
\end{multline*}
\begin{multline*}
t=(\phi(t), \phi(\alpha^{-1}t\alpha), \phi(s^{-1}ts),
\phi(s^{-1}\alpha^{-1}t\alpha s))=\\
(\phi(t), \phi(t), \phi(s^{-1}ts), \phi(s^{-1}ts))= (r, r, t, t),
\end{multline*}
\begin{multline*}
s=\pi(\phi(s^{-1}s), \phi(s^{-1}\alpha^{-1}s\alpha), \phi(s^2),
\phi(\alpha^{-1}s\alpha s))=\\
\pi(\phi(1), \phi(\gamma^{-1}\alpha^{-1}\gamma\alpha), \phi(s^2),
\phi(\gamma\alpha\gamma^{-1}\alpha^{-1}s^2))=\pi(\unit,
\beta^{-1}, \unit, \beta).
\end{multline*}
\end{proof}

\section{Polyhedral model of $f$}
\label{s:polyhedralmodel}
\subsection{Index two extension}
\label{ss:twoextension} The function $f$ has real coefficients,
hence complex conjugation of both coordinates is an automorphism
of the dynamical system $(f, \mathbb{PC}^2)$. We can take the
quotient of this dynamical system by this automorphism (i.e., by
the group of order two generated by it). The iterated monodromy
group of the quotient is, by general theory
(see~\cite[Theorem~3.7.1]{nek:book}
and~\cite[Subsection~3.8]{nek:filling}), an index two extension of
$\img{f}$.

This extension was considered in~\cite{nek:ssfamilies} and was
used to study the properties of the Cantor set of groups
associated with the iterations of the polynomials $f_{p_n}$. It is
the group $\G$ generated by the transformations
\begin{alignat}{2}\label{eq:extension1}
\alpha &= \sigma, &\qquad  a &=\pi,\\
\label{eq:extension2}\beta  &=\left(\alpha, \gamma, \alpha,
\gamma\right), &\qquad  b
&=\left(a\alpha, a\alpha, c, c\right),\\
\label{eq:extension3}\gamma &=\left(\beta, \unit, \unit,
\beta\right), &\qquad  c &=\left(b\beta, b\beta, b, b\right),
\end{alignat}
where $\sigma=(\one\two)(\three\four)$ and
$\pi=(\one\three)(\two\four)$, as before. We will not need the
fact that $\G$ is really the iterated monodromy group of the
quotient of $f$ by complex conjugation, so we will not present its
proof here.

\begin{figure}
\centering
  \includegraphics{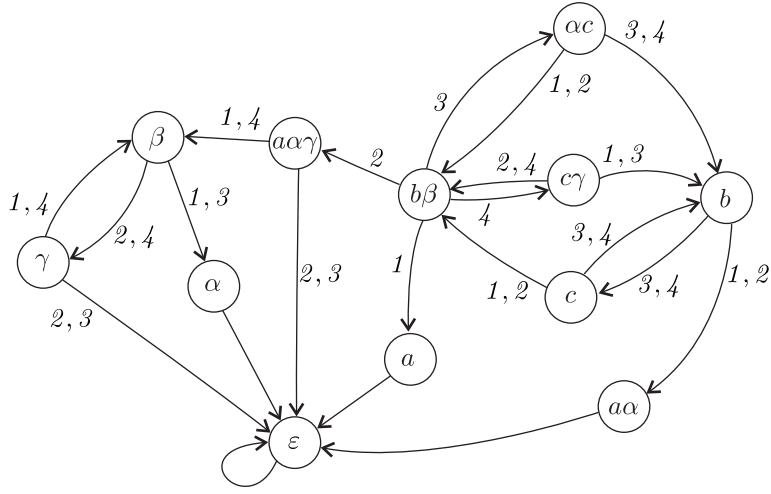}\\
  \caption{Automaton $\mathfrak{A}$ generating $\G$}\label{fig:automaton}
\end{figure}

Note that:
\begin{gather}
\label{eq:extension4}
b\beta=(a, a\alpha\gamma, \alpha c,
c\gamma),\quad a\alpha=\pi\sigma,\quad c\gamma=(b, b\beta, b,
b\beta),\\
\label{eq:extension5} a\alpha\gamma=\pi\sigma(\beta, \unit, \unit,
\beta),\quad \alpha c=\sigma(b\beta, b\beta, b, b).
\end{gather}
We see that the set $\{\alpha, \beta, \gamma, a, b, c, a\alpha,
b\beta, c\gamma, a\alpha\gamma, \alpha c, \unit\}$ is
state-closed, hence it is the set of states of an automaton
$\mathfrak{A}$ generating the group $\G$. The Moore diagram of the
automaton $\mathfrak{A}$ is shown on Figure~\ref{fig:automaton}.
The labels on the arrows show the input letters. An arrow with two
labels $i, j$ correspond to two arrows with labels $i$ and $j$.
Arrows without labels correspond to four arrows with labels $\one,
\two, \three$, and $\four$. The action of the states on the first
level is not shown on the figure (but it follows from their
labels).

Direct computation shows that all generators are involutions and
that the following relations hold
\begin{alignat*}{3}
\alpha^a &=\alpha, &\qquad\alpha^b &=\alpha, &\qquad
\alpha^c &=\alpha,\\
\beta^a  &=\beta, &\qquad\beta^b &=\beta, &\qquad\beta^c &=
\beta^\gamma,\\
\gamma^a &=\gamma^\alpha, &\qquad\gamma^b &=\gamma^\beta, &\qquad
\gamma^c &=\gamma.
\end{alignat*}

Let us show that the iterated monodromy group of $f$ is a
self-similar subgroup of index 2 in $\G$. It follows from the
relations mentioned above that the group generated by $\alpha,
\beta, \gamma, ac, bc$ is a subgroup of index two. Let us show
that it is also self-similar (i.e., becomes state-closed after
composition of the wreath recursion with an inner automorphism of
the wreath product $\symm\wr\G$). We have
\[ac=\pi(b\beta, b\beta, b, b),\quad cb=(b\beta a\alpha,
b\beta a\alpha, bc, bc).\] Let us conjugate the right-hand side by
$(\unit, \unit, b, b)$ (i.e., let us change the basis of the
permutational bimodule from $\{\one, \two, \three, \four\}$ to
$\{\one, \two, \three\cdot b, \four\cdot b\}$). We get
\begin{eqnarray*}
\alpha &=& \sigma\\
\beta  &=& \left(\alpha, \gamma, \alpha, \gamma^b\right)=
\left(\alpha, \gamma, \alpha, \gamma^\beta\right)\\
\gamma &=& \left(\beta, \unit, \unit, \beta^b\right)=\left(\beta,
\unit, \unit, \beta\right),
\end{eqnarray*}
and
\begin{eqnarray*}
ac &=&  \pi(\beta, \beta, \unit, \unit)\\
cb &=& (\beta\alpha ba, \beta\alpha ba, cb, cb).
\end{eqnarray*}

If we denote $s=ac\gamma$ and $t=cb$, then we have
\[s=ac\gamma=\pi(\unit, \beta, \unit, \beta)\]
and $\beta\alpha ba=\beta\alpha bc\gamma \gamma
ca=\beta\alpha\beta\gamma\beta bc\gamma
ca=\beta\alpha\beta\gamma\beta t^{-1}s^{-1}$, so that
\[t=(r, r, t, t),\]
where $r=\beta\alpha ba=\beta\alpha\beta\gamma\beta t^{-1}s^{-1}$,
as in Theorem~\ref{th:imgrecursion}.

We have just proved that the index two subgroup $\langle\alpha,
\beta, \gamma, ac, cb\rangle$ is isomorphic as a self-similar
group to $\img{f}$.

\subsection{Some finite subgroups of $\G$}

Direct computations (see also~\cite{nek:ssfamilies}) show that the
following relations hold in the group $\G=\langle\alpha, \beta,
\gamma, a, b, c\rangle$.
\[(\alpha\gamma)^4=\unit,\quad (\alpha\beta)^8=\unit,\quad
(\beta\gamma)^8=\unit\] and \begin{gather*}(ac)^2=(\beta, \beta,
\beta, \beta)=(\alpha\gamma)^2\\
(ab)^4=(ca\alpha, ca\alpha, a\alpha c, a\alpha
c)^2=((\alpha\gamma)^2, (\alpha\gamma)^2, (\alpha\gamma)^2,
(\alpha\gamma)^2)=
(\alpha\beta)^4\\
(bc)^4=((ab)^4(\alpha\beta)^4, (ab)^4(\alpha\beta)^4, (cb)^4,
(cb)^4)=((\alpha\beta)^8, (\alpha\beta)^8, (cb)^4, (cb)^4)=\unit.
\end{gather*}
Consequently, the products $ab$, $ac$, and $bc$ are of orders 8,
4, and 4, respectively. Note that the elements $a\alpha, b\beta,
c\gamma$ are of order 2 (since $[a, \alpha]=[b, \beta]=[c,
\gamma]=\unit$) and that we have
\[(ac\gamma)^2=ac\gamma ac\gamma=aca\alpha\gamma\alpha c\gamma=
acac\alpha\gamma\alpha\gamma=(\alpha\gamma)^4=\unit,\]
\[(a\alpha b\beta)^2=(ab)^2(\alpha\beta)^2,\quad
(a\alpha b\beta)^4=(ab)^4(\alpha\beta)^4= (\alpha\beta)^8=\unit\]
hence $a\cdot c\gamma$ is of order 2, while $a\alpha\cdot b\beta$
is of order $4$.

It follows that the group $\G$ contains the following finite
groups
\begin{eqnarray*}\G_{A_1}= &\langle\alpha, b,
c\rangle=\langle\alpha\rangle\times\langle b, c\rangle& \cong
C_2\times D_4,\\
\G_{B_1}= &\langle\beta, a, c\gamma\rangle
=\langle\beta\rangle\times\langle
a, c\gamma\rangle& \cong C_2\times D_2,\\
\G_{C_1}= &\langle\gamma, a\alpha, b\beta\rangle
=\langle\gamma\rangle\times\langle a\alpha, b\beta\rangle& \cong
C_2\times D_4,
\end{eqnarray*}
and
\begin{eqnarray*}
\G_A= & \langle\beta, \gamma, b, c\rangle=\langle\beta,
\gamma\rangle\rtimes\langle b, c\rangle & \cong D_8\rtimes D_4,\\
\G_B= & \langle\alpha, \gamma, a, c\rangle=\langle\alpha,
\gamma\rangle\rtimes\langle a, c\gamma\rangle & \cong D_4\rtimes
D_2,\\
\G_C= & \langle\alpha, \beta, a, b\rangle=\langle\alpha,
\beta\rangle\rtimes\langle a\alpha, b\beta\rangle & \cong
D_8\rtimes D_4,
\end{eqnarray*}
where in the last three cases in a semidirect product $\langle
x_1, y_1\rangle\rtimes\langle x_2, y_2\rangle$ the generators
$x_2$ and $y_2$ act on $\langle x_1, y_1\rangle$ as conjugation by
$x_1$ and $y_1$, respectively. (The group of inner automorphisms
of $D_{2n}$ is isomorphic to $D_n$.)

\subsection{Nucleus of $\G$}
\label{ss:nucleusofG} It was proved in~\cite{nek:ssfamilies} that
the group $\G$ is contracting without presenting the nucleus
explicitly. More careful analysis gives the following complete
description of the nucleus of $\G$.

\begin{proposition}
\label{pr:nucleusGamma} The set
$\nuke=\G_A\cup\G_B\cup\G_C\cup\G_{A_1}\cup\G_{B_1}\cup \G_{C_1}$
is the nucleus of the group $\G$.
\end{proposition}

\begin{proof} It is checked directly that the set $\nuke$ is
state-closed. It is sufficient to show that sections of
$\nuke\cdot\{\alpha, \beta, \gamma, a, b, c\}$ eventually belong
to $\nuke$. Since sections of $\alpha$ and $a$ are trivial, it is
sufficient to consider the set $\nuke\cdot\{\beta, \gamma, b,
c\}$.

The generators of $\G_A$ are $\beta, \gamma, b, c$, hence
$\G_A\cdot\{\beta, \gamma, b, c\}=\G_A$. The group $\G_B$ is
generated by $\alpha, \gamma=(\beta, \unit, \unit, \beta), a,
c=(b\beta, b\beta, b, b)$. It follows that the first level
sections of the elements of $\G_B$ belong to $\langle b,
\beta\rangle$. Consequently, the first level sections of
$\G_B\cdot\{\beta, \gamma, b, c\}=\G_B\cdot\{\beta, b\}$ belong
either to $\langle b, \beta, \alpha\rangle<\G_C$, or to $\langle
b, \beta, \gamma\rangle<\G_A$, or to $\langle b, \beta,
a\alpha\rangle<\G_C$, or to $\langle b, \beta, c\rangle<\G_A$.

Note that $b\beta=(a, a\alpha\gamma, c\alpha, c\gamma)$, therefore
the sections of $\G_B\cdot b\beta$ belong to \[\langle b, \beta,
a\rangle\cup\langle b, \beta, c\gamma\rangle\cup\langle b, \beta,
c\alpha\rangle\cup\langle b, \beta \gamma
a\alpha\rangle\subset\G_C\cup\G_A\cup\G_A\cdot\alpha\cup\G_A\cdot
a\alpha.\] Consequently, the second level sections of $\G_B\cdot
b\beta$ belong to $\nuke$.

The group $\G_C$ is generated by $\alpha, \beta=(\alpha, \gamma,
\alpha, \gamma), a, b=(a\alpha, a\alpha, c, c)$. The first level
sections of the elements of $\G_C$ belong to $\langle a, \alpha,
\gamma, c\rangle=\G_B$. The first level sections of
$\G_C\cdot\{\beta, \gamma, b, c\}=\G_C\cdot\{\gamma, c\}$ belong
to $\G_B\cdot\{\beta, b, b\beta\}$, hence the third level sections
of $\G_C$ belong to $\nuke$.

The group $\G_{A_1}$ is generated by $\alpha, b=(a\alpha, a\alpha,
c, c), c=(b\beta, b\beta, b, b)$, hence the first level sections
of $\G_{A_1}\cdot\{\beta, \gamma, b, c\}=\G_{A_1}\cdot\{\beta,
\gamma\}$ belong to
\[(\langle a\alpha, b\beta\rangle\cup\langle
b, c\rangle)\cdot\{\alpha, \beta, \gamma\}\subset \langle\alpha,
\beta, a, b\rangle\cup\langle\gamma, a\alpha,
b\beta\rangle\cup\langle\alpha, b, c\rangle\cup\langle\beta,
\gamma, b, c\rangle\subset\nuke.\]

The group $\G_{B_1}$ is generated by $\beta=(\alpha, \gamma,
\alpha, \gamma), a, c\gamma=(b, b\beta, b, b\beta)$. The first
level sections of $\G_{B_1}$ belong to $\langle\alpha,
b\rangle\cup\langle\gamma, b\beta\rangle$. Hence, the first level
sections of $\G_{B_1}\cdot\{\beta, \gamma, b,
c\}=\G_{B_1}\cdot\{\gamma, b, c\}$ belong to
\[(\langle\alpha,
b\rangle\cup\langle\gamma, b\beta\rangle)\cdot\{\beta, a\alpha,
b\beta, b, c\}\subset\langle\alpha, \beta, a,
b\rangle\cup\langle\alpha, b, c\rangle\cup\langle\beta, \gamma, b,
c\rangle\cup\langle\gamma, a\alpha, b\beta\rangle,\] which is a
subset of $\nuke$.

The group $\G_{C_1}$ is generated by $\gamma=(\beta, \unit, \unit,
\beta), a\alpha, b\beta=(a, a\alpha\gamma, c\alpha, c\gamma)$,
which implies that the first level sections of $\G_{C_1}$ belong
to $\langle\beta, a, c\gamma\rangle\cup\langle a\alpha\gamma,
c\alpha\rangle$. Therefore, the first level sections of the
elements of $\G_{C_1}\cdot\{\beta, \gamma, b,
c\}=\G_{C_1}\cdot\{\beta, b, c\}$ belong to
\[(\langle\beta, a, c\gamma\rangle\cup\langle a\alpha\gamma,
c\alpha\rangle)\cdot\{\alpha, \gamma, a\alpha, b\beta, b\}.\]
Since the sections of $a$ and $\alpha$ are trivial, the sections
of the elements of $\langle\beta, a, c\gamma\rangle\cdot\{\alpha,
\gamma, a\alpha, b\beta, b\}$ in non-empty words are the same as
sections of \[\langle\beta, a, c\gamma\rangle\cdot\{\gamma,
b\beta, b\}=\langle\beta, a, c\gamma\rangle\cdot\{\gamma,
b\}=\G_{B_1}\cdot\{\gamma, b\},\] but we have seen that sections
of the elements of this set in one-letter words belong to $\nuke$.

It remains to consider the set $\langle a\alpha\gamma,
c\alpha\rangle\cdot\{\alpha, \gamma, a\alpha, b\beta, b\}$. We
have
\[\langle a\alpha\gamma,
c\alpha\rangle\cdot\{\alpha, \gamma, a\alpha, b\beta,
b\}\subset\langle\alpha, \gamma, a, c\rangle\cdot\{\unit, b\beta,
b\}=\G_B\cdot\{\unit, b\beta, \beta\},\] but we have seen above
that the sections of the elements of the set $\G_B\cdot\{\unit,
b\beta, \beta\}$ in words of length two belong to $\nuke$.
\end{proof}

The sizes of the defined subgroups of $\nuke$ are $|\G_A|=128$,
$|\G_B|=32$, $|\G_C|=128$, $|\G_{A_1}|=16$, $|\G_{B_1}|=8$,
$|\G_{C_1}|=16$. Their pairwise intersections are (here we denote
by $\G_{XY}$ the subgroup $\G_X\cap\G_Y$):
\begin{alignat*}{3}
\G_{AB}&=\langle\gamma, c\rangle,&\quad \G_{AC}&=\langle\beta,
b\rangle,&\quad
\G_{BC}&=\langle\alpha, a\rangle,\\
\G_{AA_1}&=\langle b, c\rangle,&\quad \G_{AB_1}&=\langle \beta,
c\gamma\rangle,&\quad
\G_{AC_1}&=\langle\gamma, b\beta\rangle,\\
\G_{BA_1}&=\langle\alpha, c\rangle,&\quad \G_{BB_1}&=\langle a,
c\gamma\rangle,&\quad
\G_{BC_1}&=\langle\gamma, a\alpha\rangle,\\
\G_{CA_1}&=\langle\alpha, b\rangle,&\quad \G_{CB_1}&=\langle\beta,
a\rangle,&\quad \G_{CC_1}&=\langle a\alpha, b\beta\rangle.
\end{alignat*}
The intersections $\G_{A_1}\cap\G_{B_1}, \G_{A_1}\cap\G_{C_1},
\G_{B_1}\cap\G_{C_1}$ are trivial. All non-trivial pairwise
intersections are isomorphic to $C_2\times C_2$, except for
\[\G_{AA_1}\cong D_4,\quad\G_{BB_1}\cong D_2,\quad\G_{CC_1}\cong D_4.\]

The only non-trivial triple intersections (all of order 2) are
\[\G_{A_1BC}=\langle\alpha\rangle,\quad
\G_{AB_1C}=\langle\beta\rangle,\quad
\G_{ABC_1}=\langle\gamma\rangle\] and
\begin{alignat*}{3}
\G_{A_1AB}&=\langle c\rangle,&\quad\G_{A_1AC}&=\langle
b\rangle,&\quad\G_{B_1AB}&=\langle
c\gamma\rangle,\\
\G_{B_1BC}&=\langle a\rangle,&\quad\G_{C_1BC}&=\langle
a\alpha\rangle,&\quad \G_{C_1AC}&=\langle
b\beta\rangle.\end{alignat*}

Since all triple intersections are of order 2 and are pairwise
different, there are no non-trivial intersections of four or more
different groups $\G_X$.

Removing the identity and using the inclusion-exclusion formula,
we get that the size of the nucleus of $\G$ is
\[1+127+31+127+15+7+15-7-3-7-9\cdot 3+9=288.\]

\subsection{Sections of subgroups of
$\nuke$}\label{ss:sectionsofGs} Let us list sections $G|_x$ for
the groups $\G_*$.

\[\begin{array}{|c|c||c|c|c|c|}
\hline & G & G|_\one & G|_\two & G|_\three & G|_\four \\
\hline \G_A & \langle\beta, \gamma, b, c\rangle  & \langle\alpha, \beta, a, b\rangle & \langle\gamma, a\alpha, b\beta\rangle & \langle\alpha, b, c\rangle & \langle\beta, \gamma, b, c\rangle\\
\hline \G_B & \langle\alpha, \gamma, a, c\rangle & \langle\beta, b\rangle & \langle\beta, b\rangle & \langle\beta, b\rangle & \langle\beta, b\rangle \\
\hline \G_C & \langle\alpha, \beta, a, b\rangle & \langle\alpha, \gamma, a, c\rangle & \langle\alpha, \gamma, a, c\rangle & \langle\alpha, \gamma, a, c\rangle & \langle\alpha, \gamma, a, c\rangle\\
\hline \G_{A_1} & \langle\alpha, b, c\rangle & \langle a\alpha, b\beta\rangle & \langle a\alpha, b\beta\rangle & \langle b, c\rangle & \langle b, c\rangle\\
\hline \G_{B_1} & \langle\beta, a, c\gamma\rangle & \langle\alpha, b\rangle & \langle\gamma, b\beta\rangle & \langle\alpha, b\rangle & \langle\gamma, b\beta\rangle \\
\hline \G_{C_1} & \langle\gamma, a\alpha, b\beta\rangle & \langle\beta, a, c\gamma\rangle & \langle a\alpha\gamma, \alpha c\rangle & \langle a\alpha\gamma, \alpha c\rangle & \langle\beta, a, c\gamma\rangle \\
\hline
\end{array}\]

\[\begin{array}{|c|c||c|c|c|c|}
\hline & G & G|_\one & G|_\two & G|_\three & G|_\four \\
\hline \G_{AB} & \langle\gamma, c\rangle & \langle\beta, b\rangle & \langle b\beta\rangle & \langle b\rangle & \langle\beta, b\rangle \\
\hline \G_{AC} & \langle\beta, b\rangle & \langle\alpha, a\rangle & \langle\gamma, a\alpha\rangle & \langle\alpha, c\rangle & \langle\gamma, c\rangle \\
\hline \G_{BC} & \langle\alpha, a\rangle & \{\unit\} & \{\unit\} & \{\unit\} & \{\unit\} \\
\hline \G_{AA_1} & \langle b, c\rangle & \langle a\alpha, b\beta\rangle & \langle a\alpha, b\beta\rangle & \langle b, c\rangle & \langle b, c\rangle \\
\hline \G_{AB_1} & \langle\beta, c\gamma\rangle & \langle\alpha, b\rangle & \langle\gamma, b\beta\rangle & \langle\alpha, b\rangle & \langle\gamma, b\beta\rangle \\
\hline \G_{AC_1} & \langle\gamma, b\beta\rangle & \langle\beta, a\rangle & \langle a\alpha\gamma\rangle & \langle\alpha c\rangle & \langle\beta, c\gamma\rangle \\
\hline \G_{BA_1} & \langle\alpha, c\rangle & \langle b\beta\rangle & \langle b\beta\rangle & \langle b\rangle & \langle b\rangle \\
\hline \G_{BB_1} & \langle a, c\gamma\rangle & \langle b\rangle & \langle b\beta\rangle & \langle b\rangle & \langle b\beta\rangle \\
\hline \G_{BC_1} & \langle\gamma, a\alpha\rangle & \langle\beta\rangle & \{\unit\} & \{\unit\} & \langle\beta\rangle \\
\hline \G_{CA_1} & \langle\alpha, b\rangle & \langle\alpha, \gamma\rangle & \langle\alpha, \gamma\rangle & \langle\alpha, \gamma\rangle & \langle\alpha, \gamma\rangle \\
\hline \G_{CB_1} & \langle\beta, a\rangle & \langle a\alpha, c\rangle & \langle a\alpha, c\rangle & \langle a\alpha, c\rangle & \langle a\alpha, c\rangle \\
\hline \G_{CC_1} & \langle a\alpha, b\beta\rangle & \langle a, c\gamma\rangle & \langle a\alpha\gamma, \alpha c\rangle & \langle a\alpha\gamma, \alpha c\rangle & \langle a, c\gamma\rangle \\
\hline
\end{array}\]

The triple intersections $\G_{XYZ}$ are generated by single
elements $g\in\mathfrak{A}$, therefore their sections are
described on the Moore diagram of the automaton $\mathfrak{A}$
(see Figure~\ref{fig:automaton}).

\subsection{Complex associated with the nucleus of $\G$}
\label{ss:pasting} Denote by $\mathcal{G}$ the set of subgroups
$\G_A, \G_B, \G_C, \G_{A_1}, \G_{B_1}, \G_{C_1}$, and all their
pairwise and triple intersections.

Denote by $\Xi$ the simplicial complex associated with the poset
of the subsets of $\G$ of the form $G\cdot g$ for
$G\in\mathcal{G}$ and $g\in\G$. Then $\Xi$ is a $\G$-invariant
sub-complex of the barycentric subdivision of the Cayley-Rips
complex $\overline\Xi$ of $\G$ defined by the generating set
$\nuke$ (see Subsection~\ref{ss:polyhedralmodels}).

Denote by $\T_0$ the simplcial complex of the poset $\mathcal{G}$.
It is the subcomplex of $\Xi$ spanned by the vertices
corresponding to cosets containing the identity, i.e., it is the
union of the simplicies of $\Xi$ containing the vertex
corresponding to the trivial subgroup $\{\unit\}\in\mathcal{G}$ of
$\G$.

It follows from the description of the set of groups $\mathcal{G}$
that we can represent $\T_0$ as a union of three tetrahedra
$A_1ABC$, $B_1ABC$, $C_1ABC$ with a common face $ABC$, as it is
shown on Figure~\ref{fig:complex}. Every element $\G_X$,
$\G_{XY}$, or $\G_{XYZ}$ of $\mathcal{G}$ corresponds to (the
barycenter of) the corresponding vertex $X$, edge $XY$, or
triangle $XYZ$, respectively. We have labeled on
Figure~\ref{fig:complex} the triangles of $\T_0$ by the generators
of the respective groups $\G_{XYZ}$.

\begin{figure}
\centering
  \includegraphics{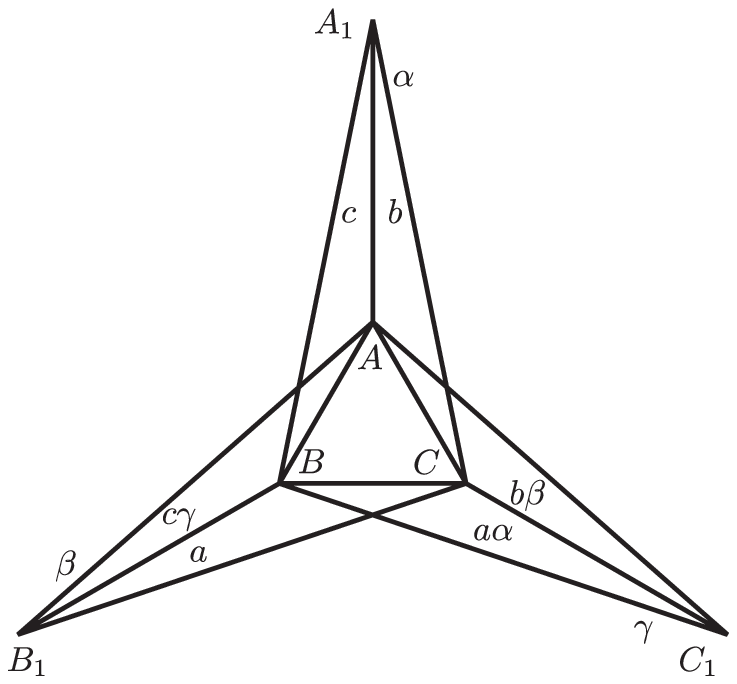}\\
  \caption{Complex $\T_0$ associated with $\nuke$}\label{fig:complex}
\end{figure}

Denote
\[\Xi_n=\Xi\otimes\bim^{\otimes
n},\quad\T_n=\T_0\otimes\alb^n\subset\Xi_n.\]

The set $\T_n$ is a fundamental domain of the action of $\G$ on
$\Xi_n$ (since $\alb^n=\alb^{\otimes n}$ is a right orbit
transversal of the action of $\G$ on $\bim^{\otimes n}$). The
space $\T_n$ is the quotient of the space $\T_0\times\alb^n$ by
the identifications
\[(\xi, v)\sim (\xi\cdot g^{-1}, g(v))\]
for all $g\in\nuke$, $\xi\in\T_0$, and $v\in\alb^n$ such that
$g|_v=\unit$ and $\xi\cdot g^{-1}\in\T_0$.

We will denote the points of $\T_n$ in the same way as the points
of $\Xi_n$: the point $\xi\otimes v$ is the equivalence class of
the point $(\xi, v)\in\T_0\times\alb^n$.

For $g\in\G$, denote $K_{g, n}=\T_n\cap\T_n\cdot g$, and let
$\kappa_{g, n}:K_{g, n}\arr K_{g^{-1}, n}$ be restriction onto
$K_{g, n}$ of the map $\xi\mapsto\xi\cdot g^{-1}$.

If $\T_0\cap\T_0\cdot g$ is non-empty for $g\in\G$, then there
exists $G\in\mathcal{G}$ such that $G\cdot g\in\mathcal{G}$. But
then $g\in G\subset\nuke$, and $G\cdot g=G$. It follows that
$K_{g, 0}$ is non-empty only for $g\in\nuke$; $K_{g, 0}=K_{g^{-1},
0}$; and $\kappa_{g, 0}$ is an identity map.

We have
\begin{gather*}K_{\alpha,
0}=A_1BC,\quad K_{\beta, 0}=B_1AC,\quad K_{\gamma, 0}=C_1AB,\\
K_{a, 0}=B_1BC,\quad K_{b, 0}=A_1AC,\quad K_{c, 0}=A_1AB,\\
K_{a\alpha, 0}=C_1CB,\quad K_{b\beta, 0}=C_1CA, \quad K_{c\gamma,
0}=B_1BA.
\end{gather*}
The remaining sets $K_{g, 0}$ for $g\in\mathfrak{A}$ are
one-dimensional:
\[K_{\alpha c, 0}=A_1B,\quad K_{a\alpha\gamma, 0}=C_1B.\]

\begin{proposition}
\label{prop:Kgn} The set $K_{g, n}$ is non-empty only for
$g\in\nuke$. The map $\kappa_{g, n}:K_{g, n}\arr K_{g^{-1}, n}$ is
given by the condition
\[\kappa_{g, n}(\xi\otimes v)=\xi\otimes h(v),\]
where $h\in\nuke$ is such that $h|_v=g$, and $\xi\in K_{h, 0}$.
\end{proposition}

\begin{proof}
If $\kappa_{g, n}$ is defined on a point $\xi\otimes v$, for
$\xi\in\T_0$ and $v\in\alb^n$, then $\xi\otimes v\cdot
g^{-1}=\xi_1\otimes v_1$ for some $\xi_1\in\T_0$ and
$v_1\in\alb^n$. Then, by the definition of a tensor product, there
exists $h\in\G$ such that $\xi=\xi_1\cdot h$ and $h\cdot v\cdot
g^{-1}=v_1$ in $\bim^{\otimes n}$. The first equality implies, by
the argument above, that $h\in\nuke$ and $\xi=\xi_1\in K_{h, 0}$.
Then $v_1=h(v)$ and $h|_v\cdot g^{-1}=\unit$, which implies that
$g=h|_v\in\nuke$.
\end{proof}

The next technical lemma will be used several times in our paper.

\begin{lemma}
\label{lem:stabilizersA} For every point $\xi\in\T_n$ the
stabilizer $\G_\xi$ of $\xi$ in $\G$ belongs either to
$\mathcal{G}$ or to the set
\[\mathcal{G}_1=\left\{\langle a\alpha\gamma, \alpha c\rangle,\quad\langle a\alpha\gamma\rangle,\quad\langle\alpha
c\rangle\right\}.\] In particular, it is generated by a subset of
$\mathfrak{A}$.

If $\xi\in\T_n$ belongs to $K_{g, n}$ for
$g\in\nuke\setminus\mathfrak{A}$, then $\kappa_{g, n}(\xi)=\xi$,
i.e., $g$ belongs to $\G_\xi$.
\end{lemma}

\begin{proof}
We have seen above that stabilizer of every point $\xi\in\T_0$
belongs to $\mathcal{G}$.

Let us prove our lemma by induction. Suppose that it is true for
$n$, and let $g$ be an element of the stabilizer of a point
$\xi\otimes vx$ for $\xi\in\T_0$, $v\in\alb^n$, and $x\in\alb$.
Then $\xi\otimes vx\cdot g=\xi\otimes vx$, which means that there
exists $h\in\G$ such that $\xi\otimes v=\xi\otimes v\cdot h$ and
$h\cdot x=x\cdot g$. Consequently, elements of the stabilizer of
$\xi\otimes vx$ are sections at $x$ of the intersection of
$\G_{\xi\otimes v}$ with the stabilizer of $x$. In other words
\[\G_{\xi\otimes vx}=\phi_x(\G_{\xi\otimes v}),\]
where $\phi_x$ is the virtual endomorphism associated with $\G$
and $x\in\alb$.

The generators of $G\in\mathcal{G}$ (as they are listed in
Subsection~\ref{ss:nucleusofG}) acting non-trivially on the first
level belong to $\{\alpha, a, a\alpha\}$. Consequently, for every
$g\in G$ and $x\in\alb$ the section $g|_x$ is equal to $\phi_x(h)$
for some $h\in G$. Therefore, $\phi_x(G)=G|_x$ for all
$G\in\mathcal{G}$. The groups $G|_x$ are listed in
Subsection~\ref{ss:sectionsofGs} (one has also to add the sections
$\langle g\rangle|_x=\langle g|_x\rangle$ for $g\in\mathfrak{A}$).

We see that $G|_x\in\mathcal{G}\cup\mathcal{G}_1$ for all
$G\in\mathcal{G}$. Sections $g|_x$ of the elements of the groups
from $\mathcal{G}_1$ belong to $\langle \beta, b\rangle$. But all
subgroups of $\langle\beta, b\rangle$ belong to $\mathcal{G}$. It
follows that stabilizers of points of $\T_n$ belong to
$\mathcal{G}\cup\mathcal{G}_1$ for all $n$.

Let us prove the remaining part of the lemma. Let $\xi\in\T_0$,
$v\in\alb^n$, and $x\in\alb$ are such that $\xi\otimes vx$ belongs
to $K_{g, n+1}$ for $g\in\nuke\setminus\mathfrak{A}$. Then there
exists $h\in\nuke$ such that $\xi\cdot h=\xi$ and $h|_{vx}=g$.
Since $\mathfrak{A}$ is state-closed, $h|_v\notin\mathfrak{A}$. We
have then $\xi\otimes v=\xi\otimes h\cdot v=\xi\otimes h(v)\cdot
h|_v$, i.e., $\xi\otimes v\in K_{h|_v, n}$. By the inductive
assumption, $h|_v$ belongs to the stabilizer of $\xi\otimes v$.

Consequently, $\xi\otimes h(v)=\xi\otimes v\cdot
h|_v^{-1}=\xi\otimes v$. It follows that there exists $h'\in\nuke$
such that $\xi\cdot h'=\xi$, $h'h(v)=v$, and $h'|_{h(v)}=\unit$.
We have then for $h_1=hh'$:
\[\xi\cdot h_1=\xi,\qquad h_1\cdot v=h'h\cdot v=v\cdot
h'|_{h(v)}h|_v=v\cdot h|_v.\] We may assume therefore that
$h(v)=v$.

Note that $h|_v$ can not belong to any of the groups of the set
$\mathcal{G}_1$, since then $g=h|_{vx}\in\langle\beta,
b\rangle\subset\mathfrak{A}$. Consequently, $h|_v$ belongs to one
of the groups of the set $\mathcal{G}$.

If $h|_v(x)=x$, then $\xi\otimes vx=\xi\otimes h\cdot
vx=\xi\otimes vx\cdot g$, and $g$ belongs to the stabilizer of
$\xi\otimes vx$.

Suppose that $h|_v(x)\ne x$. All the generators of the groups in
the set $\mathcal{G}$ acting non-trivially on the first level
belong to the set $\{\alpha, a, \alpha a\}$ of elements having
trivial sections. Consequently, if $h|_v$ belongs to one of the
stabilizers from $\mathcal{G}$, then $h|_v(x)=\delta(x)$ for
$\delta\in\{\alpha, a, \alpha a\}$, and $\xi\otimes vx=\xi\otimes
v\cdot\delta\otimes x=\xi\otimes vh|_v(x)$, therefore $g$ belongs
to the stabilizer of $\xi\otimes vx$.
\end{proof}

Let us describe now a recursive procedure of constructing the
complexes $\T_n$.

\begin{theorem}
\label{th:pastingTn} The space $\T_{n+1}$ is the quotient of the
space $\T_n\times\alb$ by the equivalence relation generated by
the identifications
\[\xi\otimes x=\kappa_{g, n}(\xi)\otimes g(x)\]
for all $g\in\mathfrak{A}$, $x\in\alb$, $\xi\in K_{g, n}$, such
that $g|_x=\unit$.

The set $K_{g, n+1}$ for $g\in\mathfrak{A}$ is equal to
\[\bigcup_{h\in\mathfrak{A}, x\in\alb, h|_x=g} K_{h, n}\otimes
x.\] The map $\kappa_{g, n+1}:K_{g, n+1}\arr K_{g, n+1}$ for
$g\in\mathfrak{A}$ acts by the rule
\[\kappa_{g, n+1}(\xi\otimes x)=\kappa_{h, n}(\xi)\otimes h(x),\]
where $h\in\mathfrak{A}$ is such that $h|_x=g$, and $\xi\in K_{h,
n}$.
\end{theorem}

All the information used in the inductive pasting rule of
Theorem~\ref{th:pastingTn} is read directly from the wreath
recursion~\eqref{eq:extension1}--\eqref{eq:extension5}
(Subsection~\ref{ss:twoextension}) or from the structure of the
automaton $\mathfrak{A}$ on Figure~\ref{fig:automaton}.

For instance, the identification of the copies $\T_n\times x$ of
$\T_n$ are given by the maps:
\begin{eqnarray*}
(\kappa_{\alpha, n}, \sigma)&:&K_{\alpha, n}\times\alb\arr K_{\alpha, n}\times\alb,\\
(\kappa_{a, n}, \pi)&:&K_{a, n}\times\alb\arr K_{a, n}\times\alb,\\
(\kappa_{a\alpha, n}, \pi\sigma)&:&K_{a\alpha, n}\times\alb\arr
K_{a\alpha, n}\times\alb,
\end{eqnarray*}
and
\begin{eqnarray*}
(\kappa_{\gamma, n}, \unit)&:&K_{\gamma, n}\times\{\two,
\three\}\arr
K_{\gamma, n}\times\{\two, \three\},\\
(\kappa_{a\alpha\gamma, n}, \pi\sigma):&:&K_{a\alpha\gamma,
n}\times\{\two, \three\}\arr K_{a\alpha\gamma, n}\times\{\two,
\three\}.
\end{eqnarray*}

\begin{proof}
If we replace $\mathfrak{A}$ by $\nuke$ everywhere in the theorem,
then it will follow directly from the definition of the tensor
product $\Xi_{n+1}=\Xi_n\otimes\bim$.

Therefore, the space $\T_{n+1}$ is obtained by taking the quotient
of the space $\T_n\times\alb$ by the identifications
\begin{equation}
\label{eq:identification}
(\xi, x)\sim (\kappa_{g, n}(\xi), g(x)),
\end{equation} where
$g\in\nuke$, $g|_x=\unit$, and $\xi\in K_{g, n}$. Suppose that $g$
does not belong to $\mathfrak{A}$. Then, by
Lemma~\ref{lem:stabilizersA}, $g$ belongs to the stabilizer
$\G_\xi$ of $\xi$, and $\G_\xi\in\mathcal{G}\cup\mathcal{G}_1$.
Identification~\eqref{eq:identification} becomes $(\xi, x)\sim
(\xi, g(x))$. If $x=g(x)$, the identification is trivial. If
$g(x)\ne x$ and $g$ is an element of one of the groups in the set
$\mathcal{G}$, then there exists $\delta\in\{\alpha, a, \alpha
a\}$ such that $\xi\cdot\delta=\xi$ and $\delta(x)=g(x)$. Then
identification~\eqref{eq:identification} is made using elements of
$\mathfrak{A}$. If $g$ is an element of a group from the set
$\mathcal{G}_1$, then either $g\in\mathfrak{A}$, or
$g\in\{a\alpha\gamma\alpha c, ac\}$. But $a\alpha\gamma\alpha
c=\pi(b\beta, b, b\beta, b)$ and $ac=\pi(b\beta, b\beta, b, b)$,
which contradicts the condition $g|_x=\unit$. We see that
identification~\eqref{eq:identification} is either trivial, or can
implemented by an element of $\mathfrak{A}$.

It remains to prove that every point of $K_{g, n+1}$ for
$g\in\mathfrak{A}$ can be represented by $\xi\otimes x$ for
$\xi\in K_{h, n}$ and $x\in\alb$, where $h\in\mathfrak{A}$ is such
that $h|_x=g$.

Every point of $K_{g, n+1}$ can be written as $\xi\otimes x$ for
$\xi\in K_{h_0, n}$ and $x\in\alb$, where $h_0\in\nuke$ is such
that $h_0|_x=g$. Suppose that $h_0\notin\mathfrak{A}$. Then, by
Lemma~\ref{lem:stabilizersA}, $h_0$ belongs to the stabilizer
$\G_\xi\in\mathcal{G}\cup\mathcal{G}_1$. It is enough then to show
that there exist $h\in\G_\xi\cap\mathfrak{A}$ and $y\in\alb$ such
that $\xi\otimes x=\xi\otimes y$ and $h|_y=g$. We have $\xi\otimes
x=\xi\otimes y$ if there exists $\delta\in\G_\xi$ such that
$\delta\cdot x=y\cdot\unit$.

Thus, theorem is proved if we show that for all
$G\in\mathcal{G}\cup\mathcal{G}_1$, $x\in\alb$, and $g\in
G|_x\cap\mathfrak{A}$ there exists $\delta\in G$ and $h\in
G\cap\mathfrak{A}$ such that $\delta|_x=\unit$ and
$h|_{\delta(x)}=g$.

Let us consider all the cases. If $G=\G_A=\langle\beta, \gamma, b,
c\rangle$, then for every $x\in\alb$ and $g\in
G|_x\cap\mathfrak{A}$ there exists $h\in G$ such that $h|_x=g$ (so
we can take $\delta=\unit$):
\begin{enumerate}
\item For $x=\one$, $G|_\one\cap\mathfrak{A}=\{\unit, \alpha, \beta, a, b, a\alpha,
b\beta\}$, and
\[\alpha=\beta|_\one,\quad\beta=\gamma|_\one,\quad a=(b\beta)|_\one,\quad
b=(c\gamma)|_\one,\quad a\alpha=b|_\one, \quad b\beta=c|_\one.\]
\item For $x=\two$, $G|_\two\cap\mathfrak{A}=\{\unit, \gamma, a\alpha, b\beta,
a\alpha\gamma\}$, and
\[\gamma=\beta|_\two,\quad a\alpha=b|_\two,\quad b\beta=c|_\two,\quad
a\alpha\gamma=(b\beta)|_\two.\]
\item For $x=\three$, $G|_\three\cap\mathfrak{A}=\{\unit, \alpha, b, c,
\alpha c\}$, and
\[\alpha=\beta|_\three,\quad b=c|_\three,\quad c=b|_\three,\quad\alpha c=(b\beta)|_\three.\]
\item For $x=\four$, $G|_\four\cap\mathfrak{A}=\{\unit, \beta, \gamma, b,
c, b\beta, c\gamma\}$, and
\[\beta=\gamma|_\four,\quad\gamma=\beta|_\four,\quad b=c|_\four,\quad
c=b|_\four,\quad b\beta=(c\gamma)|_\four,\quad
c\gamma=(b\beta)|_\four.\]
\end{enumerate}

If $G=\G_B=\langle\alpha, \gamma, a, c\rangle$, then
$G|_x=\langle\beta, b\rangle$ for all $x\in\alb$. For any pair $x,
y\in\alb$ there exists $\delta\in\langle\alpha,
a\rangle\subset\G_B$ such that $\delta(x)=y$. Hence, equalities
\[\beta=\gamma|_\one,\quad b=c|_\one,\quad b\beta=(c\gamma)|_\one,\]
finish the proof for $G=\G_B$.

The case $G=\G_C=\langle\alpha, \beta, a, b\rangle$ is considered
in the same way. We have $G|_x=\langle\alpha, \gamma, a, c\rangle$
for all $x\in\alb$, and
\begin{gather*}
\alpha=\beta|_\one,\quad a=(b\beta)|_\one,\quad a\alpha=b|_\one,\\
\gamma=\beta|_\two,\quad c=b|_\three,\quad c\gamma=(b\beta)|_\four,\\
\alpha c=(b\beta)|_\three,\quad a\alpha\gamma=(b\beta)|_\two.
\end{gather*}

Consider the case $G=\G_{A_1}=\langle\alpha, b, c\rangle$. If
$x\in\{\one, \two\}$, then $G|_x\cap\mathfrak{A}=\{\unit, a\alpha,
b\beta\}$, and
\[a\alpha=b|_x,\quad b\beta=c|_x.\]
If $x\in\{\three, \four\}$, then $G|_x\cap\mathfrak{A}=\{\unit, b,
c\}$ and
\[b=c|_x,\quad c=b|_x.\]

Cases $G=\G_{B_1}$ and $G=\G_{C_1}$ are similar to $\G_{A_1}$.

Cases $G\in\{\G_{AB}, \G_{AC}, \G_{AA_1}, \G_{AB_1}, \G_{AC_1}\}$
are straightforward: $G$ acts trivially on the first level;
$G|_x\cap\mathfrak{A}$ coincides with the standard generating set;
and for every $x\in\alb$, $g\in G|_x\cap\mathfrak{A}$ there exists
$h\in G\cap\mathfrak{A}$ such that $h|_x=g$.

There is nothing to prove for $G=\G_{BC}$.

Cases $G\in\{\G_{BA_1}, \G_{BB_1}, \G_{BC_1}, \G_{CA_1},
\G_{CB_1}, \G_{CC_1}\}$ are similar to $G=\G_{A_1}$: the group $G$
contains an element $\delta\in\{\alpha, a, a\alpha\}$;
intersection of $G|_x$ with $\mathfrak{A}$ is the standard
generating set of $G|_x$; for every $x\in\alb$ and $g\in
G|_x\cap\mathfrak{A}$ there exists $h\in G\cap\mathfrak{A}$ such
that either $h|_x=g$, or $h|_{\delta(x)}=g$.

If $G=\langle g\rangle\in\mathcal{G}\cup\mathcal{G}_1$ is cyclic,
then $g, g|_x\in\mathfrak{A}$, and we are done.

The only remaining case is $G=\langle a\alpha\gamma, \alpha
c\rangle$. Since $a\alpha\gamma\cdot \alpha c=\pi(b\beta, b,
b\beta, b)=\gamma ac$ is an involution, we have
\[G=\{\unit,\quad a\alpha\gamma=\pi\sigma(\beta, \unit, \unit,
\beta),\quad \alpha c=\sigma(b\beta, b\beta, b, b),\quad \gamma
ac=\pi(b\beta, b, b\beta, b)\},\] hence
\[G|_\one=\{\unit, \beta=(a\alpha\gamma)|_\one, b\beta=(\alpha
c)|_\one\},\quad G|_\four=\{\unit, \beta=(a\alpha\gamma)|_\four,
b=(\alpha c)|_\four\}.\] We also have
$a\alpha\gamma\cdot\two=\three\cdot\unit$,
$a\alpha\gamma\cdot\three=\two\cdot\unit$, and
\[G|_\two=G|_\three=\{\unit, b=(\alpha c)|_\two,
b=(\alpha c)|_\three\},\] which finishes the proof.
\end{proof}

\subsection{Equivariant map}
\label{ss:iota} Let us construct the complex $\T_1$. By
Theorem~\ref{th:pastingTn}, it is obtained by gluing four copies
$\T_0\otimes x$, for $x=\one, \two, \three, \four$, of $\T_0=\M$
along the following faces:
\begin{alignat*}{2}
(\kappa_{\alpha, 0}, \sigma):A_1BC\times\one &\sim A_1BC\times\two,&\quad
A_1BC\times\three&\sim A_1BC\times\four,\\
(\kappa_{a, 0}, \pi):B_1BC\times\one &\sim B_1BC\times\three,&\quad B_1BC\times\two&\sim
B_1BC\times\four,\\
(\kappa_{a\alpha, 0}, \pi\sigma):C_1BC\times\one &\sim C_1BC\times
\four,&\quad C_1BC\times\two &\sim C_1BC\times\three,
\end{alignat*}
where in each case the identification is identical on the first
coordinate. Note that the identification $(\kappa_{a\alpha\gamma,
0}, \pi\sigma)$ of $C_1B\times\two$ with $C_1B\times\three$
follows from the identification of $C_1BC\times\two$ with
$C_1BC\times\three$. The identification $(\kappa_{\gamma, 0},
\unit)$ is trivial.

The resulting complex $\T_1$ consists of a square pyramid and two
tetrahedra such that one face of each tetrahedron is attached to a
diagonal of the square pyramid (see Figure~\ref{fig:covcomplex},
but ignore the labels of the vertices this time).
Figure~\ref{fig:covcomplexn} shows the parts $\T_0\otimes i$ of
$\T_1$ (the vertices of each part $\T_0\otimes i$ are labeled as
in $\T_0$).

The map $\xi\otimes x\mapsto\xi$ folds the square pyramid in four
using reflections with respect to the planes passing through the
point $C\otimes\one=C\otimes\two=C\otimes\three=C\otimes\four$ and
through midpoints of two opposite sides of the base of the
pyramid, and folds the two tetrahedra along the planes passing
through the images of $C_1\otimes i$ and the hight of the pyramid.

\begin{figure}
\centering
\includegraphics{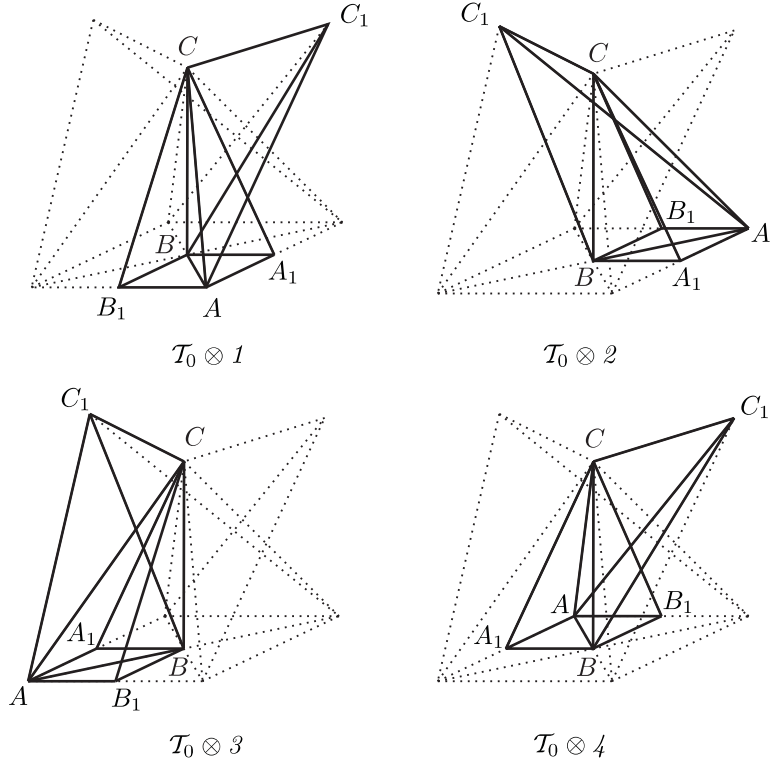}
\caption{Pasting together $\T_1$}\label{fig:covcomplexn}
\end{figure}

The elements of $\mathfrak{A}$ acting non-trivially on the first
level are $\alpha, a, a\alpha, \alpha c$, and $a\alpha\gamma$. The
first three elements have only trivial sections, hence they
produce only the identifications of the copies of $\T_0$ and no
maps $\kappa_{g, 1}$.

The element $\alpha c=\sigma(b\beta, b\beta, b, b)$ produces a
part of the map $\kappa_{b\beta, 1}$ switching $K_{\alpha c,
0}\times\one$ with $K_{\alpha c, 0}\times\two$ and a part of
$\kappa_{b, 1}$ switching $K_{\alpha c, 0}\times\three$ with
$K_{\alpha c, 0}\times\four$. But $K_{\alpha c, 0}$ is the segment
$A_1B$, for which we have identifications $A_1B\times\one\sim
A_1B\times\two$ and $A_1B\times\three\sim A_1B\times\four$ as
parts of the identifications $A_1BC\times i\sim
A_1BC\times\sigma(i)$.

The element $a\alpha\gamma=\pi\sigma(\beta, \unit, \unit, \beta)$
produces a part of the map $\kappa_{\beta, 1}$ switching
$K_{a\alpha\gamma, 0}\times\one=C_1B\times\one$ with
$K_{a\alpha\gamma, 0}\times\three=C_1B\times\three$. But these
sets are also identified in $\T_1$.

We see that all the maps $\kappa_{g, 1}$ are identical. Their
domains are
\begin{gather*}
K_{\alpha, 1}=B_1AC\otimes\{\one, \three\},\quad K_{\beta,
1}=C_1AB\otimes\{\one, \four\}, \quad K_{\gamma,
1}=B_1AC\otimes\{\two, \four\}\\
K_{a, 1}=C_1CA\otimes\one,\quad K_{c,
1}=A_1AC\otimes\{\three, \four\},\\
K_{b, 1}=A_1AB\otimes\{\three, \four\}\cup B_1BA\otimes\{\one,
\three\},\\
K_{a\alpha, 1}=A_1AC\otimes\{\one, \two\},\quad K_{c\gamma, 1}=C_1CA\otimes\four\\
K_{b\beta, 1}=A_1AB\otimes\{\one, \two\}\cup B_1BA\otimes\{\two,
\four\},\\
K_{a\alpha\gamma, 1}=C_1CA\otimes\two,\quad K_{\alpha c,
1}=C_1CA\otimes\three.
\end{gather*}

We have seen in Subsection~\ref{ss:sectionsofGs}
that for every $G\in\mathcal{G}$ and
every $x\in\alb$ there exists $H\in\mathcal{G}$ such that
$G|_x\subseteq H$.

For $G\in\{\G_A, \G_B, \G_C, \G_{A_1}, \G_{B_1}, \G_{C_1}\}$ and
for $x\in\alb$ denote by $I(G, x)$ the intersection of the groups
$H\in\mathcal{G}$ for which $G|_x\subseteq H$.

Denote for $G\in\{\G_A, \G_B, \G_C, \G_{A_1}, \G_{B_1},
\G_{C_1}\}$ and $g\in\G$:
\[I(G\cdot g, x)=I(G, g(x))\cdot g|_x.\]
Then $I(G\cdot g, x)$ is the intersection of the cosets $H\cdot
h$, for $H\in\mathcal{G}$ and $h\in\G$, containing the set
$(G\cdot g)|_x$. The maps $I(\cdot, x)$ satisfy the condition
\[I(U\cdot g, x)=I(U, g(x))\cdot g|_x,\]
for all cosets $U$ and for all $g\in\G$, $x\in\alb$.

Recall that $\Xi=\bigcup_{g\in\G}\T_0\cdot g$, where $\T_0$ is a
union of three tetrahedra $A_1ABC$, $B_1ABC$, $C_1ABC$ with
vertices corresponding to the groups $\G_A$, $\G_B$, $\G_C$, $\G_{A_1}$,
$\G_{B_1}$, and $\G_{C_1}$. It is checked directly that for every
$x\in\alb$ the image of the set of vertices of each of these
tetrahedra under $I(\cdot, x)$ is a subset of one of the
tetrahedra (we will see this also in the geometric description of
the maps $I(\cdot, x)$ below). Consequently, we can extend
$I(\cdot, x)$ by linearity to the whole complex $\Xi$. In this way
we get continuous maps satisfying the condition
\[I(\xi\cdot g, x)=I(\xi, g(x))\cdot g|_x,\]
for all $\xi\in\Xi$, $g\in\G$, and $x\in\alb$.
Hence, the map $I(\xi\otimes x)=I(\xi, x)$ is a well
defined continuous equivariant map from $\Xi_1=\Xi\otimes\bim$ to $\Xi$.

Figure~\ref{fig:covcomplex} shows the complex $\T_1$ in the same
way as it is shown on Figure~\ref{fig:covcomplexn}, but with
vertices labeled by their images under the map $I$ (except for
$B_1'$, which is mapped to $B$). One tetrahedron ($B_1'A_1BC_1$ on
Figure~\ref{fig:covcomplex}) is collapsed by $I$ onto the diagonal
of the square pyramid (triangle $A_1BC_1$). The remaining part of
$\T_1$ is a union of three tetrahedra and is mapped by a locally
affine homeomorphism onto $\T_0$.

\begin{figure}
\centering
\includegraphics{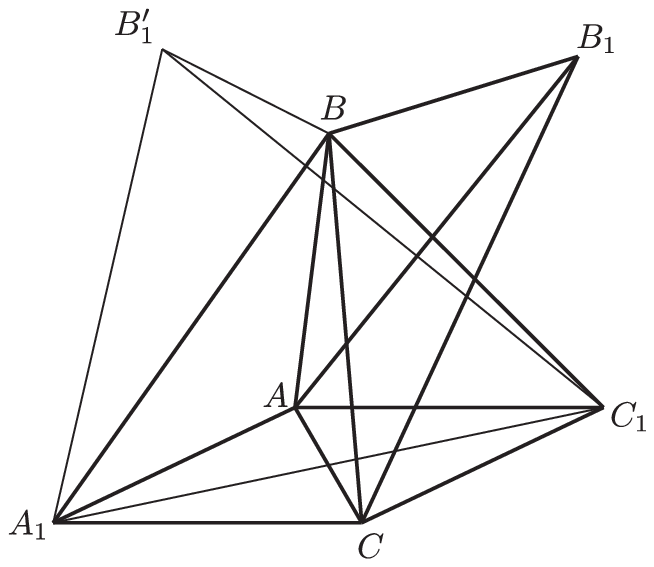}\\
\caption{The complex $\T_1$} \label{fig:covcomplex}
\end{figure}

We have the following formulae for $I$ (see
Figures~\ref{fig:complex}):
\begin{alignat*}{3}
I(A\otimes\one) &= C,   &\quad I(B\otimes\one) &= (AC), &\quad I(C\otimes\one) &= B,\\
I(A\otimes\two) &= C_1, &\quad I(B\otimes\two) &= (AC), &\quad I(C\otimes\two) &= B,\\
I(A\otimes\three) &= A_1, &\quad I(B\otimes\three) &= (AC), &\quad I(C\otimes\three) &= B,\\
I(A\otimes\four) &= A, &\quad I(B\otimes\four) &= (AC), &\quad
I(C\otimes\four) &= B,
\end{alignat*}
and
\begin{alignat*}{3}
I(A_1\otimes\one) &= (CC_1), &\quad I(B_1\otimes\one)
  &= (CA_1), &\quad I(C_1\otimes\one) &= B_1,\\
I(A_1\otimes\two) &=(CC_1), &\quad  I(B_1\otimes\two) &= (AC_1),
&\quad  I(C_1\otimes\two) &= B,\\
I(A_1\otimes\three) &= (AA_1), &\quad  I(B_1\otimes\three) &=
(CA_1), &\quad I(C_1\otimes\three) &= B,\\
I(A_1\otimes\four) &= (AA_1), &\quad I(B_1\otimes\four) &= (AC_1),
&\quad I(C_1\otimes\four) &= B_1,
\end{alignat*}
where $(XY)$ denotes the midpoint of the segment $XY$. The
vertices $X\otimes i$ are shown on Figure~\ref{fig:covcomplexn}: a
vertex $X\otimes i$ is labeled by $X$ on the part showing
$\T_0\otimes i$.

The map $I:\Xi\otimes\bim\arr\Xi$ is uniquely determined by its
restriction $I:\T_1\arr\T_0$ onto the fundamental domain, due to
equivariance.

Let us introduce a Euclidean structure on the complex $\Xi$ by
embedding the complex $\T_0$ into $\R^5$ in such a way that
\[\{\overrightarrow{BA},\quad\overrightarrow{BC},\quad\overrightarrow{AA_1},\quad\overrightarrow{BB_1},\quad
\overrightarrow{CC_1}\}\]
is the standard orthonormal basis of $\R^6$.

\begin{proposition}
\label{pr:Icontracting} The map $I:\Xi\otimes\bim\arr\Xi$ is
contracting.
\end{proposition}

\begin{proof}
The map $\xi\mapsto I(\xi\otimes x)$ is affine on $\T_0$ for every
$x\in\alb$. Let $\mathcal{I}_x$ be its linear part.

We have
\begin{gather*}
\mathcal{I}_\one(\overrightarrow{BA})  = \frac
12\overrightarrow{AC}= -\frac 12\overrightarrow{BA}+\frac
12\overrightarrow{BC},\quad
\mathcal{I}_\one(\overrightarrow{BC})  = -\frac 12\overrightarrow{BA}-\frac 12\overrightarrow{BC},\\
\mathcal{I}_\one(\overrightarrow{AA_1})= \frac
12\overrightarrow{CC_1},\quad
\mathcal{I}_\one(\overrightarrow{BB_1})= \frac
12\overrightarrow{AA_1},\quad
\mathcal{I}_\one(\overrightarrow{CC_1})= \overrightarrow{BB_1},
\end{gather*}
hence
\[\mathcal{I}_\one=\left(\begin{array}{ccccc}
-1/2 & -1/2 & 0   & 0   & 0\\
1/2  & -1/2  & 0   & 0   & 0\\
0    &  0   & 0   & 1/2 & 0\\
0    &  0   & 0   & 0   & 1\\
0    &  0   & 1/2 & 0   & 0
\end{array}\right).\]

We have
\begin{gather*}
\mathcal{I}_\two(\overrightarrow{BA})   = \frac 12\overrightarrow{AC_1}+\frac 12\overrightarrow{CC_1}=-\frac 12\overrightarrow{BA}+\frac 12\overrightarrow{BC}+\overrightarrow{CC_1}\\
\mathcal{I}_\two(\overrightarrow{BC})   = \frac 12\overrightarrow{AB}+\frac 12\overrightarrow{CB}=-\frac 12\overrightarrow{BA}-\frac 12\overrightarrow{BC}\\
\mathcal{I}_\two(\overrightarrow{AA_1}) = -\frac
12\overrightarrow{CC_1},\quad
\mathcal{I}_\two(\overrightarrow{BB_1}) = \frac
12\overrightarrow{CC_1},\quad
\mathcal{I}_\two(\overrightarrow{CC_1}) = \overrightarrow 0,
\end{gather*}
hence
\[
\mathcal{I}_\two=\left(\begin{array}{ccccc}
-1/2 & -1/2 & 0    & 0   & 0\\
1/2  & -1/2 & 0    & 0   & 0\\
0    & 0    & 0    & 0   & 0\\
0    & 0    & 0    & 0   & 0\\
1    & 0    & -1/2 & 1/2 & 0
\end{array}\right).
\]

The map $\mathcal{I}_\three$ acts by
\begin{gather*}
\mathcal{I}_\three(\overrightarrow{BA})   = \frac
12\overrightarrow{AA_1}+\frac 12\overrightarrow{CA_1}=\frac
12\overrightarrow{BA}-\frac
12\overrightarrow{BC}+\overrightarrow{AA_1},\quad
\mathcal{I}_\three(\overrightarrow{BC})   = -\frac 12\overrightarrow{BA}-\frac 12\overrightarrow{BC},\\
\mathcal{I}_\three(\overrightarrow{AA_1}) = -\frac
12\overrightarrow{AA_1},\quad
\mathcal{I}_\three(\overrightarrow{BB_1}) = \frac
12\overrightarrow{AA_1},\quad
\mathcal{I}_\three(\overrightarrow{CC_1}) = \overrightarrow 0,
\end{gather*}
hence
\[
\mathcal{I}_\three=\left(\begin{array}{ccccc}
1/2  & -1/2 & 0    & 0   & 0\\
-1/2 & -1/2 & 0    & 0   & 0\\
1    & 0    & -1/2 & 1/2 & 0\\
0    & 0    & 0    & 0   & 0\\
0    & 0    & 0    & 0   & 0
\end{array}\right).
\]

Finally,
\begin{gather*}
\mathcal{I}_\four(\overrightarrow{BA})   = \frac
12\overrightarrow{CA}=\frac 12\overrightarrow{BA}-\frac
12\overrightarrow{BC},\quad
\mathcal{I}_\four(\overrightarrow{BC})   = -\frac 12\overrightarrow{BA}-\frac 12\overrightarrow{BC},\\
\mathcal{I}_\four(\overrightarrow{AA_1}) = \frac
12\overrightarrow{AA_1},\quad
\mathcal{I}_\four(\overrightarrow{BB_1}) = \frac
12\overrightarrow{CC_1},\quad
\mathcal{I}_\four(\overrightarrow{CC_1}) = \overrightarrow{BB_1},
\end{gather*}
hence
\[
\mathcal{I}_\four=\left(\begin{array}{ccccc}
1/2  & -1/2 & 0   & 0   & 0\\
-1/2 & -1/2 & 0   & 0   & 0\\
0    & 0    & 1/2 & 0   & 0\\
0    & 0    & 0   & 0   & 1\\
0    & 0    & 0   & 1/2 & 0
\end{array}\right).
\]

We see that all matrices $\mathcal{I}_x$ are of the
block-triangular form $\left(\begin{array}{c|c}U_x & 0\\ \hline
W_x & V_x\end{array}\right)$, where $U_x$ and $V_x$ are of size
$2\times 2$ and $3\times 3$, respectively. For every vector $\vec
v$ and every $x\in\alb$ the Euclidean length $\|U_x\vec v\|$ is
equal to $\|v\|/\sqrt{2}$. Consequently, the norm of any product
$U_{x_1}U_{x_2}\cdots U_{x_n}$ of length $n$ is equal to
$2^{-n/2}$.

It is straightforward to check that for any two indices $x_1,
x_2\in\alb$ the norm of $V_{x_1}V_{x_2}$ is not more than
$1/\sqrt{2}$. Consequently, there is a constant $C$ such that norm
of $V_{x_1}V_{x_2}\cdots V_{x_n}$ is not more than $C2^{-n/4}$.
Norm of $W_x$ does not exceed $1$.

The product
$\mathcal{I}_{x_1}\mathcal{I}_{x_2}\cdots\mathcal{I}_{x_n}$ is of
the form $\left(\begin{array}{c|c}U & 0\\ \hline W &
V\end{array}\right)$, where
\[U=U_{x_1}U_{x_2}\cdots U_{x_n},\quad V=V_{x_1}V_{x_2}\cdots
V_{x_n},\] and
\[W=\sum_{k=1}^nV_{x_1}\cdots V_{x_{k-1}}W_{x_k}U_{x_{k+1}}\cdots U_{x_n}.\]
The norm of $W$ is estimated then as follows
\begin{multline*}\|W\|\le\sum_{k=1}^n\|V_{x_1}\cdots
V_{x_{k-1}}\|\cdot\|W_{x_k}\|\cdot\|U_{x_{k+1}}\cdots U_{x_n}\|\le\\
\sum_{k=1}^n C2^{-k/4}\cdot
2^{-(n-k)/2}\le\sum_{k=1}^nC2^{-k/4}\cdot 2^{-(n-k)/4}=nC2^{-n/4}.
\end{multline*}

It follows that the norm of the product
$\mathcal{I}_{x_1}\mathcal{I}_{x_2}\cdots\mathcal{I}_{x_n}$
uniformly converges to $0$ as $n$ goes to infinity. Consequently,
there exists $n$ such that the map $\xi\mapsto I^{(n)}(\xi\otimes
v)$ contracts all distances in $\Xi$ at least by $1/2$ for all
$v\in\alb^n$.
\end{proof}

\subsection{Complexes approximating the Julia set}
\label{ss:compleximgf}
Recall (see~\ref{ss:twoextension}) that $\img{f}$ is the index two
subgroup of $\G$ generated by $\alpha, \beta, \gamma, ab, bc$.
Then $\Xi$ is also a co-compact proper $\img{f}$-space.

Let $\bim$ and $\bim_f$ be the self-similarity bimodules of $\G$
and $\img{f}$, respectively. Since $\img{f}$ is a subgroup of
$\G$, the bimodule $\bim_f$ is a subset of $\bim$. Let
$\mathsf{Y}=\{\one, \two, \three\cdot b, \four\cdot b\}$ be the
common basis of these bimodules, corresponding to the wreath
recursion in Theorem~\ref{th:imgrecursion} defining $\img{f}$.

\begin{lemma}
The identical map $\Xi\times\mathsf{Y}^n\arr\Xi\times\mathsf{Y}^n$
induces a homeomorphism $\Xi\otimes_{\G}\bim^{\otimes
n}\arr\Xi\otimes_{\img{f}}\bim_f^{\otimes n}$.
\end{lemma}

\begin{proof}
Every element of $\Xi_n=\Xi\otimes_{\G}\bim^{\otimes n}$ can be
represented by $(\xi, \one^n)$ for some $\xi\in\Xi$, since the
group $\G$ is self-replicating (i.e., the left action of $\G$ on
$\bim$, and hence on $\bim^{\otimes n}$, is transitive). The same
is true for $\Xi\otimes_{\img{f}}\bim_f$. Two pairs $(\xi_1,
\one^n)$ and $(\xi_2, \one^n)$ represent the same point of
$\Xi\otimes_{\G}\bim^{\otimes n}$ (resp.\ of
$\Xi\otimes_{\img{f}}\bim_f^{\otimes n}$) if and only if there
exists $g\in\G$ (resp.\ $g\in\img{f}$) such that
\[\xi_1\cdot g=\xi_2,\quad g(\one^n)=\one^n,\quad g|_{\one^n}=\unit.\]

Denote by $K$ the kernel of the virtual endomorphism of $\G$
associated with the word $\one^n\in\alb^n$ (i.e., the subgroup of the
elements of $\G$ such that $g(\one^n)=\one^n$ and $g|_{\one^n}=\unit$). It is
sufficient to prove that $K<\img{f}$.

It follows from~\cite[Proposition~4.7]{nek:ssfamilies} that if a
product of the generators $\alpha, \beta, \gamma, a, b, c$ is
trivial in $\G$, then the numbers of occurrences of each of the
letters $a, b, c$ are even. Consequently, a product of the
generators of $\G$ is an element of $\img{f}$ if and only if the
total number of occurrences of the letters $a, b, c$ is even.

It follows from the wreath recursion defining $\G$ that the parity
of the total number of occurrences of the letters $a, b, c$ in $g$
is the same as in $g|_\one$, if $g(\one)=\one$. It follows that if
$g(\one^n)=\one^n$, then the total number of occurrences of the
letters $a, b, c$ in $g$ is the same as in $g|_{\one^n}$.
Consequently, if $g|_{\one^n}=\unit$, then $g\in\img{f}$.
\end{proof}

As a corollary of the lemma, we get that the map
$I:\Xi\otimes\bim\arr\Xi$ can be seen as an $\img{f}$-equivariant
map $I:\Xi\otimes\bim_f\arr\Xi$, and that the induced maps
$I_n:\Xi\otimes\bim_f^{\otimes (n+1)}\arr\Xi\otimes\bim_f^{\otimes
n}$ are the same as the maps $I_n:\Xi\otimes\bim^{\otimes
(n+1)}\arr\Xi\otimes\bim^{\otimes n}$.

Consequently, just restricting our construction to the index two
subgroup $\img{f}$ of $\G$ we get approximations
$\M_n=\Xi_n/\img{f}$ of the limit space $\lims[\img{f}]$. Namely,
the following theorem follows directly from
Corollary~\ref{cor:approximatinglims}.

\begin{theorem}
\label{th:JuliasetMn} Let $\iota_n:\M_{n+1}\arr\M_n$ and
$p_n:\M_{n+1}\arr\M_n$ be the maps induced by
$I_n:\Xi_{n+1}\arr\Xi_n$ and the correspondence $\xi\otimes
x\mapsto\xi$ for $x\in\bim_f$, respectively. Let $p_\infty$ be the
map induced by the maps $p_n$ on the inverse limit $\M_\infty$ of
the sequence
\[\M_0\stackrel{\iota_0}{\longleftarrow}\M_1\stackrel{\iota_1}{\longleftarrow}\M_2
\stackrel{\iota_2}{\longleftarrow}\cdots.\] Then the dynamical
system $(\M_\infty, p_\infty)$ is topologically conjugate to the
limit dynamical system of $\img{f}$, which is conjugate to the
action of $f$ on its Julia set (if $f$ is sub-hyperbolic).
\end{theorem}

The approximations $\M_n$ of the limit space $\lims[\img{f}]$
(i.e., of the Julia set of $f$, if Conjecture from
Section~\ref{s:thefunction} is true) can be constructed from the
complexes $\T_n$ in the following way.

\begin{proposition}
\label{pr:whMnpasting} Let $\til_n$ and $\kappa_{g, n}$ be as in
Theorem~\ref{th:pastingTn}. The complex $\M_n$ is obtained by
pasting two copies $\til_n$ and $\til_n\cdot a$ of $\til_n$ along
the sets $K_{a, n}$, $K_{b, n}$, $K_{c, n}$, $K_{a\alpha, n}$,
$K_{b\beta, n}$, $K_{c\gamma, n}$, $K_{a\alpha\gamma, n}$, and
$K_{\alpha c, n}$ by the action of the respective maps $\kappa_{g,
n}$ (i.e., identifying a point $\xi$ of one copy with the point
$\kappa_{g, n}(\xi)\cdot a$ in the other copy) and pasting the
sets $K_{\alpha, n}, K_{\beta, n}$, and $K_{\gamma, n}$ to
themselves (inside each of the copies) by the respective
$\kappa_{g, n}$.
\end{proposition}

\begin{proof}
The set $\T_0\cup\T_0\cdot a\subset\Xi$ is a fundamental domain of
the action of $\img{f}$ on $\Xi$. Consequently, the orbispace
$\M_n$ is obtained by identifying in the union $\T_0\cup\T_0\cdot
a$ any two points belonging to one $\img{f}$-orbit. Two different
points $\xi_1, \xi_2\in\T_0\cdot\unit$ belong to one
$\img{f}$-orbit if and only if there exists $g\in\nuke\cap\img{f}$
such that $\xi_1\cdot g=\xi_2$. By Lemma~\ref{lem:stabilizersA},
$g$ belongs to $\mathfrak{A}$. But $\{\alpha, \beta, \gamma\}$ are
the only elements of $\mathfrak{A}\cap\img{f}$. Two points
$\xi_1\cdot a, \xi_2\cdot a\in\T_0\cdot a$ belong to one
$\img{f}$-orbit if and only if $\xi_1, \xi_2\in\T_0$ belong to one
$\img{f}$-orbit. We have proved that two points inside one of the
copies of $\T_0$ are identified in $\M_n$ if and only if they are
either equal or are obtained from each other by application of one
of the transformations $\kappa_{\alpha, n}$, $\kappa_{\beta, n}$,
or $\kappa_{\gamma, n}$.

Suppose that $\xi_1\in\T_0\cdot\unit$ and $\xi_2\cdot
a\in\T_0\cdot a$ belong to one $\img{f}$-orbit, i.e., that
$\xi_1\cdot g=\xi_2\cdot a$ for some $g\in\img{f}$.

If $\xi_1=\xi_2$, then $ga$ belongs to the stabilizer $\G_{\xi_1}$
of $\xi_1$, which by Lemma~\ref{lem:stabilizersA} is generated by
elements of $\mathfrak{A}$. Since $ga\notin\img{f}$, one of these
generators $h$ does not belong to $\img{f}$. Then
$\xi_1=\xi_2\cdot h$ for $h\in\mathfrak{A}\setminus\img{f}$.

If $\xi_1\ne\xi_2$, then $h=ga\in\mathfrak{A}$ by
Lemma~\ref{lem:stabilizersA}, and we again have $\xi_1=\xi_2\cdot
h$ for $h\in\mathfrak{A}\setminus\img{f}$. Consequently, two
points belonging to different copies of $\T_0$ are identified by
transformations $\kappa_{g, n}$ for
$g\in\mathfrak{A}\setminus\img{f}$.
\end{proof}

In particular, the space $\M=\Xi/\img{f}$ is obtained by taking
two copies $\T_0$ and $\T_0\cdot a$ of $\T_0$ and pasting them
together by the maps $\xi\mapsto\xi\cdot a$ along $A_1AB$,
$A_1AC$, $B_1BA$, $B_1BC$, $C_1CA$, $C_1CB$. We get in this way
three solid balls (with surfaces equal to doubles of the faces
$A_1BC$, $AB_1C$ and $ABC_1$) with a common spherical hole (whose
surface is double of the triangle $ABC$). See a schematic diagram
of the complex on Figure~\ref{fig:complex2}.

\begin{figure}\centering
\includegraphics{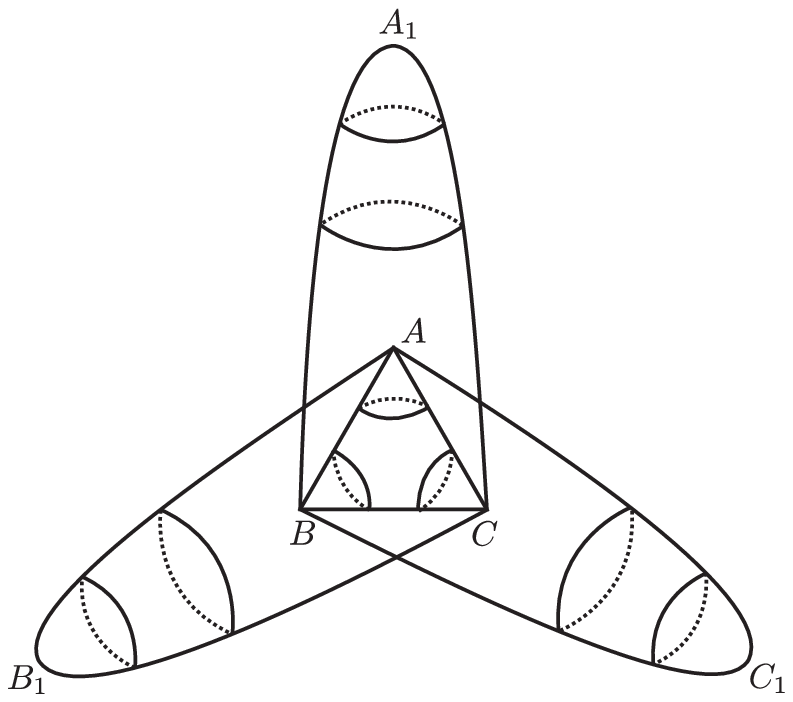}\\
\caption{Complex $\M=\Xi/\img{f}$} \label{fig:complex2}
\end{figure}

The covering $p_n:\M_{n+1}\arr\M_n$ is induced by the
correspondence $(\xi\otimes x)\mapsto\xi$, where
$\xi\in\Xi_n=\Xi\otimes\bim_f^{\otimes n}$ and $x\in\bim_f$.

We will denote by $[\xi]$ the image of $\xi\in\Xi_n$ in $\M_n$,
i.e., the $\img{f}$-orbit of $\xi$.

\begin{proposition}
\label{pr:coveringhatmn} The covering $p_n:\M_{n+1}\arr\M_n$ acts
by the rule
\begin{alignat*}{2}
p_n([\xi\otimes\one]) &=[\xi], &\quad p_n([\xi\otimes\two]) &= [\xi],\\
p_n([\xi\otimes\three]) &= [\xi\cdot a], &\quad
p_n([\xi\otimes\four]) &= [\xi\cdot a],
\end{alignat*} and
\begin{alignat*}{2}
p_n([\xi\otimes\one\cdot a]) &= [\xi\cdot a], &\quad
p_n([\xi\otimes\two\cdot a]) &= [\xi\cdot a],\\
p_n([\xi\otimes\three]) &= [\xi], &\quad p_n([\xi\otimes\four]) &=
[\xi].
\end{alignat*}
\end{proposition}

\begin{proof}
We have for $\xi\in\Xi_n$:
\[\xi\otimes\three=\xi\cdot c\otimes\three\cdot b,\]
hence $p_n([\xi\otimes\three])=[\xi\cdot c]=[\xi\cdot a]$.
Similarly,
\[\xi\otimes\four=\xi\cdot c\otimes\four\cdot b,\]
implies that $p_n([\xi\otimes\four])=[\xi\cdot a]$.

Since
\[\xi\otimes\one\cdot a=\xi\cdot b\beta\otimes\one,
\]
$p_n([\xi\otimes\one\cdot a])=[\xi\cdot b\beta]=[\xi\cdot a]$.
Similarly, since
\[\xi\otimes\two\cdot a=\xi\cdot b\alpha\beta\alpha\otimes\two,\]
$p_n([\xi\otimes\two\cdot a])=[\xi\cdot
b\alpha\beta\alpha]=[\xi\cdot a]$.

Equalities
\[\xi\otimes\three\cdot a=\xi\otimes\three\cdot b\cdot ba,\quad\xi\otimes\four\cdot a=
\xi\otimes\four\cdot b\cdot ba\] show that $p_n([\xi\otimes i\cdot
a])=[\xi]$ for $i\in\{\three, \four\}$.
\end{proof}

Equivariance of $I_n:\Xi_{n+1}\arr\Xi_n$ and the fact that
$I_n(\T_{n+1})=\T_n$ imply that for any $\xi\in\T_{n+1}$
\[\iota_n([\xi])=[I_n(\xi)],\qquad\iota_n([\xi\cdot a])=[I_n(\xi)\cdot a].\]

\subsection{Spaces $\M_n$ as subsets of the Julia set}

Restriction of the map $I:\T_1\arr\T_0$ onto closure of
$\T_1\setminus (C_1ABC\otimes\{\two, \three\})$ is a homeomorphism
(see Figures~\ref{fig:covcomplexn},~\ref{fig:covcomplex} and
definition of $I$ in Subsection~\ref{ss:iota}). Let
$\Theta:\T_0\arr\T_1$ be its inverse.

It is checked directly that for every $g\in\mathfrak{A}$ we have
$\Theta(K_{g, 0})\subseteq K_{g, 1}$. Suppose that $\xi\cdot
g_1=\xi\cdot g_2$ for $g_1, g_2\in\G$ and $\xi\in\T_0$. Then
$\xi=\xi\cdot g_2g_1^{-1}$, hence $g_2g_1^{-1}$ belongs to the
stabilizer $\G_\xi$ of $\xi$. The stabilizer of $\Theta(\xi)$
contains $\G_\xi$, by Lemma~\ref{lem:stabilizersA}, since
$\Theta(K_{g, 0})\subseteq K_{g, 1}$ and all transformations
$\kappa_{g, 1}$ act trivially. Consequently, $\Theta(\xi)\cdot
g_1=\Theta(\xi)\cdot g_2$.

We have proved that $\Theta$ can be extended by the rule
$\Theta(\xi\cdot g)=\Theta(\xi)\cdot g$ to a $\G$-equivariant map
$\Theta:\Xi\arr\Xi\otimes\bim$. It will be a section of the map
$I:\Xi\otimes\bim\arr\Xi$, i.e., $I\circ\Theta:\Xi\arr\Xi$ is
identical.

Define
\[\Theta_n(\xi\otimes v)=\Theta(\xi)\otimes v\]
for $\xi\in\Xi$ and $v\in\bim^{\otimes n}$. The map $\Theta_n$ is
well defined by equivariance of $\Theta$. We get hence a sequence
$\Theta_n:\Xi_n\arr\Xi_{n+1}$ of sections of the maps
$I_n:\Xi_{n+1}\arr\Xi_n$.

Denote by $\theta_n:\M_n\arr\M_{n+1}$ the maps induced by
$\Theta_n$ on the orbispaces $\M_n=\Xi_n/\img{f}$. The maps
$\theta_n$ are sections of the maps $\iota_n:\M_{n+1}\arr\M_n$.

We get hence natural homeomorphism $\wt\theta_n$ of the space
$\M_n$ with a subset of the inverse limit
$\M_\infty\approx\lims[\img{f}]$. It is the limit of the maps
$\theta_{n+k}\circ\theta_{n+k-1}\circ\cdots\circ\theta_n:\M_n\arr\M_{n+k+1}$
as $k\to\infty$.

The next theorem is now straightforward.

\begin{theorem}
\label{th:wttheta} Let $\wt\theta_n:\M_n\arr\M_\infty$ be the
limit of the maps $\theta_{n+k}\circ\cdots\circ\theta_n$. Then
\[\wt\theta_{n+1}(\M_{n+1})\supset\wt\theta_n(\M_n),\quad p_\infty^{-1}(\wt\theta_n(\M_n))=\wt\theta_{n+1}(\M_{n+1}),\quad
p_n={\wt\theta}^{-1}_n\circ p_\infty\circ\wt\theta_{n+1},\] and
the set $\bigcup_{n\ge 1}\wt\theta_n(\M_n)$ is dense in
$\M_\infty$.
\end{theorem}

We will give later a natural description of the sets
$\wt\theta_n(\M_n)$ as subsets of the Julia set of $f$.

\section{Skew product decomposition}
\label{s:tripods}
\subsection{The projection $(z, w)\mapsto w$}
\label{ss:zpp} Projection $(z, w)\mapsto w$ is a semicojugacy of
$f$ with the rational function $\wh{f}=(1-2/w)^2$. By
functoriality of the iterated monodromy groups
(see~\cite{nek:filling}), the projection induces a group
homomorphism
\[\nu:\img{f}\arr\img{\wh f}\] and a morphism
\[\mu:\bim_f\arr\bim_{\wh f}\] of the corresponding self-similarity
bimodules such that \[\mu(g_1\cdot x\cdot
g_2)=\nu(g_1)\cdot\mu(x)\cdot\nu(g_2)\] for all $g_1,
g_2\in\img{f}$ and $x\in\bim_f$. The images of the generators
$\alpha, \beta, \gamma$ in $\img{\wh f}$ are trivial (since they
correspond to loops in which $w$ is constant). The images of the
generators $s$ and $t$ are generators of $\img{\wh f}$ (which we
will also denote $s$ and $t$) corresponding to the loops around
the post-critical points $0$ and $1$ of $\wh f$, respectively (see
Figure~\ref{fig:generators}). The basis elements $\one,
\two\in\alb$ will be mapped by $\mu$ to the same element of
$\bim_{\wh f}$, since the corresponding coset representatives
$\unit, \alpha$ are mapped to the same element $\unit$ of
$\img{\wh f}$. Similarly the elements $\three\cdot b, \four\cdot
b\in\bim_f$ corresponding to the last two coordinates of the
wreath recursion from Theorem~\ref{th:imgrecursion} will be mapped
to the same element by $\mu$.

Consequently, applying the maps $\nu$ and $\mu$ to the wreath
recursion of Theorem~\ref{th:imgrecursion}, we get the following
wreath recursion generating $\img{\wh f}$
\begin{equation}\label{eq:ts}t=(t^{-1}s^{-1}, t),\qquad s=\sigma,
\end{equation}
where $\sigma$ is the transposition. If $\{\onep', \twop'\}$ is
the basis of the bimodule $\bim_{\wh f}$ corresponding to this
wreath recursion, then
\[
\mu(\one)=\mu(\two)=\onep',\qquad\mu(\three\cdot b)=\mu(\four\cdot
b)=\twop'.
\]

Let ${\wh\G}$ be the group generated by
\[a=\sigma,\quad b=(a, c),\quad c=(b, b).\]
It is a quotient of the group $\G$ as a self-similar group. The
corresponding epimorphisms of groups and self-similarity bimodules
(which we will also denote by $\mu:\bim\arr\bim_{\wh\G}$ and
$\nu:\G\arr{\wh\G}$, as they are extensions of the maps $\mu$ and
$\nu$) are defined by
\[\mu(\one)=\mu(\two)=\onep,\qquad\mu(\three)=\mu(\four)=\twop,\]
and by $\nu(\alpha)=\unit$, $\nu(\beta)=\unit$,
$\nu(\gamma)=\unit$, $\nu(a)=a$, $\nu(b)=b$, $\nu(c)=c$, where
$\onep'=\onep$ and $\twop'=\twop\cdot b$.

It is proved in~\cite{nek:ssfamilies} that the group ${\wh\G}$ is
given by the presentation
\[{\wh\G}\cong\langle a, b, c\;|\;a^2=b^2=c^2=(ac)^2=(ab)^4=(bc)^4=\unit\rangle,\]
and hence is isomorphic to the group generated by reflections of
the Euclidean space with respect to the sides of an isosceles
rectangular triangle (so that $b$ is the reflection with respect
to the hypothenuse).

Let us take, for instance, triangle $\mathcal{D}\subset\R^2$ with
the vertices $A'=(1, 0)$, $B'=(0, 0)$, and $C'=(1, 1)$. Let the
generators $a$, $b$, and $c$ correspond to reflections with
respect to the lines $B'C'$, $A'C'$, and $B'A'$, respectively, as
it is shown on Figure~\ref{fig:triangle}.

\begin{figure}
\centering
\includegraphics{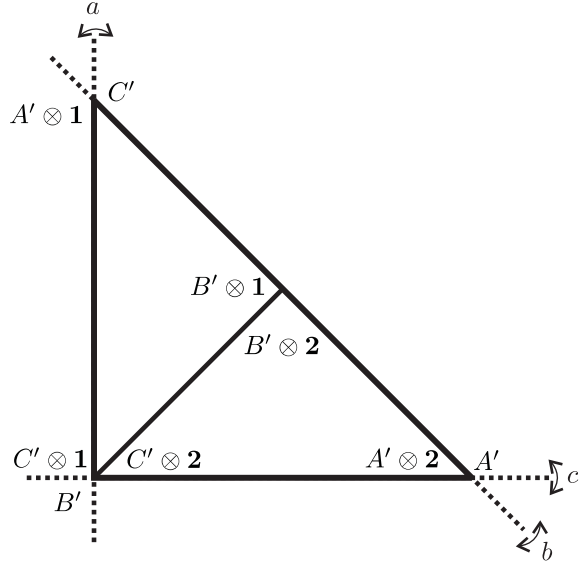}
\caption{Triangle $\mathcal{D}$} \label{fig:triangle}
\end{figure}

It follows from Proposition~\ref{pr:nucleusGamma} (and it is also
easy to prove directly, see~\cite{nek:ssfamilies}) that the
nucleus of ${\wh\G}$ is the union of the groups \begin{eqnarray*}
{\wh\G}_A&=&\langle b,
c\rangle=\nu(\G_A)=\nu(\G_{A_1})\cong D_4,\\
{\wh\G}_B&=&\langle a, c\rangle=\nu(\G_B)=\nu(\G_{B_1})\cong D_2,\\
{\wh\G}_C&=&\langle a, b\rangle=\nu(\G_C)=\nu(\G_{C_1})\cong D_4.
\end{eqnarray*}
The pairwise intersections of the subgroups of the nucleus are the
obvious ones:
\[{\wh\G}_{CB}=\langle a,
b\rangle\cap\langle a, c\rangle=\langle a\rangle,\quad
{\wh\G}_{CA}=\langle a, b\rangle\cap\langle b, c\rangle=\langle
b\rangle,\quad {\wh\G}_{AB}=\langle b, c\rangle\cap\langle a,
c\rangle=\langle c\rangle,\] and are groups of order two. Denote
\[\wh{\mathcal{G}}=\left\{{\wh\G}_A, {\wh\G}_B, {\wh\G}_C, {\wh\G}_{BC}, {\wh\G}_{CA}, {\wh\G}_{AB},
\{\unit\}\right\},\] and let ${\wh\Xi}$ be the image of the
complex $\Xi$ under the map $\nu:\G\arr{\wh\G}$, i.e., the complex
associated with the poset of sets of the form $G\cdot h$ for $h\in
{\wh\G}$ and $G\in\wh{\mathcal{G}}$. The complex ${\wh\Xi}$ is
isomorphic to the barycentric subdivision of the simplicial
complex obtained by tiling the Euclidean plane by the images of
the triangle $\mathcal{D}$ under the action of the group
${\wh\G}$. The groups ${\wh\G}_A, {\wh\G}_B, {\wh\G}_C$ correspond
to the vertices $A'$, $B'$, $C'$ of the triangle $\mathcal{D}$,
respectively. The groups ${\wh\G}_{AB}, {\wh\G}_{BC},
{\wh\G}_{AC}$ correspond to the edges $A'B'$, $B'C'$, $A'C'$,
respectively. Each group is the stabilizer of the corresponding
simplex of ${\wh\Xi}$. The triangle $\mathcal{D}$ is a fundamental
domain of the action of ${\wh\G}$ on ${\wh\Xi}$.

Define the map $P:\T_0\arr\mathcal{D}$ so that it is affine on the
tetrahedra $A_1ABC$, $B_1ABC$, $C_1ABC$ and acts on their vertices
by the rules
\begin{alignat}{3}
\label{eq:PABC1} P(A)&=A',&\quad P(B)&=B',&\quad P(C)&=C',\\
\label{eq:PABC2} P(A_1)&=A',&\quad P(B_1)&=B',&\quad P(C_1)&=C',
\end{alignat}
i.e., mapping a vertex corresponding to $\G_X\in\mathcal{G}$ to
the vertex corresponding to its image under $\nu$. Recall that we
have introduced in Subsection~\ref{ss:iota} a Euclidean structure
on $\T_0$ identifying it with a subset of $\R^5$. The points $A,
B, C$ become vertices of an isosceles right triangle after this
identification. The vectors $\overrightarrow{AA_1}$,
$\overrightarrow{BB_1}$, and $\overrightarrow{CC_1}$ are
orthogonal to the triangle $ABC$. Consequently, if we identify
$A', B', C'$ with the points points $A, B, C$ of $\R^5$, then $P$
will be the orthogonal projection of $\T_0$ onto the plain spanned
by $\vec{BA}$ and $\vec{BC}$.

The map $P:\T_0\arr\mathcal{D}$ can be extended to a continuous
map $P:\Xi\arr{\wh\Xi}$ such that
\begin{equation}
\label{eq:P} P(\xi\cdot g)=P(\xi)\cdot\nu(g)
\end{equation} for
every $\xi\in\Xi$ and $g\in\G$. The map $P$ will map the vertex
corresponding to a coset $\G_X\cdot g$ to the vertex corresponding
to the coset $\nu(\G_X\cdot g)$.

Denote by $L:{\wh\Xi}\otimes\bim_{\wh\G}\arr{\wh\Xi}$ the
${\wh\G}$-equivariant map induced by $I:\Xi\otimes\bim\arr\Xi$,
where $\bim_{\wh\G}$ is the self-similarity ${\wh\G}$-bimodule. It
is given by
\[L(P(\xi\otimes v))=P(I(\xi\otimes v)),\]
and is well defined by equivariance of $I$ and
property~\eqref{eq:P}.

Formulae defining $I$ imply that
\begin{alignat*}{3}
L(A'\otimes\onep)&=C',&\quad L(B'\otimes\onep)&=(A'C'),&\quad L(C'\otimes\onep)&=B',\\
L(A'\otimes\twop)&=A',&\quad L(B'\otimes\twop)&=(A'C'),&\quad
L(C'\otimes\twop)&=B'.
\end{alignat*}

One checks directly using equivariance of the map $L$ (see
also~\cite[4.8.3]{nek:filling}) that the maps $\xi\mapsto
L(\xi\otimes x)$ act on $\wh\Xi\approx\R^2$ by affine
transformations with the linear parts $\left(\begin{array}{rr}-1/2
& -1/2\\ 1/2 &
-1/2\end{array}\right)$ and $\left(\begin{array}{cc}1/2 & -1/2\\
-1/2 & -1/2\end{array}\right)$ for $x=\onep$ and $x=\twop$,
respectively. Compare these matrices with the top left corners of
the matrices $\mathcal{I}_x$ in the proof of
Proposition~\ref{pr:Icontracting}. Note that in both cases the
transformation $L(\cdot\otimes x)$ divides all distances of $\R^2$
by $\sqrt{2}$.

It is proved in~\cite{nek:filling} that the map
$L:{\wh\Xi}\otimes\bim_{\wh\G}\arr{\wh\Xi}$ and hence the maps
$L_n:{\wh\Xi}\otimes\bim_{\wh\G}^{\otimes
(n+1)}\arr{\wh\Xi}\otimes\bim_{\wh\G}^{\otimes n}$ are
homeomorphisms.

By equivariance of the maps $L_n$, the action of ${\wh\G}$ on
${\wh\Xi}\otimes\bim^{\otimes n}$ is obtained by conjugating the
action of ${\wh\G}$ on ${\wh\Xi}$ by the homeomorphism $L_n$. One
can show that the fundamental domains $\D_n=\D\otimes\{\onep,
\twop\}^n$ are rectangular isosceles triangles of area $2^{n-1}$
tiled by isometric copies of the triangle $\D$, as it is shown on
Figure~\ref{fig:dn}. The orbispaces ${\wh\Xi}\otimes\bim^{\otimes
n}/{\wh\G}$ can be identified with the fundamental domains $\D_n$
(i.e., the natural map from $\D_n$ to the orbispace is a
homeomorphism).

\begin{figure}
\centering
\includegraphics{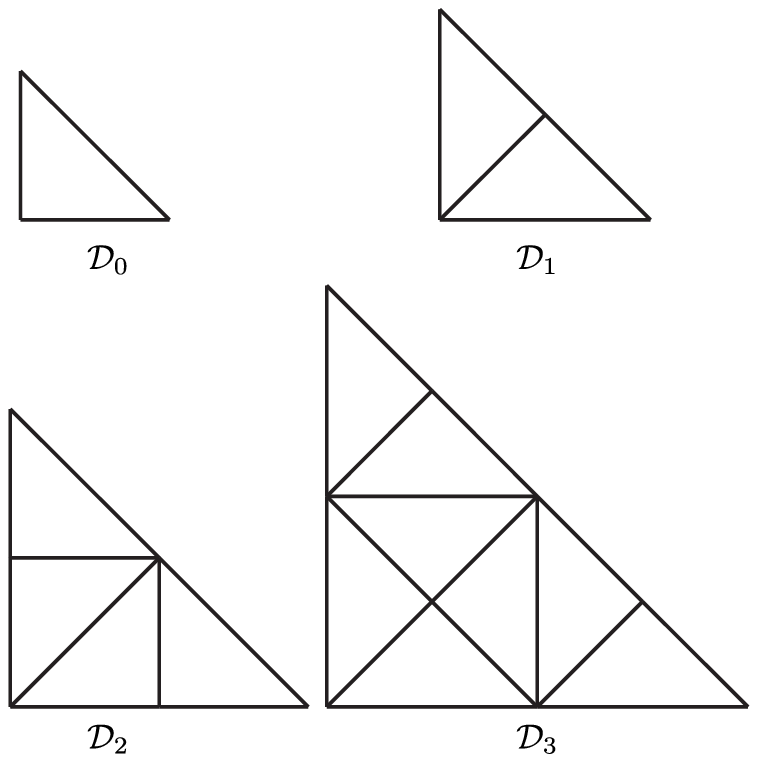}
\label{fig:dn} \caption{The sets $\D_n$}
\end{figure}

The spaces $\mathcal{S}_n={\wh\Xi}\otimes\bim^{\otimes n}/\img{\wh
f}$ are obtained by ``doubling'' the triangles: by taking two
copies $\D_n$ and $\D_n\cdot a$, and pasting them together along
the boundary.

Denote by
\[q_n:\mathcal{S}_{n+1}\arr\mathcal{S}_n,\quad\lambda_n:\mathcal{S}_{n+1}\arr\mathcal{S}_n\]
the covering induced by the projection $\zeta\otimes
vx\mapsto\zeta\otimes v$ and the map induced by $L_n$. We will
also denote $\mathcal{S}=\mathcal{S}_0$, $q=q_0$, and
$\lambda=\lambda_0$. The maps $\lambda_n$ act on each copy of
$\D_{n+1}$ in $\mathcal{S}_{n+1}$ as similitudes with coefficient
$\sqrt{2}/2$. The covering $q_n$ maps halves of each copy of
$\D_{n+1}$ isometrically to copies of $\D_n$ in $\mathcal{S}_n$
(according to the rules similar to the rules of
Proposition~\ref{pr:coveringhatmn}).

We can identify the spaces ${\wh\Xi}\otimes\bim_{\wh\G}^{\otimes
n}$ with $\C$ in such a way that the group ${\wh\G}$ acts on them
as the group of all isometries of the lattice $\Z[i]$, and the map
$L_n$ is identical (see~\cite[4.8.2--3]{nek:filling} for details).
The subgroup $\img{\wh f}$ of ${\wh\G}$ acts then on
${\wh\Xi}\otimes\bim_{\wh\G}^{\otimes n}\approx\C$ as the group of
affine transformations of the form $z\mapsto i^kz+z_0$, where
$k\in\Z$ and $z_0\in\Z[i]$. The coverings
$q_n:\mathcal{S}_{n+1}\arr\mathcal{S}_n$ are induced then by the
map $z\mapsto (1-i)z$.

It follows that the limit dynamical system of $\img{\wh f}$ is
conjugate with the map $F_{i-1}:\mathcal{S}_0\arr\mathcal{S}_0$
induced by the transformation $z\mapsto (1-i)z$ of $\C$. It is
well known (and follows from the general theory of iterated
monodromy groups) that the map $F_{1-i}$ is conjugate to the
rational function $\wh f$. The conjugating map is induced on
$\mathcal{S}_0$ by the map $z\mapsto\left(\wp(z)\right)^2$, where
$\wp$ is the Weierstrass function associated with the lattice
$\Z[i]$.

\subsection{Fibers of $\M_n\arr\mathcal{S}_n$}

Denote by $P_n:\Xi\otimes\bim^{\otimes
n}\arr{\wh\Xi}\otimes\bim_{\wh\G}^{\otimes n}$ the map given by
\[P_n(\xi\otimes v)=P(\xi)\otimes\mu(v),\]
where $\xi\in\Xi$ and $v\in\bim^{\otimes n}$. We also denote
$P_0=P$.

The next proposition follows directly from the definitions.

\begin{proposition}
\label{pr:rhon} The map $P_n:\Xi\otimes\bim^{\otimes
(n+1)}\arr{\wh\Xi}\otimes\bim_{\wh\G}^{\otimes n}$ induces maps
$\rho_n:\M_n\arr\mathcal{S}_n$ making the diagrams
\[\begin{array}{ccc}\M_{n+1} &\stackrel{p_n}{\arr} & \M_n\\
\mapdown{\rho_{n+1}} & & \mapdown{\rho_n}\\
\mathcal{S}_{n+1} & \stackrel{q_n}{\arr} & \mathcal{S}_n
\end{array},\qquad\begin{array}{ccc}\M_{n+1} &\stackrel{\iota_n}{\arr} & \M_n\\
\mapdown{\rho_{n+1}} & & \mapdown{\rho_n}\\
\mathcal{S}_{n+1} & \stackrel{\lambda_n}{\arr} & \mathcal{S}_n
\end{array}\]
commutative.
\end{proposition}

Denote by $Z_{\alpha, n}, Z_{\beta, n}, Z_{\gamma, n}$ the sets of
points of $\M_n$ of the form $[\xi]$ and $[\xi\cdot a]$, where
$\xi\in\T_n$ is a fixed point of the transformation
$\kappa_{\alpha, n}$, $\kappa_{\beta, n}$, $\kappa_{\gamma, n}$,
respectively.

\begin{theorem}
\label{th:fibersrhon} The set $\rho_n^{-1}(\xi)$ is a finite tree
for every $\xi\in\mathcal{S}_n$.

The map $p_n:\M_{n+1}\arr\M_n$ restricts for every
$\xi\in\M_{n+1}$ to a degree two branched covering
$\rho_{n+1}^{-1}(\xi)\arr\rho_n^{-1}(q_n(\xi))$. Denote by
$Z_{n+1}(\xi)$ its critical point.

Intersections of $Z_{\alpha, n}$, $Z_{\beta, n}$, and $Z_{\gamma,
n}$ with $\rho_n^{-1}(\xi)$ are singletons. Let us denote them by
$Z_{\alpha, n}(\xi)$, $Z_{\beta, n}(\xi)$, and $Z_{\gamma,
n}(\xi)$, respectively.

Let $\xi=[\zeta\otimes x]\in\mathcal{S}_{n+1}$, where
$\zeta\in\mathcal{D}_n$ and $x\in\{\onep, \onep\cdot a, \twop,
\twop\cdot a\}$. The tree $\rho_{n+1}^{-1}(\xi)$ is union of two
trees $T_1$ and $T_2$ such that
\begin{itemize}
\item[(i)] the common point of $T_1$ and $T_2$ is $Z_{n+1}(\xi)$;
\item[(ii)] restrictions of the map
$p_n:\rho_{n+1}^{-1}(\xi)\arr\rho_n^{-1}(q_n(\xi))$ onto $T_1$ and
$T_2$ are homeomorphisms;
\item[(iii)] $p_n(Z_{n+1}(\xi))=Z_{\alpha, n}(q_n(\xi))$;
\item[(iv)] $p_n^{-1}(Z_{\beta, n}(q_n(\xi)))\cap T_1=Z_{\alpha, n+1}(\xi)$, and
$p_n^{-1}(Z_{\beta, n}(q_n(\xi)))\cap T_2=Z_{\gamma, n+1}(\xi)$;
\item[(v)] if $x\in\{\onep, \onep\cdot a\}$, then $p_n^{-1}(Z_{\gamma, n}(q_n(\xi)))\cap
T_1=Z_{\beta, n+1}(\xi)$;
\item[(vi)]
if $x\in\{\twop, \twop\cdot a\}$, then $p_n^{-1}(Z_{\gamma,
n}(q_n(\xi)))\cap T_2=Z_{\beta, n+1}(\xi)$.
\end{itemize}
\end{theorem}

\begin{proof}
The set $\rho_{n+1}^{-1}([\zeta\otimes x])$ is equal to the set of
points of the form $[\xi\otimes y]$, where $\xi\in\T_n$ and
$y\in\alb\cdot\{\unit, a\}$ are such that $P_n(\xi)=\zeta$ and
$\mu(y)=x$. Note that for every $x\in\{\onep, \twop\}$ there
exists exactly two elements $y_1, y_2\in\alb\cdot\{\unit, a\}$
such that $\mu(y)=x$.

Depending on $x$, restriction of the map $p_{n+1}$ onto
$\rho_{n+1}^{-1}([\zeta\otimes x])$ acts by the rule
\[[\eta\otimes y]\mapsto [\eta],\quad\text{or}\quad[\eta\otimes y]\mapsto [\eta\cdot a],\]
as it is described in Proposition~\ref{pr:coveringhatmn}.

It follows that the restriction is two-to-one except for the
points $\eta$ such that $[\eta\otimes y_1]=[\eta\otimes y_2]$,
where $\{y_1, y_2\}=\mu^{-1}(x)$. It is sufficient to consider the
case $x\in\{\onep, \twop\}$ and $y\in\alb$. The equality
$[\eta\otimes y_1]=[\eta\otimes y_2]$ is equivalent to existence
of an element $g\in\nuke$ such that $\eta\cdot g=\eta$ and $g\cdot
y_1=y_2\cdot h$ for $h\in\img{f}$. Note that since
$\mu(y_1)=\mu(y_2)=x$, the element $g\in\nuke$ has an even number
of factors $a$ in its decomposition into a product of generators.
It follows now from $h\in\img{f}$ that $g$ has an even total
number of factors $b$ and $c$, i.e., that $g$ is an element of the
set $\langle\alpha, \beta, \gamma\rangle\cdot\langle bc\rangle$.
Since $y_1\ne y_2$, the element $g$ contains an odd number of
factors $\alpha$. Looking through the groups $\G_A, \G_B, \G_C,
\G_{A_1}, \G_{B_1}, \G_{C_1}$, we conclude that
\[g\in\langle\alpha, \beta\rangle\cup\langle\alpha, \gamma\rangle\cup\alpha\cdot\langle
bc\rangle.\] Consequently, either $g=\alpha$, or
$g\in\nuke\setminus\mathfrak{A}$ and the stabilizer of the point
$\eta$ contains $\alpha$, by Lemma~\ref{lem:stabilizersA}. In both
cases $\eta\cdot\alpha=\eta$, i.e., $[\eta]\in Z_{\alpha, n}$. In
the other direction, if $[\eta]\in Z_{\alpha, n}$, then
$[\eta]=[\eta']$ for $\eta'\cdot\alpha=\eta'$, and then
$[\eta'\otimes y_1]=[\eta'\cdot\alpha\otimes y_1]=[\eta'\otimes
y_2]$.

We have shown that a point $[\eta]$ has one preimage under the
restriction of $\rho_{n+1}$ onto $\rho_{n+1}^{-1}([\zeta\otimes
x])$ if and only if $[\eta]\in Z_{\alpha, n}$. In all the other
cases it has two preimages.

Let us prove by induction on $n$ that the intersections of
$Z_{\alpha, n}$, $Z_{\beta, n}$, and $Z_{\gamma, n}$ with the
fibers of the map $\rho_n$ are singletons.

It is true for $n=0$. By Theorem~\ref{th:pastingTn}, a point
$\eta\otimes y\in\T_{n+1}$, for $\eta\in\T_n$ and $y\in\alb$,
belongs to $K_{\alpha, n+1}$ if and only if $\eta\in K_{\beta, n}$
and $y\in\{\one, \three\}$.

We have \[\kappa_{\alpha, n+1}(\eta\otimes y)=\kappa_{\beta,
n}(\eta)\otimes y=\eta\cdot\beta\otimes y.\] If $\eta\otimes
y=\eta\cdot\beta\otimes y$, then it follows from
Theorem~\ref{th:pastingTn} that $\eta\cdot\beta=\eta\cdot g$ and
$g\cdot y=y$ for $g\in\langle a, \alpha, \gamma\rangle$. But then
$\eta=\eta\cdot\beta g^{-1}$, which implies that $\beta
g^{-1}\in\nuke$, and then by Lemma~\ref{lem:stabilizersA} that
$\eta$ is fixed under $\beta$.

It follows now from the inductive assumption (since
$\mu(\one)\ne\mu(\three)$) that intersection of $Z_{\alpha, n+1}$
with the fibers of $\rho_{n+1}$ are singletons.

Similarly, $\eta\otimes y$ belongs to $K_{\gamma, n+1}$ if and
only if $\eta\in K_{\beta, n}$ and $y\in\{\two, \four\}$. By the
same arguments as above, $\eta\otimes y$ is fixed under
$\kappa_{\gamma, n+1}$ if and only if $\eta$ is fixed under
$\kappa_{\beta, n}$. Consequently, intersections of $Z_{\gamma,
n+1}$ with the fibers of $\rho_{n+1}$ are singletons.

A point $\eta\otimes y$ belongs to $K_{\beta, n+1}$ if and only if
$\eta\in K_{\gamma, n}$ and $y\in\{\one, \four\}$, or $\eta\in
K_{a\alpha\gamma, n}$ and $y\in\{\one, \four\}$.

Suppose that $\eta\in K_{\gamma, n}$. Then \[\kappa_{\beta,
n+1}(\eta\otimes y)=\eta\cdot\gamma\otimes y.\] Again, by
Theorem~\ref{th:pastingTn}, if $\eta\otimes
y=\eta\cdot\gamma\otimes y$, then $\eta\cdot\gamma=\eta\cdot g$
and $g\cdot y=y$ for some $g\in\langle\alpha, \gamma, a\rangle$.
The equality $g(y)=y$ implies that $g\in\{\unit, \gamma,
\gamma^\alpha\}$ (as $\gamma^a=\gamma^\alpha$ and $\langle a,
\alpha\rangle$ acts faithfully on $\alb$). The case $g=\gamma$
contradicts $g\cdot y=y$ and $y\in\{\one, \four\}$. If
$g=\gamma^\alpha$, then $\eta$ is fixed by $\gamma\gamma^\alpha$,
which implies by Lemma~\ref{lem:stabilizersA} that it is also
fixed by $\gamma$. If $g=\unit$, then $\eta\cdot
g=\eta\cdot\gamma$ implies that $\eta$ is fixed by $\gamma$.

Suppose that $\eta\in K_{a\alpha\gamma, n}$ and $y\in\{\one,
\four\}$. Then $\kappa_{\beta, n+1}(\eta\otimes y)=\eta\cdot
a\alpha\gamma\otimes \sigma\pi(y)$. If $\eta\otimes y=\eta\cdot
a\alpha\gamma\otimes\sigma\pi(y)$, then there exists
$g\in\langle\alpha, a, \gamma\rangle$ such that $\eta\cdot
g=\eta\cdot a\alpha\gamma$ and $g\cdot y=\sigma\pi(y)$. The
condition $g(y)=\sigma\pi(y)$ implies that $g\in\{a\alpha,
a\alpha\gamma, a\alpha\gamma^\alpha\}$. The condition $g|_y=\unit$
eliminates the case $g=a\alpha\gamma$. If $g=a\alpha$, then
$\eta\cdot g=\eta\cdot a\alpha\gamma$ implies
$\eta\cdot\gamma=\xi$ (since $a\alpha$ commutes with $\gamma$). If
$g=a\alpha\gamma^\alpha$, then $\eta=\eta\cdot\gamma^\alpha$,
which also by Lemma~\ref{lem:stabilizersA} implies that
$\eta=\eta\cdot\gamma$.

We see that in all cases $[\eta]\in Z_{\gamma, n}$, hence
intersections of $Z_{\beta, n+1}$ with the fibers of $\rho_{n+1}$
are singletons.

Note that we have already proved that \[p_n(Z_{\alpha,
n+1}(\xi))=p_n(Z_{\gamma, n+1}(\xi))=Z_{\beta, n}(q_n(\xi))\] and
$p_n(Z_{\beta, n+1}(\xi))=Z_{\gamma, n}(q_n(\xi))$. The equality
$p_n(Z_{n+1}(\xi))=Z_{\alpha, n}(q_n(\xi))$ was also proved
before.

By the proved above, the fiber $\rho_{n+1}^{-1}(\zeta\otimes x)$
is obtained by identifying in the sets $\rho_n^{-1}(\zeta)\otimes
y_1$ and $\rho_n^{-1}(\zeta)\otimes y_2$ the points $Z_{\alpha,
n}(\zeta)\otimes y_1$ with $Z_{\alpha, n}(\zeta)\otimes y_2$.
Since the fibers of $\rho_0$ are trees, we get by induction that
the fibers of $\rho_n$ are also trees. It also follows that the
point $Z_{n+1}(\zeta\otimes x)=Z_{\alpha, n}(\zeta)\otimes
y_1=Z_{\alpha, n}(\zeta)\otimes y_2$ separates the subtrees
$\rho_n^{-1}(\zeta)\otimes y_1$ and $\rho_n^{-1}(\zeta)\otimes
y_2$. Let $y_1\in\{\one, \one\cdot a, \three, \three\cdot a\}$ and
$y_2\in\{\two, \two\cdot a, \four, \four\cdot a\}$. Then denote
$T_1=\rho_n^{-1}(\zeta)\otimes y_1$ and
$T_2=\rho_n^{-1}(\zeta)\otimes y_2$.

The rest of the theorem follows now from the recurrent description
of the sets $Z_{\alpha, n}$, $Z_{\beta, n}$, and $Z_{\gamma, n}$,
and the sets $K_{\alpha, n}$, $K_{\beta, n}$, and $K_{\gamma, n}$.
\end{proof}

\subsection{Unfolding trees and the map $\iota_n$}
\label{ss:unfolding} Consider the embedding of the complex $\T_0$
into $\R^5$ described in Subsection~\ref{ss:iota} and the
embedding of $\D$ into $\R^2$ described in
Subsection~\ref{ss:zpp}. Then projection $P:\T_0\arr\mathcal{D}$
acts as the map $(x_1, x_2, x_3, x_4, x_5)\mapsto (x_1, x_2)$. The
fiber $P^{-1}(x_1, x_2)$ is a union of three segments with one
common end $(x_1, x_2, 0, 0, 0)$. The other ends of the segments
are $Z_\alpha=Z_{\alpha, 0}((x_1, x_2))$, $Z_\beta=Z_{\beta,
0}((x_1, x_2))$, $Z_\gamma=Z_{\gamma, 0}((x_1, x_2))$,
respectively.

The point $Z_\alpha$ is the intersection of $P^{-1}(x_1, x_2)$
with the plane $A_1BC$, hence it has coordinates $(x_1, x_2, x_1,
0, 0)$. The coordinates of the point
\[Z_\beta=P^{-1}(x_1, x_2)\cap B_1AC\] are $(x_1, x_2,
0, 1-x_1-x_2, 0)$. The coordinates of the point
\[Z_\gamma=P^{-1}(x_1, x_2)\cap C_1AB\] are $(x_1, x_2, 0, 0,
x_2)$.

Consequently, $P^{-1}(x_1, x_2)$ is a tripod with feet $Z_\alpha$,
$Z_\beta$, and $Z_\gamma$ and lengths of legs $x_1$, $1-x_1-x_2$,
and $x_2$, respectively.

We see that the set of the graphs $P^{-1}(\xi)$ for
$\xi\in\mathcal{D}$ coincides with the set of all tripods with
feet marked by $Z_\alpha, Z_\beta, Z_\gamma$ such that the sum of
lengths of legs (the \emph{mass} of the tripod) is equal to one.
The triangle $\mathcal{D}$ is interpreted as the configuration
space of such tripods.

Theorem~\ref{th:fibersrhon} is reformulated then as the following
recurrent procedure of constructing the fibers $\rho_n^{-1}(\xi)$.

A \emph{marked tree} is a finite tree together with a marking of
three points by the letters $Z_\alpha$, $Z_\beta$, and $Z_\gamma$.
We assume that not all marked points coincide.

\begin{proposition}
\label{pr:Phi12} Let $T$ be a marked tree. Take two copies $T_1$
and $T_2$ of $T$ and paste them together, identifying the copies
of the point marked by $Z_\alpha$. Mark in the obtained tree
$\Phi_i(T)$ (for $i=\onep, \twop$) the copy of the point $Z_\beta$
in $T_1$ by $Z_\alpha$, the copy of $Z_\beta$ in $T_2$ by
$Z_\gamma$. If $i=\onep$, then mark the copy of $Z_\gamma$ in
$T_1$ by $Z_\beta$. Otherwise, mark by $Z_\beta$ the copy of
$Z_\gamma$ in $T_2$.

Then for every $\xi\in\mathcal{D}$, $x_1x_2\ldots x_n\in\{\onep,
\twop\}^n$, and $\delta\in\{\unit, a\}$ the tree
$\rho_n^{-1}(\xi\otimes x_1x_2\ldots x_n\cdot\delta)$ is
isomorphic as a marked tree to
\[\Phi_{x_n}\circ\cdots\circ\Phi_{x_2}\circ\Phi_{x_1}(\rho^{-1}(\xi)).\]
The covering $p_n:\rho_{n+1}^{-1}(\xi\otimes x_1x_2\ldots
x_{n+1})\arr\rho_n^{-1}(\xi\otimes x_1x_2\ldots x_n)$ maps the
copies $T_1$ and $T_2$ of $\rho_n^{-1}(\xi\otimes x_1x_2\ldots
x_n)$ tautologically onto $\rho_n^{-1}(\xi\otimes x_1x_2\ldots
x_n)$.
\end{proposition}

\begin{figure} \centering
  \includegraphics{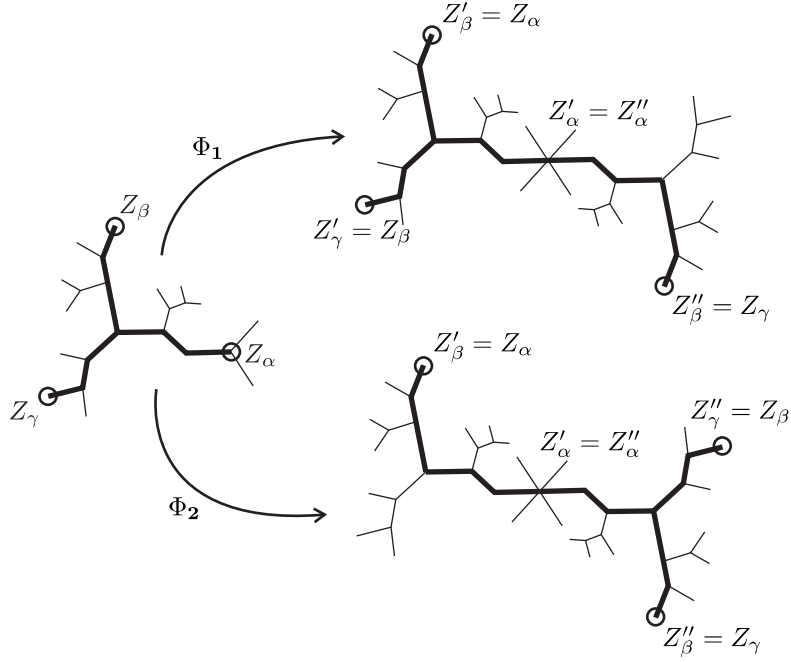}\\
  \caption{Unfolding transformations}\label{fig:unfolding}
\end{figure}

See Figure~\ref{fig:unfolding} for a description of the maps
$\Phi_\onep$ and $\Phi_\twop$. We denote by $Z_\alpha', Z_\beta',
Z_\gamma'$ the copies of $Z_\alpha, Z_\beta, Z_\gamma$ in $T_1$,
and by $Z_\alpha'', Z_\beta'', Z_\gamma''$ the copies of
$Z_\alpha, Z_\beta, Z_\gamma$ in $T_2$

If $T$ is a marked tree, then we call the tripod $Z_\alpha Z_\beta
Z_\gamma$ (i.e., the convex hull of the points $Z_\alpha$,
$Z_\beta$, and $Z_\gamma$) the \emph{Hubbard tripod} of $T$. Let
us denote by $O$ the common point of the legs of the Hubbard
tripod.

 Let us describe how the map $\iota:\M_1\arr\M$ acts on the
fibers of the maps $\rho_1$ and $\rho$. Denote by $x^{(1)}$ the
element $\one$, $\one\cdot a$, $\four$, $\four\cdot a$, if
$x=\onep$, $\onep\cdot a$, $\twop$, $\twop\cdot a$, respectively.
Denote $x^{(2)}=\two, \two\cdot a, \three, \three\cdot a$, if
$x=\onep, \onep\cdot a, \twop, \twop\cdot a$, respectively. Then
$\{x^{(1)}, x^{(2)}\}=\mu^{-1}(x)$, so that
$\rho_1^{-1}(\xi\otimes x)=\rho^{-1}(\xi)\otimes
x^{(1)}\cup\rho^{-1}(\xi)\otimes x^{(2)}$.

\begin{proposition}
\label{pr:iota0} Let $\xi, \zeta\in\mathcal{S}$ and $x\in\{\onep,
\onep\cdot a, \twop, \twop\cdot a\}$ be such that
$\lambda(\xi\otimes x)=\zeta$. Then $\iota:\rho_1^{-1}(\xi\otimes
x)\arr\rho^{-1}(\zeta)$ maps the edge connecting
$Z_\gamma(\xi)\otimes x^{(1)}$ and $O(\xi)\otimes x^{(1)}$
isometrically to $Z_\beta(\zeta)O(\zeta)$; collapses the edge
connecting $Z_\gamma(\xi)\otimes x^{(2)}$ and $O(\xi)\otimes
x^{(2)}$ to one point; and maps the path connecting
$Z_\beta(\xi)\otimes x^{(1)}$ and $Z_\beta(\xi)\otimes x^{(2)}$ to
the path connecting $Z_\alpha(\zeta)$ and $Z_\beta(\zeta)$
dividing all the distances in it by two.

If $x\in\{\onep, \onep\cdot a\}$, then $\iota(Z_\beta(\xi)\otimes
x^{(1)})=Z_\alpha(\zeta)$ and $\iota(Z_\beta(\xi)\otimes
x^{(2)})=Z_\gamma(\zeta)$, otherwise $\iota(Z_\beta(\xi)\otimes
x^{(1)})=Z_\gamma(\zeta)$ and $\iota(Z_\beta(\xi)\otimes
x^{(2)}=Z_\alpha(\zeta)$.
\end{proposition}

Here we wrote $Z_\alpha$, $Z_\beta$, and $Z_\gamma$ instead of
$Z_{\alpha, 0}$, $Z_{\beta, 0}$, and $Z_{\gamma, 0}$.

\begin{proof}
Let $\xi=(x_1, x_2)\in\mathcal{D}$ (the case
$\xi\in\mathcal{D}\cdot a$ is similar). Then $\rho^{-1}(\xi)$ is
the tripod with lengths of legs $(x_1, x_2, 1-x_1-x_2)$.

The map $\iota$ acts on $\rho^{-1}(\xi\otimes\onep)$ and on
$\rho^{-1}(\xi\otimes\onep\cdot a)$ by the rule (see the proof of
Proposition~\ref{pr:Icontracting}):
\[\iota(Z_\alpha\otimes\one)=
\left(\begin{array}{c}\frac 12\\ \frac 12\\ 0\\ 0\\
0\end{array}\right)+\mathcal{I}_\one\left(\begin{array}{c}x_1\\
x_2\\ x_1\\ 0\\ 0\end{array}\right)=\left(\begin{array}{c}(1-x_1-x_2)/2\\
(1+x_1-x_2)/2\\ 0\\ 0\\ x_1/2\end{array}\right),\]
\[\iota(Z_\beta\otimes\one)=\left(\begin{array}{c}(1-x_1-x_2)/2\\
(1+x_1-x_2)/2\\ (1-x_1-x_2)/2\\ 0\\
0\end{array}\right)=Z_\alpha(\zeta),\] similarly
\[\iota(Z_\gamma\otimes\one)=\left(\begin{array}{c}(1-x_1-x_2)/2\\
(1+x_1-x_2)/2\\ 0\\ x_2\\ 0\end{array}\right)=Z_\beta(\zeta),\]
\[\iota(O\otimes\one)=\left(\begin{array}{c}(1-x_1-x_2)/2\\
(1+x_1-x_2)/2\\ 0\\ 0\\ 0\end{array}\right)=O(\zeta),\] and
\[\iota(Z_\alpha\otimes\two)=\iota(Z_\alpha\otimes\one),\]
\[\iota(Z_\beta\otimes\two)=\left(\begin{array}{c}\frac 12\\ \frac 12\\ 0\\ 0\\
0\end{array}\right)+\mathcal{I}_\two\left(\begin{array}{c}x_1\\
x_2\\ 0\\ 1-x_1-x_2\\
0\end{array}\right)=\left(\begin{array}{c}(1-x_1-x_2)/2\\
(1+x_1-x_2)/2\\ 0\\ 0\\
(1+x_1-x_2)/2\end{array}\right)=Z_\gamma(\zeta)\]
\[\iota(Z_\gamma\otimes\two)=\left(\begin{array}{c}(1-x_1-x_2)/2\\
(1+x_1-x_2)/2\\ 0\\ 0\\
x_1\end{array}\right)=\iota(O\otimes\two).\] The statement of the
proposition follows now from the above formulae. The case
$x=\twop$ is analogical.
\end{proof}

\begin{corollary}
The map $\theta:\M\arr\M_1$ is isometric on the edges
$O(\zeta)Z_\beta(\zeta)$ of the trees $\rho^{-1}(\zeta)$ and
multiplies the lengths of the legs $O(\zeta)Z_\alpha(\zeta)$ and
$O(\zeta)Z_\gamma(\zeta)$ by two.
\end{corollary}

\subsection{Folding tripods and fibers of $(z, w)\mapsto w$}
\label{ss:foldingtripods} Consider the map
$p_0\circ\theta:\M\arr\M$. Note that since
$\lambda:\mathcal{S}_1\arr\mathcal{S}$ is a homeomorphism, we have
$\rho(\theta(\xi))=\lambda^{-1}(\rho(\xi))$ for all $\xi\in\M$.

It follows from the description of the fiberwise action of $p_0$
and $\theta$ that restriction
$p_0\circ\theta:\rho^{-1}(\xi)\arr\rho^{-1}(q_0\circ\lambda^{-1}(\xi))$
acts in the following way.

Double all the distances inside the legs $Z_\alpha(\xi)O(\xi)$ and
$Z_\gamma(\xi)O(\xi)$ of $\rho^{-1}(\xi)$, and then fold the path
$Z_\alpha(\xi)Z_\gamma(\xi)$ in two. The common image of
$Z_\alpha(\xi)$ and $Z_\gamma(\xi)$ is $Z_\beta(\zeta)$, the image
of the middle of the path $Z_\alpha(\xi)Z_\gamma(\xi)$ is
$Z_\alpha(\zeta)$ and the image of the vertex $Z_\beta(\xi)$ is
$Z_\gamma(\xi)$ (see Figure~\ref{fig:map}).

\begin{figure}
\centering  \includegraphics{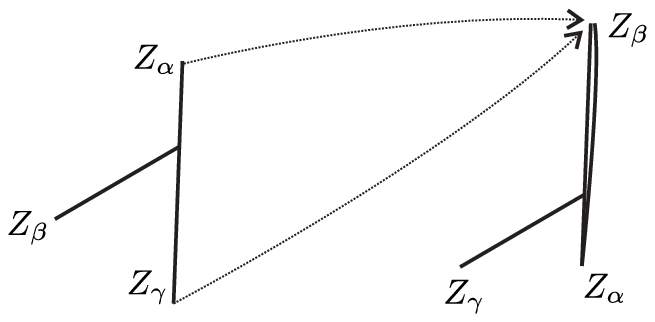} \caption{The folding map
$F$}\label{fig:map}
\end{figure}

If $(x_1, x_2, x_3)$ are the lengths of the legs $Z_\alpha O$,
$Z_\beta O$, and $Z_\gamma O$ in $\rho^{-1}(\xi)$, respectively,
then the lengths of the corresponding legs in the tripod
$\rho^{-1}(p_0\circ\theta(\zeta))$ are
\[F(x_1, x_2, x_3)=\Bigl(|x_1-x_3|,\quad 2\min(x_1, x_3),\quad x_2\Bigr),\]
respectively. Recall that $x_1+x_2+x_3=1$ and note that
$|x_1-x_3|+2\min(x_1, x_3)+x_2=x_1+x_2+x_3$.

\emph{Itinerary} of a triple $(x_1, x_2, x_3)$ is the sequences
$(k_1, k_2, \ldots)$, where $k_n=\onep$ if in $F^{n-1}(x_1, x_2,
x_3)=(y_1, y_2, y_3)$ we have $y_3>y_1$; otherwise, $k_n=\twop$.

Denote by $\rho_\infty:\M_\infty\arr\mathcal{S}_\infty$ the limit
of the maps $\rho_n:\M_n\arr\mathcal{S}_n$, where
$\M_\infty\approx\lims[\img{f}]$ and
$\mathcal{S}_\infty\approx\lims[\img{\wh f}]$ are the projective
limits of $\M_n$ and $\mathcal{S}_n$ (with respect to the maps
$\iota_n$ and $\lambda_n$), respectively. We denote $q^n=q_0\circ
q_1\circ\cdots\circ q_{n-1}$ and
$\lambda^n=\lambda_0\circ\lambda_1\circ\cdots\circ\lambda_{n-1}$.
We have $q_n\circ\lambda_{n+1}=\lambda_n\circ q_{n+1}$ for all
$n\ge 0$.

\begin{proposition}
\label{pr:fiberinvlim} Let $(\zeta_0,
\zeta_1=\lambda^{-1}(\zeta_0), \zeta_2=\lambda_1^{-1}(\zeta_1),
\ldots)$ be a point of $\mathcal{S}_\infty$ and suppose that
$\rho^{-1}(\zeta_0)$ is a tripod with the lengths of legs $(x_1,
x_2, x_3)$. Let $(k_1, k_2, \ldots)$ be the itinerary of the
tripod $\rho^{-1}(\zeta_0)$. Then the trees $\rho_n^{-1}(\zeta_n)$
are isometric to $\Phi_{k_1}\circ\cdots\circ\Phi_{k_n}(T_n)$,
where $T_n$ is the tripod with lengths of legs $F^n(x_1, x_2,
x_3)$.

The maps
$\iota_n:\rho_{n+1}^{-1}(\zeta_{n+1})\arr\rho_n^{-1}(\zeta_n)$
project each of the $2^n$ copies of the graph
$\Phi_{k_{n+1}}(T_{n+1})$ in
$\Phi_{k_1}\circ\cdots\circ\Phi_{k_n}(\Phi_{k_{n+1}}(T_{n+1}))$
onto its Hubbard tripod, divides the lengths of the legs $Z_\alpha
O$ and $Z_\gamma O$ by two, and then identifies each copy with the
corresponding copy of $T_n$ in
$\Phi_{k_1}\circ\cdots\circ\Phi_{k_n}(T_n)$.
\end{proposition}

\begin{proof}
If $(k_1, k_2, \ldots)$ is the itinerary of the tripod
$\rho^{-1}(\zeta)$, then $(k_2, k_3, \ldots)$ is the itinerary of
the tripod
$\rho^{-1}(q_0\circ\lambda_0^{-1}(\zeta))=\rho^{-1}(q_0(\zeta_1))$,
and $\zeta_1\in\D\otimes k_1\cdot\delta$ for $\delta\in\{\unit,
a\}$ (see the description of the action of
$p_0\circ\theta:\rho^{-1}(\zeta)\arr\rho^{-1}(\lambda_0^{-1}(\zeta))$
and Proposition~\ref{pr:Phi12}).

Note that
\[\lambda_0^{-1}\circ (q_0\circ q_1\circ\cdots q_{n-1})=(q_1\circ q_2\circ\cdots\circ q_n)\circ
\lambda_n^{-1},\]
which implies
\[q_0\circ\lambda_0^{-1}\circ q^n\circ(\lambda^n)^{-1}=
q_0\circ (q_1\circ q_2\circ\cdots\circ
q_n)\circ\lambda_n^{-1}\circ (\lambda^n)^{-1}=q^{n+1}\circ
(\lambda^{n+1})^{-1}.\]

It follows now by induction that $(k_{n+1}, k_{n+2}, \ldots)$ is
the itinerary of the tripod $\rho^{-1}(q^n\circ
(\lambda^n)^{-1}(\zeta))=\rho^{-1}(q^n(\zeta_n))$, which has
lengths of legs $F^n(x_1, x_2, x_3)$, and that
$q_1\circ\cdots\circ q_{n-1}(\zeta_n)\in\D\otimes k_n\cdot\delta$
for some $\delta\in\{\unit, a\}$.

We know that
\[\lambda_k(\D_k\otimes x)=\D_{k-1}\otimes x\]
for all $k\ge 1$ and $x\in\{\onep, \onep\cdot a, \twop, \twop\cdot
a\}$.

Therefore, if $\zeta_1\in\D\otimes k_1\cdot\delta_1$ for
$\delta_1\in\{\unit, a\}$, then
\[\zeta_n=(\lambda_1\circ\cdots\circ\lambda_{n-1})^{-1}(\zeta_1)\in\D_{n-1}\otimes
k_1\cdot\delta_1\] for all $n$.

Similarly, if $\delta_m\in\{\unit, a\}$ is such that
$q_1\circ\cdots\circ q_{m-1}(\zeta_m)\in\D\otimes
x_m\cdot\delta_m$, then for all $n>m$ we have
\begin{multline*}q_{n-m+1}\circ\cdots\circ q_{n-1}(\zeta_n)=\\
q_{n-m+1}\circ\cdots\circ q_{n-1}\circ
\lambda_{n-1}^{-1}\circ\cdots\circ\lambda_m^{-1}(\zeta_m)=\\
\lambda_{n-m}^{-1}\circ\cdots\circ\lambda_1^{-1}\circ
q_1\circ\cdots\circ q_{m-1}(\zeta_m)\in\D_{n-m}\otimes
x_m\cdot\delta_m,
\end{multline*}
since $q_{i+1}\circ\lambda_{i+1}^{-1}=\lambda_i^{-1}\circ q_i$ for
all $i$.

It follows that $\zeta_n\in\D\otimes k_nk_{n-1}\ldots
k_1\cdot\delta_1$. The first paragraph of the proposition follows
now from Proposition~\ref{pr:Phi12}. The second paragraph follows
from the definition of the map $\iota_n$ and
Proposition~\ref{pr:iota0}.
\end{proof}

\subsection{The spaces $\M_n$ as subsets of the Julia set}

A \emph{dendrite} is a path connected and locally path connected
space without simple closed curves.

\begin{proposition}
\label{pr:dendrites} The fibers of the map
$\rho_\infty:\lims[\img{f}]\arr\lims[\img{\wh f}]$ are dendrites.
\end{proposition}

\begin{proof}
By Proposition~\ref{pr:fiberinvlim}, every fiber of $\rho_\infty$
is homeomorphic to the inverse limit $A_\infty$ of a sequence of
finite trees $A_n$ with respect to maps $\iota_n:A_{n+1}\arr A_n$,
where $\iota_n$ is composition of a projection of $A_{n+1}$ onto a
subtree $A_{n+1}'$ and a homeomorphism $A_{n+1}'\arr A_n$.

Let us show that $A_\infty$ is path connected. Let $x, y\in
A_\infty$ be arbitrary points, and let $x_n, y_n$ be their images
in $A_n$. Let $\gamma_n$ be the unique arc connecting $x_n$ to
$y_n$ in $A_n$. The projection of $A_n$ onto $A_n'$ maps
$\gamma_n$ to its sub-arc $\gamma_n'$ by mapping the connected
components of $\gamma_n\setminus\gamma_n'$ to the endpoints of
$\gamma_n'$. It follows that the inverse limit of the arcs
$\gamma_n$ is an arc in $A_\infty$ connecting $x$ to $y$.

Every open neighborhood of $x\in A_\infty$ contains a
neighborhood which is the inverse image of an open subset of $A_n$ for some
$n$, hence it contans an open neighborhood which is the inverse limit
of subtrees of the trees $A_n$. By the argument above, this inverse
limit is path connected. Consequently, $A_\infty$ is locally path connected.

Suppose that $\gamma$ is an arc with endpoints $x, y\in
A_\infty$. Denote by $\gamma_{(n)}$ its image in $A_n$, and let $x_n, y_n$
be the images of $x, y$ in $A_n$. The set $\gamma_{(n)}$ is connected,
hence it
has to contain the unique arc $\gamma_n$ connecting $x_n$ and $y_n$. It follows
that $\gamma$ contains the inverse limit of the arcs $\gamma_n$.
But we have proved that the inverse limit of $\gamma_n$ is an arc
connecting $x$ and $y$. Consequently, there exists only one arc
connecting $x$ and $y$ in $A_\infty$.
\end{proof}

Recall that $\wt\theta_n:\M_n\arr\M_\infty$ denotes the limit of
$\theta_{n+k-1}\circ\cdots\theta_n:\M_n\arr\M_{n+k}$ and is a
homeomorphisms of $\M_n$ with $\wt\theta_n(\M_n)$. The spaces
$\M_n$ can be hence identified with subsets of the Julia set of
$f$. Let us denote by $\wt\M_n$ the image of $\wt\theta_n(\M_n)$
under the natural homeomorphism of
$\M_\infty\approx\lims[\img{f}]$ with the Julia set of $f$.

Denote by $\wt\iota_n:\wt\M_{n+1}\arr\wt\M_n$ the map
obtained from $\iota_n$ after identification of the spaces $\M_n$ with
the sets $\wt\M_n$.

\begin{proposition}
\label{pr:MninJulia} Intersections $\mathcal{J}_{w_0}$ of the
Julia set of $f$ with the lines $w=w_0$ are dendrites. The set
$\wt\M=\wt\M_0$ is the union of the convex hulls of the sets
$\{(0, w_0), (1, w_0), (w_0, w_0)\}$ inside the dendrite
$\mathcal{J}_{w_0}$.

We have $\wt\M_n=f^{-n}(\wt\M)$. The map
$\wt\iota_n:\wt\M_{n+1}\arr\wt\M_n$ acts as projection of the
trees $\wt\M_{n+1}\cap\mathcal{J}_{w_0}$ onto their sub-trees
$\wt\M_n\cap\mathcal{J}_{w_0}$.
\end{proposition}

If $T$ is a tree and $T'\subset T$ is a sub-tree, then projection
of $T$ onto $T'$ maps a point $t\in T$ to the end of the unique
path $\gamma$ starting at $t$, ending in a point $t'$ of $T'$ and
such that the only common point of $\gamma$ with $T'$ is $t'$.

\begin{proof}
The fact that intersections of the Julia set of $f$ with the lines
$w=w_0$ are dendrites and the last paragraph of the theorem follow
directly from Proposition~\ref{pr:iota0} and
Theorem~\ref{th:wttheta}.

By Theorem~\ref{th:fibersrhon}, for every
$\xi\in\mathcal{S}_{n+1}$ the point $Z_{\alpha, n}(q_n(\xi))$ is
the image of the critical point of the restriction of the covering
$p_n$ onto the fiber $\rho_{n+1}^{-1}(\xi)$. The critical point of
the restriction $(1-2z/w_0)^2$ of $f$ onto the fiber $w=w_0$ is
$w_0/2$, and its image is $0$. Consequently, the points $Z_\alpha$
correspond to the points $z=0$ of the respective slices of the
Julia set of $f$. Since $p_n(Z_{\alpha, n+1}(\xi))=Z_{\beta,
n}(q_n(\xi))$, the points corresponding to $Z_\beta$ are $z=1$;
since $p_n(Z_{\beta, n+1}(\xi))=Z_{\gamma, n}(q_n(\xi))$, the
points corresponding to $Z_\gamma$ in the slice $w=w_0$ is the
point $z=w_0$.

Since the maps $\theta_n$ are homeomorphic embeddings of trees, the
set $\theta_{n-1}\circ\cdots\circ\theta_0(\M)\subset\M_n$ is the union of the
convex hulls of the points $Z_\alpha(\xi), Z_\beta(\xi),
Z_\gamma(\xi)$ inside the fibers $\rho_n^{-1}(\xi)$. It follows that
$\wt\M$ is the union of the hulls of the points $(0, w_0)$, $(1,
w_0)$, $(w_0, w_0)$ inside the intersections of the Julia set with the
lines $w=w_0$, which are dendrites by Theorem~\ref{pr:dendrites}.
\end{proof}

Note that the set $\rho_0^{-1}(\D\cap\D\cdot a)\subset\M$ is the
union of tripods in which at least one leg has length zero. The
map $(z, w)\mapsto (\overline z, \overline w)$ is an automorphism
of the dynamical system $(\CP, f)$ and it changes the orientation
of the Hubbard tripod of $J_{w_0}\subset\C$ to the opposite one
(here the Hubbard tripod of $J_{w_0}$ is the convex hull of the
points $\{0, 1, w_0\}$ inside the dendrite $J_{w_0}$, i.e.,
intersection of $J_{w_0}$ with $\M$). It follows that the Hubbard
tripods of $J_{w_0}$ for real values of $w_0$ have only one
orientation, i.e., that they have at least one leg of length zero.
The real line of the Riemann sphere is homeomorphic to a circle.
The set $\D\cap\D\cdot a$ is also homeomorphic to a circle (it is
the boundary of the triangle $\D$). It follows that the subset of
the Julia set of $f$ corresponding to $\rho_0^{-1}(\D\cap\D\cdot
a)$ is precisely the union of the sets $J_{w_0}$ for
$w_0\in\R\cup\{\infty\}$.

The virtual endomorphism $\phi$ from
Proposition~\ref{pr:imgvirtend} corresponds to the letter
$\one\in\alb$ and is computed using the fixed point $(z,
w)=(0.3002\ldots + 0.3752\ldots i, 2i)$. The Hubbard tripod
$Z_\alpha Z_\beta Z_\gamma$ of the slice $J_{2i}$ is oriented
counterclockwise (see Figure~\ref{fig:generators}).

It follows that the sets $\T_0=\rho_0^{-1}(\D_0)$ and $\T_0\cdot
a=\rho_0^{-1}(\D_0\cdot a)$ correspond to intersections of the set
$\wt\M$ with the half-spaces $\Im(w)\ge 0$ and $\Im(w)\le 0$,
respectively. The Hubbard tripod of $J_{w_0}$ is oriented
counterclockwise if $\Im(w_0)>0$ and clockwise if $\Im(w_0)<0$.

\section{Miscellany}

\subsection{A metric on the fibers of the Julia set}
\label{ss:metric}

The folding transformation $F$ used in
Proposition~\ref{pr:fiberinvlim} stretches the legs of tripods in
a non-uniform way: two are stretched twice, one is not stretched
at all. A modified transformation might be more natural. Let us
show that it is essentially equivalent to the transformation $F$
and use this fact to construct a metric on the dendrite slices of
the Julia set of $f$.

Let $T$ be a tripod with feet labeled by $Z_\alpha, Z_\beta$, and
$Z_\gamma$. Denote by $F_1(T)$ the tripod obtained by folding the
path $Z_\alpha Z_\gamma$ in two and labeling the common image of
$Z_\alpha$ and $Z_\beta$ by $Z_\beta$, the image of $Z_\beta$ by
$Z_\gamma$, and the image of the midpoint of $Z_\alpha Z_\gamma$
by $Z_\alpha$.

If $(x, y, z)$ are lengths of the legs $Z_\alpha O$, $Z_\beta O$,
and $Z_\gamma O$, respectively, then the corresponding lengths of
legs of $F_1(T)$ are
\[\left(\frac{|x-z|}2,\quad\min(x, z),\quad y\right).\]

\begin{defi}
Denote $k(T)=\onep$ if $x\le z$, and $k(T)=\twop$ if $x>z$. Then
the \emph{$F_1$-itinerary} of a tripod $T$ is the sequence $k_0,
k_1, \ldots$, where $k_n=k(F^n_1(T))$.
\end{defi}

The fibers of the map $\rho:\M\arr\mathcal{S}$ are
\emph{normalized} tripods, i.e., tripods of mass (sum of lengths
of the legs) equal to one. Consider therefore the transformation
$\wt F$ of tripods equal to $F_1$ followed by division of all
distances in $F_1(T)$ by the mass of $F_1(T)$.

We will also denote by $F_1$ and $\wt F$ the corresponding
transformations of triples of lengths of legs. Let, as before, $F$
be the action of the map $p_0\circ\theta_0$ on the fibers of
$\rho_0$, also seen as a map on the triples of lengths. We have
\[F(x, y, z)=\left(|x-z|,\quad 2\min(x, z),\quad y\right).\]

The map $\wt F$ is given by
\begin{equation}\wt F(x, y, z)=\left(\frac{|x-z|}{1+y},\quad
\frac{2\min(x, z)}{1+y},\quad \frac{2y}{1+y}\right).
\end{equation}

\begin{lemma}
\label{l:seconditeration} The second iteration of $\wt F$ is
uniformly expanding on every tripod by a factor not less than $2$.
\end{lemma}

\begin{proof}
The second iteration of $\wt F$ is equal to $F^2_1$ followed by
dividing all distances by the mass of the obtained tripod. The
lengths of the legs after one folding are either $((z-x)/2, x,
y)$, or $((x-z)/2, z, y)$. Consequently, the triple of lengths of
the legs of the tripod after two folding belongs to the list
\[
\left(\frac{y-\frac{z-x}2}2, \frac{z-x}2, x\right)=
\left(\frac{x+2y-z}4, \frac{z-x}2, x\right),
\]
\[
\left(\frac{y-\frac{x-z}2}2, \frac{x-z}2, z\right)=
\left(\frac{-x+2y+z}4, \frac{x-z}2, z\right)
\]
\[
\left(\frac{\frac{z-x}2-y}2, y, x\right)= \left(\frac{-x-2y+z}4,
y, x\right),\]
\[
\left(\frac{\frac{x-z}2-y}2, y, z\right)= \left(\frac{x-2y-z}4, y,
z\right).
\]
The mass of the folded tripod is $\frac{3x+2y+z}4$ if $x\le z$,
and $\frac{x+2y+3z}4$ if $x\ge z$. In each case the number is not
more than $\frac{2x+2y+2z}4=1/2$.
\end{proof}

The branches of ${\wt F}^{-1}$ are the functions $\wt\Phi_\onep$,
$\wt\Phi_\twop$ acting on the triples of lengths of legs by
\begin{equation}
\label{eq:Phi1}\wt\Phi_\onep(x, y,
z)=\left(\frac{y}{1+x+y},\quad\frac z{1+x+y},\quad\frac
{2x+y}{1+x+y}\right)
\end{equation}
and
\begin{equation}
\label{eq:Phi2}\wt\Phi_\twop(x, y,
z)=\left(\frac{2x+y}{1+x+y},\quad\frac z{1+x+y},\quad\frac
y{1+x+y}\right).
\end{equation}

\begin{proposition}
\label{pr:conjugatef0} The maps $F$ and $\wt F$ acting on the
triangle $\Delta=\{(x, y, z)\;:\;0\le x,\;0\le y,\;0\le
z,\;x+y+z=1\}$ are topologically conjugate.
\end{proposition}

\begin{proof}
Both maps fold the triangle $R$ along the bisectrix of the angle
with vertex $(0, 0, 1)$, and act on the vertices by the rule
\[(1, 0, 0)\mapsto (1, 0, 0),\quad (0, 1, 0)\mapsto (0, 0,
1),\quad (0, 0, 1)\mapsto (1, 0, 0).\] Both maps are projective
(the map $F$ is affine), hence they map straight lines to straight
lines.

It is sufficient hence to prove that for every sequence $i_1, i_2,
i_3, \ldots$ of symbols $\onep$ and $\twop$ the diameters of the
nested triangles
\[\Delta_{i_1i_2\ldots i_n}=
\wt\Phi_{i_1}\circ\wt\Phi_{i_2}\circ\cdots\circ\wt\Phi_{i_n}(\Delta)\]
exponentially converge to zero as $n$ grows.

The vertices of the triangle $\Delta_{i_1i_2\ldots i_n}$ are the
normalized (i.e., divided by the sum of their coordinates) columns
of the matrix
\[B_{i_1i_2\ldots i_n}=B_{i_1}B_{i_2}\cdots B_{i_n},\]
where
\[B_\onep=\left(\begin{array}{ccc} 0 & 1 & 0\\
0 & 0 & 1\\ 2 & 1 & 0\end{array}\right),\quad
B_\twop=\left(\begin{array}{ccc} 2 & 1 & 0\\ 0 & 0 & 1\\ 0 & 1 &
0\end{array}\right).\]

Let $\vec a_n, \vec b_n, \vec c_n$ be first, second, and third
columns of the matrix $B_{i_1i_2\ldots i_n}$, respectively and let
$a_n, b_n, c_n$ be the sums of the coordinates of the vectors
$\vec a_n, \vec b_n, \vec c_n$.

\begin{lemma}
\label{l:quotients} For every $n\ge 1$ we have
\[\frac 13<\frac{a_n}{c_n}<3,\quad \frac 13<\frac{a_n}{b_n}<\frac 32,\quad
\frac 13<\frac{c_n}{b_n}<\frac 32.\]
\end{lemma}

\begin{proof}
Let us prove the lemma by induction. We have $(a_1, b_1, c_1)=(1,
1, 1)$, which satisfies the conditions of the lemma.

The sequence $(a_{n+1}, b_{n+1}, c_{n+1})$ is equal to one of the
sequences
\[(2c_n,
a_n+c_n, b_n),\quad (2a_n, a_n+c_n, b_n).\] Since $a_n$ and $c_n$
play a symmetric role in our lemma, it is sufficient to check only
the second case $a_{n+1}=2a_n$, $b_{n+1}=a_n+c_n$, $c_{n+1}=b_n$.

Then
\[\frac{a_{n+1}}{c_{n+1}}=\frac{2a_n}{b_n}\in\left(\frac 23, 3\right)\subset
\left(\frac 13, 3\right).\]

We also have
\begin{gather*}\frac{a_{n+1}}{b_{n+1}}=\frac{2a_n}{a_n+c_n}<\frac{2a_n}{a_n+a_n/3}
=\frac 32,\\
\frac{a_{n+1}}{b_{n+1}}=\frac{2a_n}{a_n+c_n}>\frac{2a_n}{a_n+3a_n}=
\frac 12>\frac 13,\end{gather*} and
\begin{gather*}\frac{c_{n+1}}{b_{n+1}}=\frac{b_n}{a_n+c_n}<\frac{b_n}{b_n/3+b_n/3}=\frac 32,\\
\frac{c_{n+1}}{b_{n+1}}=\frac{b_n}{a_n+c_n}>\frac{b_n}{3b_n/2+3b_n/2}=
\frac 13,
\end{gather*}
which finishes the inductive argument.
\end{proof}

Denote by $\vec\alpha_n=\frac{\vec a_n}{a_n}$,
$\vec\beta_n=\frac{\vec b_n}{b_n}$ and $\vec\gamma_n=\frac{\vec
c_n}{c_n}$ the vertices of the triangle $\Delta_{i_1i_2\ldots
i_n}$.

We have that either
\[\vec a_{n+1}=2\vec a_n,\quad\vec b_{n+1}=\vec a_n+\vec c_n,\quad
\vec c_{n+1}=\vec b_n,\] or
\[\vec a_{n+1}=2\vec c_n,\quad\vec b_{n+1}=\vec a_n+\vec c_n,\quad
\vec c_{n+1}=\vec b_n.\]

Consequently, the vertices of the triangle $\Delta_{i_1i_2\ldots
i_n1}$ are equal to
\[\vec\alpha_{n+1}=\vec\alpha_n,\quad
\vec\beta_{n+1}=\frac{a_n\vec\alpha_n+c_n\vec\gamma_n}{a_n+c_n},
\quad\vec\gamma_{n+1}=\vec\beta_n.\] The vertices of the triangle
$\Delta_{i_1i_2\ldots i_n2}$ are equal to
\[\vec\alpha_{n+1}=\vec\gamma_n,\quad
\vec\beta_{n+1}=\frac{a_n\vec\alpha_n+c_n\vec\gamma_n}{a_n+c_n},
\quad\vec\gamma_{n+1}=\vec\beta_n.\]

Hence, the triangle $\Delta_{i_1i_2\ldots i_n}$ is divided into
the triangles $\Delta_{i_1i_2\ldots i_n\onep}$ and
$\Delta_{i_1i_2\ldots i_n\twop}$ by the line connecting the vertex
$\vec\beta_n$ with the point
$\frac{a_n\vec\alpha_n+c_n\vec\gamma_n}{a_n+c_n}$ on the opposite
side of the triangle. Note that by Lemma~\ref{l:quotients}
\begin{multline*}\left|\vec\alpha_n-\frac{a_n\vec\alpha_n+c_n\vec\gamma_n}{a_n+c_n}\right|=
\frac{c_n}{a_n+c_n}\left|\vec\alpha_n-\vec\gamma_n\right|<\\
<\frac{c_n}{c_n/3+c_n}\left|\vec\alpha_n-\vec\gamma_n\right|=\frac
34\left|\vec\alpha_n-\vec\gamma_n\right|.\end{multline*}
Similarly,
\[\left|\vec\gamma_n-\frac{a_n\vec\alpha_n+c_n\vec\gamma_n}{a_n+c_n}\right|<
\frac 34\left|\vec\alpha_n-\vec\gamma_n\right|.\] Hence the line
dividing $\Delta_{i_1i_2\ldots i_n}$ into the triangles
$\Delta_{i_1i_2\ldots i_n\onep}$ and $\Delta_{i_1i_2\ldots
i_n\twop}$ divides the side $[\vec\alpha_n, \vec\gamma_n]$ in
proportion between $1:3$ and $3:1$.

Let us prove that diameter of the triangles $\Delta_{i_1\ldots
i_{n+1}i_{n+2}i_{n+3}}$ is less than $\frac 34$ of diameter of the
triangle $\Delta_{i_1i_2\ldots i_n}$. Figure~\ref{fig:triangles}
shows how the triangle $\Delta_{i_1i_2\ldots i_n}=\triangle ABC$
is subdivided into the 8 triangles $\Delta_{i_1i_2\ldots
i_ni_{n+1}i_{n+2}i_{n+3}}$ for different $i_{n+1}, i_{n+2},
i_{n+3}$ (ignore the dashed lines and shading for a while).
\begin{figure}
  \centering
  \includegraphics{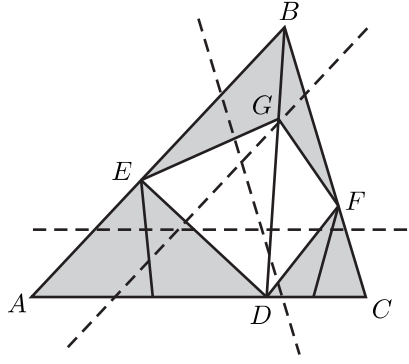}\\
  \caption{Triangles $\Delta_{i_1\ldots i_{n+3}}$}\label{fig:triangles}
\end{figure}

By the proved above, the points $D$, $E$ and $F$ divide the sides
of the triangle $\triangle ABC$ in proportion between $1:3$ to
$3:1$. Similarly the point $G$ divides the segment $BD$ in a
proportion belonging to the same range. It follows that the shaded
triangles $\triangle AED$, $\triangle CDF$ and quadrilateral
$BFGE$ are subsets of the images of the triangle $\triangle ABC$
under the homotheties with coefficient $\frac 34$ and centers in
the points $A$, $C$, and $B$, respectively. The images of the
lines $BC$, $AB$, and $AC$ under these homotheties are shown as
dashed lines on Figure~\ref{fig:triangles}.

It follows that diameters of the shaded triangles and
quadrilateral are less that three quarters of the diameter of
$\triangle ABC$. The diagonal $GD$ of the quadrilateral $EDFG$ is
less than $\frac 34$ times the length of the segment $BD$. The
other diagonal and the sides of $EDFG$ belong to one of the shaded
triangles or quadrilateral, hence their lengths are also less than
three quarters of the diameter of $\triangle ABC$. Consequently,
diameter of the quadrilateral $EDFG$ is also less than $\frac 34$
times the diameter of $\triangle ABC$.

Each of the triangles $\Delta_{i_1\ldots i_{n+1}i_{n+2}i_{n+3}}$
is a subset of one of the triangles and quadrilaterals for which
we have proved that their diameter is less than $\frac 34$ times
the diameter of $\triangle ABC=\Delta_{i_1i_2\ldots i_n}$, which
finishes the proof.
\end{proof}

\begin{defi}
Let $(k_1, k_2, \ldots)$ be the $F_1$-itinerary of the tripod $T$.
Denote by $T^{(n)}$ the tree
$\Phi_{k_1}\circ\Phi_{k_2}\circ\cdots\circ\Phi_{k_n}(F^n_1(T))$.
\end{defi}

It follows from the description of the transformations $F_1$ and
$\Phi_i$ that the Hubbard tripod of $T^{(n)}$ is isometric to the
tripod $T$. In particular, $T$ is naturally identified with the
Hubbard tripod of $T^{(1)}=\Phi_{k_1}(F_1(T))$ (see
Figure~\ref{fig:T1}).

\begin{figure}
\centering\includegraphics{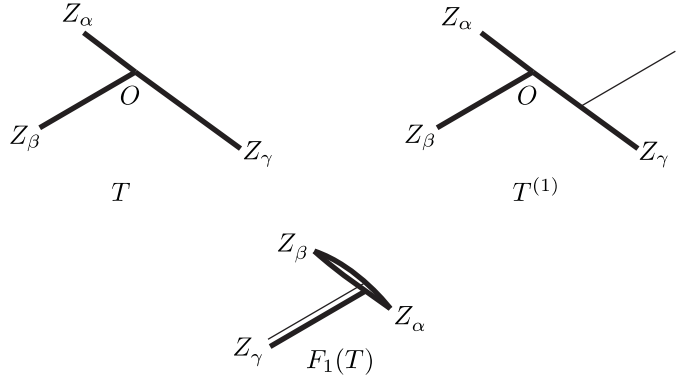}\\
  \caption{Extension of $F_1$ to a covering}\label{fig:T1}
\end{figure}

The tree $T^{(1)}$ is obtained from $T$ by attaching a copy of the
segment $OZ_\beta$ to the point of the geodesic $Z_\alpha
Z_\gamma$ symmetric to $O$ with respect to the center of $Z_\alpha
Z_\gamma$, i.e., by making the tree symmetric with respect to the
midpoint of $Z_\alpha Z_\gamma$. The covering $T^{(1)}\arr F_1(T)$
folds then the segment $Z_\alpha Z_\gamma$ twice and identifies
the leg $OZ_\beta$ with the added copy of it.

The midpoint of $Z_\alpha Z_\gamma$ divides $T^{(1)}$ into two
copies of $F_1(T)$ such that the covering $T^{(1)}\arr F_1(T)$ is
an isometry of these copies with $F_1(T)$.

By the description of the transformations $\Phi_i$, the tree
$T^{(n+1)}$ is obtained then as two copies of $F_1(T)^{(n)}$
pasted together along the copies of the vertex $Z_\alpha$ of
$F_1(T)$, which will be the midpoint of $Z_\alpha Z_\gamma$ in
$T^{(n+1)}$. This shows by induction that the tree $T^{(n)}$ is
naturally identified with a subtree of $T^{(n+1)}$ and that the
map $F_1:T\arr F_1(T)$ is extended to a branched two-to-one
coverings $T^{(n+1)}\arr F_1(T)^{(n)}$.

We get then for every $n$ a branched covering $T^{(n)}\arr
F^n_1(T)$ of degree $2^n$ equal to the composition of the degree
two coverings
\[T^{(n)}\arr F_1(T)^{(n-1)}\arr F_1^2(T)^{(n-2)}\arr\cdots\arr F_1^n(T).\]

Consequently, $T^{(n)}$ consists of $2^n$ copies of the tripod
$F^n_1(T)$ connected to each other along their feet in some way.
Then the tree $T^{(n+1)}$ is obtained from $T^{(n)}$ by making
each of these copies of $F^n_1(T)$ symmetric with respect to the
midpoint of the copy of the segment $Z_\alpha Z_\gamma$ of
$F^n_1(T)$.

See on Figure~\ref{fig:growth} the sequence $T^{(n)}$, $n=0,
\ldots, 5$ for a concrete tripod $T$. On the first three trees
$T^{(k)}$ the labels come from the labelling of the copies of the
tripod $F^k_1(T)$, i.e., the labels of the images of the
corresponding points under the covering $F^k_1:T^{(k)}\arr
F^k_1(T)$. In each of the trees $T^{(n)}$ the sub-tree $T^{(n-1)}$
is shown by thicker lines.

\begin{figure}
  \centering
  \includegraphics{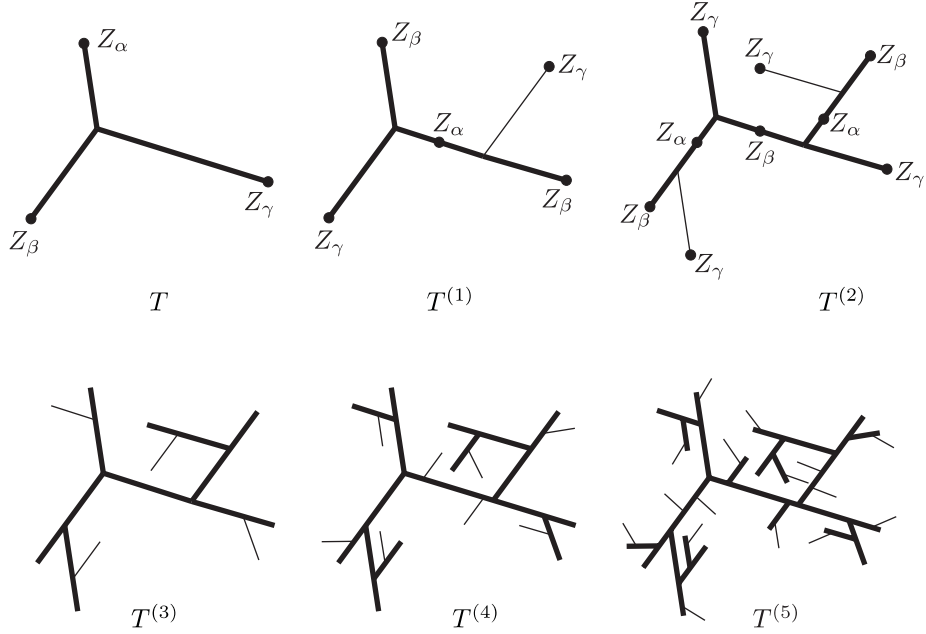}\\
  \caption{Growing the trees $T^{(n)}$}\label{fig:growth}
\end{figure}

\begin{proposition}
\label{prop:disttn} Let $d_n$ be the Hausdorff distance between
$T^{(n)}$ and $T^{(n+1)}$ inside $T^{(n+1)}$, i.e., maximum over
$\xi\in T^{(n+1)}$ of the distance of $\xi$ to $T^{(n)}$. Then
$d_n<\frac d{2^{(n-1)/2}}$, where $d$ is mass of $T$.
\end{proposition}

\begin{proof}
The tree $T^{(n+1)}$ is obtained from $T^{(n)}$ by making each of
the $2^n$ copies of $F^n_1(T)$ in $T^{(n)}$ symmetric with respect
to the midpoint of $Z_\alpha Z_\gamma$, i.e., by attaching a copy
of the segment $OZ_\beta$ of $F^n_1(T)$. It follows that the
Hausdorff distance between $T^{(n)}$ and $T^{(n+1)}$ is equal to
the length of the segment $OZ_\beta$ in $F^n_1(T)$. But it is not
more than mass of $F^n_1(T)$, which is not more than $\frac
d{2^{(n-1)/2}}$, by Lemma~\ref{l:seconditeration}.
\end{proof}

\begin{defi}
Let $T$ be a tripod. Then the \emph{limit dendrite} of $T$,
denoted $T^{(\infty)}$ is completion of the inductive limit of the
metric trees $T^{(n)}$ as $n$ goes to infinity.
\end{defi}

Denote by $\wt\iota_n:T^{(n+1)}\arr T^{(n)}$ the projection of
$T^{(n+1)}$ onto its subtree $T^{(n)}$.

\begin{proposition}
\label{pr:dendroids} The limit dendrite $T^{(\infty)}$ is
homeomorphic to the inverse limit of the trees $T^{(n)}$ with
respect to the maps $\wt\iota_n$.
\end{proposition}

\begin{proof}
It follows directly from Proposition~\ref{prop:disttn} that the
inductive limit of the spaces $T^{(n)}$ is completely bounded
(i.e., has a finite $\epsilon$-net for every positive $\epsilon$).
It follows then that the completion $\D(T)$ is compact.

Note that the projections $\wt\iota_n$ are idempotent maps (i.e.,
$\wt\iota_n(x)=x$ for all $x\in T^{(n)}$). By
Proposition~\ref{prop:disttn}, the map $\wt\iota_n$ moves points
not more than by $\frac d{2^{(n-1)/2}}$.

Let $x_n\in T^{(n)}$ be a sequence representing a point $x$ of the
inverse limit, i.e., such that $\wt\iota_n(x_{n+1})=x_n$. It
follows from Proposition~\ref{prop:disttn} that the corresponding
sequence $x_n\in T^{(\infty)}$ is fundamental, hence it converges
to a point $\delta(x)$ of $T^{(\infty)}$.

Let us show that the map $\delta$ from the inverse limit to
$T^{(\infty)}$ is a homeomorphism. It is continuous by
Proposition~\ref{prop:disttn}. If $x_n$ and $y_n$ are sequences
representing different points of the inverse limit, then $d(x_k,
y_k)>0$ for some $k$, which implies that $d(x_n, y_m)\ge d(x_k,
y_k)$ for all $n, m\ge k$, by the elementary properties of trees
and projections $\wt\iota_n$. It follows that the map $\delta$ is
injective.

Let $x\in T^{(\infty)}$ be arbitrary. It is a limit of a
fundamental sequence $x_n$ in the inductive limit of the trees
$T^{(n)}$. Passing to a subsequence, and then repeating the
entries of the subsequence, if necessary, we may assume that
$x_n\in T^{(n)}$. Consider for every $k$ the sequence $x_{k,
n}=\wt\iota_k\circ\cdots\circ\wt\iota_{n-1}(x_n)$ for $n>k$. As
above, we have $d(x_n, x_m)\ge d(x_{k, n}, x_{k, m})$ for all $m,
n>k$, which implies that the sequence $x_{k, n}$ converges to a
point $y_k$ in $T^{(k)}$. We have $\wt\iota(y_{k+1})=y_k$. Limit
of the sequence $y_n$ is equal to the limit of $x_n$, since
$\wt\iota_n$ moves points not more than by $\frac d{2^{(n-1)/2}}$.
Consequently, $\delta$ is onto, and hence a homeomorphism (since
the inverse limit and $T^{(\infty)}$ are compact).
\end{proof}

\begin{theorem}
Denote by $\psi:\Delta\arr\Delta$ the homeomorphism conjugating
the maps $F$ and $\wt F$, i.e., such a homeomorphism that
$\psi\circ F=\wt F\circ\psi$.

If the image of $\zeta\in\mathcal{S}_\infty$ in $\mathcal{S}_0$ is
$\zeta_0$ and $(x_1, x_2, x_3)$ are the lengths of legs of the
tripod $\rho^{-1}(\zeta_0)$, then denote by $T_\zeta$ the tripod
with the lengths of legs $\psi(x_1, x_2, x_3)$.

Then there exists a family of homeomorphisms $\tau_\zeta:
T_\zeta^{(\infty)}\arr\rho_\infty^{-1}(\zeta)$ conjugating the
maps
$p_\infty:\rho_\infty^{-1}(\zeta)\arr\rho_\infty^{-1}(q_\infty(\zeta))$
with the coverings $\wt F:T_\zeta^{(\infty)}\arr
T_{q_\infty(\zeta)}^{(\infty)}$.
\end{theorem}

Here $\wt F$ is covering of limit dendrites obtained as the limit
of the coverings $T^{(n)}\arr{\wt F}_1(T)^{(n-1)}$.

\begin{proof}
Let $(\zeta_0, \zeta_1=\lambda^{-1}(\zeta_0),
\zeta_2=\lambda_1^{-1}(\zeta_1), \ldots)$ be a point of
$\mathcal{S}_\infty$. Let $(x_1, x_2, x_3)$ be the lengths of legs
of the tripod $\rho^{-1}(\zeta_0)$. Let $(k_1, k_2, \ldots)$ be
the itinerary of the tripod $\rho^{-1}(\zeta_0)$ (with respect to
the folding map $F$).

Then $(k_1, k_2, \ldots)$ is the $F_1$-itinerary of the tripod
$T_\zeta$ with the lengths of legs $\psi(x_1, x_2, x_3)$. The
triple of lengths of legs of the tripod $F_1^n(T_\zeta)$ is
$F_1^n(\psi(x_1, x_2, x_3))$. If we divide the triple
$F_1^n(\psi(x_1, x_2, x_3))$ by the sum of its entries, then we
get ${\wt F}^n(\psi(x_1, x_2, x_3))=\psi(F^n(x_1, x_2, x_3))$.

It follows that the tripod $T_n$ with the lengths of legs
$F^n(x_1, x_2, x_3)$ is homeomorphic to the tripod
$F_1^n(T_\zeta)$. Let us fix some homeomorphism
$F_1^n(T_\zeta)\arr T_n$. Applying it to the copies of
$F_1^n(T_\zeta)$ in
$T_\zeta^{(n)}=\Phi_{k_1}\circ\Phi_{k_2}\circ\cdots\circ\Phi_{k_n}(F_1^n(T_\zeta))$,
we get a homeomorphism
\[\tau_n:T_\zeta^{(n)}\arr\rho_n^{-1}(\zeta_n),\]
since the tree $\rho_n^{-1}(\zeta_n)$ is homeomorphic to
$\Phi_{k_1}\circ\Phi_{k_2}\circ\cdots\circ\Phi_{k_n}(T_n)$ by
Proposition~\ref{pr:fiberinvlim}.

It follows from the description of the action of
$\iota_n:\rho_{n+1}^{-1}(\zeta_{n+1})\arr\rho_n^{-1}(\zeta_n)$
given in Proposition~\ref{pr:fiberinvlim} that the homeomorphisms
$\tau_n$ conjugate $\iota_n$ with the projection
$\wt\iota_n:T_\zeta^{(n+1)}\arr T_\zeta^{(n)}$. Therefore, the
limit of the homeomorphisms $\tau_n$ is a homeomorphism
$\tau_\zeta:T_\zeta^{(\infty)}\arr\rho_\infty^{-1}(\zeta)$.

The statement about the map $p_n$ follows now from
Proposition~\ref{pr:Phi12}.
\end{proof}

We get in this way a family of length metrics on the slices of the
Julia set of $f$ such that for every $w_0$ there exists a constant
$c\ge 1$ such that the map $f:J_{w_0}\arr J_{(1-2/w_0)^2}$
multiplies every curve of $J_{w_0}$ by $c$.

\subsection{Triangle-filling}

The inductive unfolding procedures $\Phi_i$, described in
Subsection~\ref{ss:unfolding} can be realized in a way leading to
a family of surjections from the dendrites $T^{(\infty)}$ (i.e.,
from the slices of the Julia set of $f$) onto an isosceles right
triangle.

Draw a tripod $T=Z_\alpha Z_\beta Z_\gamma$ inside an isosceles
right triangle in such a way that its foot $Z_\beta$ belongs to
the hypothenuse and the feet $Z_\alpha$ and $Z_\gamma$ are
symmetric points on the catheti, see top of
Figure~\ref{fig:righttriangle}.

\begin{figure}
\centering
  \includegraphics{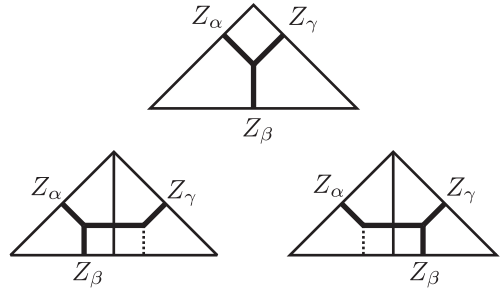}\\
  \caption{Unfolding the tripods inside of a triangle}\label{fig:righttriangle}
\end{figure}

We can now paste two copies of the tripod $T$ together with the
circumscribed triangle in such a way that the copies are
reflections of each other and the union of the two triangles is
again an isosceles right triangle.

If we label three vertices of tree $\Phi_i(T)$, accordingly to the
unfolding rule, the vertices $Z_\alpha$ and $Z_\gamma$ will be
again symmetric points on the catheti and the vertex $Z_\beta$
will belong to the hypothenuse. See bottom of
Figure~\ref{fig:righttriangle}, where both case I (on the left
hand side) and case II (on the right hand side) are shown.

We can iterate now the process (choosing one of the two cases on
each step). For better visualization (in order the vertices of the
trees $\Phi_{i_1}\circ\cdots\circ\Phi_{i_n}(T)$ not to collide),
we may on each step shorten (or delete) the edge containing the
copy of the vertex $Z_\gamma$ of
$\Phi_{i_2}\circ\cdots\circ\Phi_{i_n}(T)$ that did not become the
vertex $Z_\beta$ of $\Phi_{i_1}\circ\cdots\circ\Phi_{i_n}(T)$. In
this case only three vertices $Z_\alpha, Z_\beta, Z_\gamma$ of
$\Phi_{i_1}\circ\cdots\circ\Phi_{i_n}(T)$ will belong to the
perimeter of the triangle.

See, for instance Figure~\ref{fig:tntr}, where the result of
application of this procedure ten times is shown. Here always the
transformation $\Phi_\onep$ is applied and we have deleted the
edges containing the unlabeled copies of $Z_\gamma$. We get in
this way the graph of the action of $\img{z^2+i}$ on the tenth
level of the tree, approximating the Julia set of $z^2+i$.

\begin{figure}
\centering
  \includegraphics{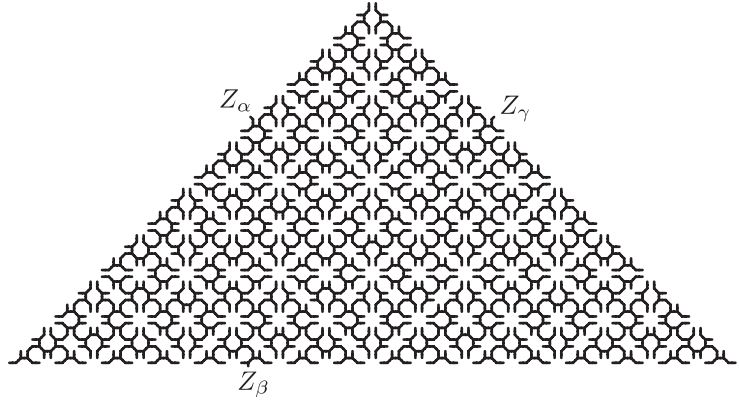}\\
  \caption{A Schreier graph of $\img{z^2+i}$}\label{fig:tntr}
\end{figure}

If we apply $\Phi_\twop$ each time, then we get in the limit the
well known \emph{Sierpi\'nski plane filling curve},
see~\cite{sierpinski:curve,mandelbrot}.

Figure~\ref{fig:schreiergraphs} shows all possible graphs obtained
in this way by applying seven transformations $\Phi_i$.

\begin{figure}
\includegraphics{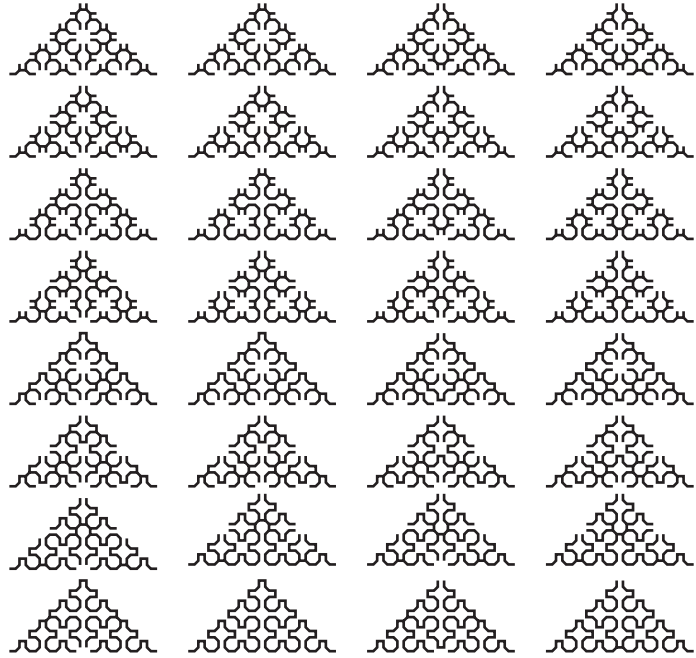}
\caption{Schreier graphs of $\G_0$}\label{fig:schreiergraphs}
\end{figure}

\subsection{External rays}
A standard tool in dynamics of complex polynomials are
\emph{external rays} (see~\cite{DH:orsayI,DH:orsayII}). If $K$ is
the \emph{filled-in Julia set} of a monic polynomial $f$ ($K$ is
equal to the Julia set, if it is a dendrite), then there exists a
bi-holomorphic isomorphism of the complement of the unit disc
$\C\setminus\mathbb{D}$ with $\C\setminus K$, conjugating the
action of $z\mapsto z^{\deg f}$ on $\C\setminus\mathbb{D}$ with
the action of $f$ on $\C\setminus K$. The images $R_\theta$ of the
rays $\{r\exp{i\theta}\;:\;r>1\}\subset\C\setminus\mathbb{D}$ in
$\C\setminus K$ under this isomorphism are called the
\emph{external rays} to the Julia set of $f$.

Since our rational function $f(z,
w)=\left(\left(1-\frac{2z}{w}\right)^2,\quad
\left(1-\frac{2}{w}\right)^2\right)$ is a polynomial on the first
coordinate, we also may define external rays to the Julia set of
$f$.

Namely, for every $w_0\in\CS$ let $J_{w_0}$ be the intersection of
the Julia set of $f$ with the line $w=w_0$. We know that $J_{w_0}$
is a dendrite. The complement of $J_{w_0}$ in the line $w=w_0$
(without the superattracting point $[1:0:0]$) is an annulus,
bi-holomorphically isomorphic to the complement of a disc in $\C$.
We define the rays $R_{\theta, w_0}$ in the line $w=w_0$ in the
same way as for complex polynomials. The angle $\theta$ of a ray
is not canonically defined, since the bi-holomorphic isomorphism
is defined up to a rotation. The set of all external rays
$R_{\theta, w_0}$ is in a bijective correspondence with the
boundary of a small ball around $[1:0:0]$, i.e., with a 3-sphere.
We introduce a topology on the set $\mathcal{R}$ of external rays
$R_{\theta, w}$ using this bijection. The function $f$ induces
then a self-map $f_\mathcal{R}:\mathcal{R}\arr\mathcal{R}$, since
the image of a ray $R_{\theta, w}$ under the action of $f$ will be
a ray $R_{\theta', \wh f}$. The aim of this section is to
understand the action of $f_\mathcal{R}$ on $\mathcal{R}$.

The action of $f$ on $\mathcal{R}$ is expanding (i.e.,
sub-hyperbolic), since it is a double covering of the circle of
rays corresponding to $(1-2/w)^2$ by the circle or rays
corresponding to $w$, and the rational function $(1-2/w)^2$ is
sub-hyperbolic. It follows that the dynamical system
$(\mathcal{R}, f_\mathcal{R})$ is topologically conjugate to the
limit dynamical system of the iterated monodromy group of
$f_\mathcal{R}:\mathcal{R}\arr\mathcal{R}$. This group is the
subgroup of $\img{f}$ generated by the loops not intersecting the
Julia set of $f$, i.e.,
\[\img{f_\mathcal{R}}=\langle\alpha\gamma\beta, s, t\rangle.\]

By Theorem~\ref{th:imgrecursion}, we have
\[\alpha\gamma\beta=\sigma(\beta, 1, 1, \beta)(\alpha, \gamma,
\alpha, \beta\gamma\beta)=\sigma(\beta\alpha, \gamma, \alpha,
\gamma\beta).\] Compose the wreath recursion with conjugation by
$(\alpha, \beta, \alpha, 1)$. We get then in the new wreath
recursion:
\[\alpha\gamma\beta=\sigma(\beta\cdot\beta\alpha\cdot\alpha,
\alpha\cdot\gamma\cdot\beta, \alpha\cdot\alpha,
\alpha\cdot\gamma\beta)=\sigma(1, \alpha\gamma\beta, 1,
\alpha\beta\gamma),\]
\[s=\pi(\alpha\alpha, \beta\beta, \alpha\alpha, \beta\beta)=\pi,\]
and
\[t=(r^\alpha, r^\beta, t^\alpha, t).\]

We have, by relations~\eqref{eq:conj1}--\eqref{eq:conj3}
\begin{multline*}r^\alpha=\alpha\beta\alpha\beta\gamma\beta
t^{-1}s^{-1}\alpha=\\ t^{-1}\alpha\cdot \gamma\beta\gamma\cdot
\alpha\cdot\gamma\beta\gamma\cdot\gamma\beta s^{-1}\alpha=
t^{-1}\alpha\gamma\beta\gamma\alpha\gamma s^{-1}\alpha=\\
t^{-1}s^{-1}\alpha\gamma\cdot\beta\cdot\alpha\gamma\alpha\cdot
\alpha\gamma\alpha=t^{-1}s^{-1}\alpha\gamma\beta,
\end{multline*}
and $r^\beta=r^{(\alpha\beta)\beta}=r^\alpha$.

Denote $\alpha\gamma\beta=\tau$. We see that the subgroup
generated by $\tau, s, t$ is self-similar and is given by the
recursion
\begin{eqnarray}\label{eq:tST1}
\tau &=& \sigma(1, \tau, 1, \tau),\\
\label{eq:tST2} s &=& \pi,\\
\label{eq:tST3} t &=& (t^{-1}s^{-1}\tau, t^{-1}s^{-1}\tau, t, t).
\end{eqnarray}

Note that $\tau$ commutes with $s$ and $t$. We also have $s^2=1$
and
\[(st)^2=(s^{-1}\tau, s^{-1}\tau, t^{-1}s^{-1}t\tau, t^{-1}s^{-1}t\tau),\] hence
\[(st)^4=(\tau^2, \tau^2, \tau^2, \tau^2)=\tau^4,\] and
\[t^4=(1, 1, t^4, t^4)=1.\]

Consider the elements $X=t^2s$ and $Y=tst$. We have
\[X=\pi(t^2, t^2, t^{-1}st^{-1}s\tau^2, t^{-1}st^{-1}s\tau^2)\]
and
\[Y=\pi(s\tau, s\tau, t^{-1}st\tau, t^{-1}st\tau).\]
Composing the wreath recursion with conjugation by $(s, s, 1, 1)$,
we get the recursion
\begin{eqnarray*}X &=& \pi(X, X, Y\tau^{-2}, Y\tau^{-2}),\\
Y &=& \pi(\tau, \tau, X^{-1}Y\tau, X^{-1}Y\tau),\\
\tau &=& \sigma(1, \tau, 1, \tau). \end{eqnarray*}

\begin{proposition}
\label{prop:heisenberg} The group generated by $X, Y, \tau$ is
isomorphic to the group
generated by the matrices $\varphi(X)=\left(\begin{array}{ccc}1 & 1 & 0\\
0 & 1 & 0\\ 0 & 0 & 1\end{array}\right),
\varphi(Y)=\left(\begin{array}{ccc}1 & 0 & 0\\ 0 & 1 & 1\\ 0 & 0 &
1\end{array}\right), \varphi(\tau)=\left(\begin{array}{ccc}1 & 0 & 1/4\\
0 & 1 & 0\\ 0 & 0 & 1\end{array}\right)$. In particular, the group
generated by $X, Y, \tau^4$ is isomorphic to the Heisenberg group
over integers.
\end{proposition}

\begin{proof} Note that we have
\[[X, Y]=st^2t^{-1}st^{-1}t^2stst=(st)^4=\tau^4.\]
It follows that the group $\langle X, Y, \tau\rangle$ is a
homomorphic image of the group given by the presentation
\[G_1=\langle X, Y, \tau\;:\;[X, Y]=\tau^4, [X, \tau]=[Y,
\tau]=1\rangle.\] Let us prove that it is isomorphic to this
group, which will finish the proof, since the group generated by
the matrices listed in the proposition is given by this
presentation.

Consider the homomorphism $\{\one, \two, \three,
\four\}^*\arr\{\onep, \twop\}^*$ of the free monoids given by
$\one\mapsto \onep, \two\mapsto\onep, \three\mapsto\twop,
\four\mapsto\twop$. It is easy to see that this homomorphism
agrees with the action of the group $G_0=\langle X, Y,
\tau\rangle$ on the tree $\{\one, \two, \three, \four\}$, so that
after taking projection we get an action of $G_0$ on the tree
$\{\onep, \twop\}^*$. This action is given by the wreath recursion
\[X=\sigma(X, Y),\quad Y=\sigma(1, X^{-1}Y),\quad \tau=1.\]
We get hence an epimorphism from $G_0$ to the given group acting
on the binary tree. It follows from the recursion
(see~\cite{neksid} and~\cite[Proposition~2.9.2]{nek:book}) that
$X$ and $Y$ generated in a free abelian group in this epimorphic
image. Consequently, the kernel of the natural epimorphism
$G_1\arr G_0$ is contained in $\langle\tau\rangle$. But the group
$\langle\tau\rangle$ acts faithfully on the tree $\{\one, \two,
\three, \four\}^*$, hence the kernel is trivial.
\end{proof}

It is checked directly that the map
\[\phi\left(\begin{array}{ccc} 1 & a & c\\ 0 & 1 & b\\ 0 & 0 &
1\end{array}\right)=\left(\begin{array}{ccc} 1 & \frac{a-b}2 &
\frac c2+\frac{a^2-b^2}8-\frac{ab}4-\frac a2\\
0 & 1 & \frac{a+b}2\\ 0 & 0 & 1\end{array}\right)\] is an
automorphism of the Heisenberg group.

The first level stabilizer of the group $\langle X, Y,
\tau\rangle$ is generated (since $X$ and $Y$ commute with $\tau$
and $YX=XY\tau^{-4}$) by $\tau^2, XY, Y^2$, which are equal to
\begin{eqnarray*}
\tau^2 &=& (\tau, \tau, \tau, \tau),\\
XY   &=& (Y\tau^{-1}, Y\tau^{-1}, Y\tau, Y\tau),\\
Y^2 &=& (X^{-1}Y\tau^2, X^{-1}Y\tau^2, X^{-1}Y\tau^2,
X^{-1}Y\tau^2).
\end{eqnarray*}

Direct computations show that the isomorphism $\varphi$ from
Proposition~\ref{prop:heisenberg} conjugates the virtual
endomorphism associated with the first coordinate of the wreath
recursion with the action of $\phi$ on the lattice generated by
$\varphi(X), \varphi(Y)$, and $\varphi(\tau)$. This gives another
proof of Proposition~\ref{prop:heisenberg}, but it also gives a
description of the limit space of the group $\langle X, Y,
\tau\rangle$.

\begin{proposition}
The limit $G$-space $\limg[\langle X, Y, \tau\rangle]$ of the
group $\langle X, Y, \tau\rangle$ is the real Heisenberg group
with the right action of $\langle X, Y, \tau\rangle$ given by the
isomorphism $\varphi$ from Proposition~\ref{prop:heisenberg}. The
limit space $\lims[\langle X, Y, \tau\rangle]$ is hence the
quotient of the Heisenberg group by the action of the lattice
$\langle\varphi(X), \varphi(Y), \varphi(\tau)\rangle$.
\end{proposition}

The next theorem gives us a description of the action of the map
$f$ on the space of the external rays, i.e., the action of $f$ on
the neighborhood of ``infinity'', i.e., of the superattracting
point $[1:0:0]$. It would be interesting to deduce
Theorem~\ref{th:externalHeisenberg} analytically.

\begin{theorem}
\label{th:externalHeisenberg} The subgroup $\langle X, Y,
\tau\rangle$ has index $4$ in $\langle s, t, \tau\rangle$. The
limit $G$-space $\limg[\img{f_{\mathcal{R}}}]$ of
$\img{f_{\mathcal{R}}}=\langle s, t, \tau\rangle$ is the real
Heisenberg group together with the action of the group given by
\begin{gather*}\left(\begin{array}{ccc}1 & x & z\\ 0 & 1 & y\\ 0 & 0 &
1\end{array}\right)\cdot s=\left(\begin{array}{ccc}1 & -x-1 & z+2y\\
0 & 1 & -y\\ 0 & 0 & 1\end{array}\right),\\
\left(\begin{array}{ccc}1 & x & z\\ 0 & 1 & y\\ 0 & 0 &
1\end{array}\right)\cdot t=\left(\begin{array}{ccc}1 & -y-1 & z-xy-x-1\\
0 & 1 & x+1\\ 0 & 0 & 1\end{array}\right),\end{gather*} and
\[\left(\begin{array}{ccc}1 & x & z\\ 0 & 1 & y\\ 0 & 0 &
1\end{array}\right)\cdot\tau=\left(\begin{array}{ccc}1 & x & z+1/4\\
0 & 1 & y\\ 0 & 0 & 1\end{array}\right).\]

The space of external rays $\mathcal{R}$ of the function $f$ is
homeomorphic to the quotient of the real Heisenberg group by the
described action. The action
$f_{\mathcal{R}}:\mathcal{R}\arr\mathcal{R}$ of $f$ on it is
induced by the automorphism
\[\phi^{-1}:\left(\begin{array}{ccc} 1 & a & c\\ 0 & 1 & b\\ 0 & 0 &
1\end{array}\right)\mapsto\left(\begin{array}{ccc} 1 & a+b & 2c+a+b-ab-\frac{a^2-b^2}2\\
0 & 1 & b-a\\ 0 & 0 & 1\end{array}\right)\] on the Heisenberg
group.
\end{theorem}

Note that the automorphism $\phi^{-1}$ coincides (up to an inner
automorphism of the real Heisneberg group) with the automorphism
used in~\cite{gelbrich:heisenberg}, see
also~\cite{petenek:scaleinv}.

\begin{proof}
The first level stabilizer of the group $\langle s, t,
\tau\rangle$ is generated by $\tau^2, t, sts$. They are equal (in
the wreath recursion~\eqref{eq:tST1}--\eqref{eq:tST3} conjugated
by $(s, s, 1, 1)$) to
\begin{eqnarray*}
\tau^2 &=& (\tau, \tau, \tau, \tau),\\
t      &=& (st^{-1}\tau, st^{-1}\tau, t, t),\\
sts    &=& (sts, sts, t^{-1}s\tau, t^{-1}s\tau).
\end{eqnarray*}

The following equalities are checked directly:
\[\left(\begin{array}{ccc}1 & x & z\\ 0 & 1 & y\\ 0 & 0 &
1\end{array}\right)\cdot sts=\left(\begin{array}{ccc}1 & -y & z-xy-x+y\\
0 & 1 & x\\ 0 & 0 & 1\end{array}\right),\] and
\[\left(\begin{array}{ccc}1 & x & z\\ 0 & 1 & y\\ 0 & 0 &
1\end{array}\right)\cdot st^{-1}\tau=
\left(\begin{array}{ccc}1 & -1-y & z-xy-x+y+1/4\\ 0 & 1 & x\\
0 & 0 & 1\end{array}\right).\]

We have
\begin{multline*}
\phi\left(\left(\begin{array}{ccc}1 & x & z\\ 0 & 1 & y\\ 0 & 0 &
1\end{array}\right)\cdot sts\right)= \left(\begin{array}{ccc}1 &
\frac{-x-y}2 & \frac{z-xy-x+y}2+\frac{y^2-x^2}8+
\frac{xy}4+\frac y2\\
0 & 1 & \frac{x-y}2\\ 0 & 0 & 1\end{array}\right)=\\
\left(\begin{array}{ccc}1 & \frac{-x-y}2 & \frac z2+\frac{y^2-x^2}8-\frac{xy}4-\frac x2+y\\
0 & 1 & \frac{x-y}2\\ 0 & 0 & 1\end{array}\right),
\end{multline*}
and
\begin{multline*}
\phi\left(\begin{array}{ccc}1 & x & z\\ 0 & 1 & y\\ 0 & 0 &
1\end{array}\right)\cdot sts=\left(\begin{array}{ccc}1 &
\frac{x-y}2 & \frac z2+\frac{x^2-y^2}8-\frac{xy}4-\frac x2\\ 0 & 1 & \frac{x+y}2\\
0 & 0 & 1\end{array}\right)\cdot sts=\\
\left(\begin{array}{ccc}1 & \frac{-x-y}2 & \frac z2+\frac{y^2-x^2}8-\frac{xy}4-\frac x2+y\\
0 & 1 & \frac{x-y}2\\ 0 & 0 & 1\end{array}\right).
\end{multline*}

\begin{multline*}
\phi\left(\left(\begin{array}{ccc}1 & x & z\\ 0 & 1 & y\\ 0 & 0 &
1\end{array}\right)\cdot t\right)=\\
\left(\begin{array}{ccc}1 &
\frac{-x-y}2-1 & \frac{z-xy-x-1}2+\frac{(1+y)^2-(x+1)^2}8+\frac{(x+1)(y+1)}4+\frac{1+y}2\\
0 & 1 & \frac{x-y}2\\
0 & 0 & 1\end{array}\right)=\\
\left(\begin{array}{ccc}1 &
\frac{-x-y}2-1 & \frac z2-\frac{xy}4+\frac{y^2-x^2}8-\frac x2+y+\frac 14\\
0 & 1 & \frac{x-y}2\\
0 & 0 & 1\end{array}\right),
\end{multline*}
while
\begin{multline*}
\phi\left(\begin{array}{ccc}1 & x & z\\ 0 & 1 & y\\ 0 & 0 &
1\end{array}\right)\cdot st^{-1}\tau=\left(\begin{array}{ccc}1 &
\frac{x-y}2 & \frac z2+\frac{x^2-y^2}8-\frac{xy}4-\frac x2\\ 0 & 1 & \frac{x+y}2\\
0 & 0 & 1\end{array}\right)\cdot st^{-1}\tau=\\
\left(\begin{array}{ccc}1 &
\frac{-x-y}2-1 & \frac z2-\frac{xy}4+\frac{y^2-x^2}8-\frac x2+y+\frac 14\\
0 & 1 & \frac{x-y}2\\
0 & 0 & 1\end{array}\right).
\end{multline*}

We see that in both cases $\phi$ conjugates the transformations to
the action of the first coordinate of the wreath recursion, which
finishes the proof.
\end{proof}

\subsection{Smooth cases} Let $A', B', C'$ be the vertices of the
triangle $\D\subset\mathcal{S}$, as in Subsection~\ref{ss:zpp}.
Note that $A'\cdot a=A'$, $B'\cdot a=B'$, and $C'\cdot a=C'$ in
$\mathcal{S}$, since these points belong to the boundary of $\D$,
which is identified in $\mathcal{S}$ with the boundary of $\D\cdot
a$. Note also that since $L(A'\otimes\two)=A'$, the point $A'$ is
invariant under the map
$q\circ\lambda^{-1}:\mathcal{S}\arr\mathcal{S}$. Consequently, the
fiber $\rho^{-1}(A')$ is invariant under the map
$p\circ\theta:\M\arr\M$.

The set $\rho^{-1}(A')$ is the tripod with the lengths of legs
$(1, 0, 0)$, so that the points $Z_\gamma$ and $Z_\beta$ coincide,
i.e., it is a segment with the ends $Z_\alpha$ and
$Z_\gamma=Z_\beta$. The map
$p\circ\theta:\rho^{-1}(A')\arr\rho^{-1}(A')$ doubles the
distances inside the segment $\rho^{-1}(A')$ and then folds it in
two. The common image of the endpoints $Z_\alpha$ and
$Z_\gamma=Z_\beta$ is $Z_\beta=\Z_\gamma$, the image the midpoint
is $Z_\alpha$ (and the image of $Z_\beta=Z_\gamma$ is
$Z_\gamma=Z_\beta)$, see Subsection~\ref{ss:foldingtripods}.
Hence, the map $p\circ\theta:\rho^{-1}(A')\arr\rho^{-1}(A')$ is
the tent map.

The itinerary of the tripod $\rho^{-1}(A')$ is therefore $(\twop,
\twop, \twop, \ldots)$. Consequently, the fibers $\rho_n^{-1}(A')$
and $\rho_\infty^{-1}(A')$ are also segments, and the maps
$\iota_n:\rho_{n+1}^{-1}(A')\arr\rho_n^{-1}(A')$ are
homeomorphisms.

The corresponding slice of the Julia set of $f$ is $J_{w_0}$ for
$w_0=1$, where the map $f:J_1\arr J_1$ is equal to the polynomial
$z\mapsto (1-2z)^2$. This is the Ulam-von~Neumann map (see, for
instance~\cite[Section~7, Example~2]{milnor}), its Julia set is
the interval $[0, 1]$, and it is topologically conjugate on the
Julia set with the tent map.

Since the extension of $F$ onto the limit dendrites are branched
coverings, we immediately get a set of limit dendrites that are
graphs.

\begin{figure}
  \centering\includegraphics{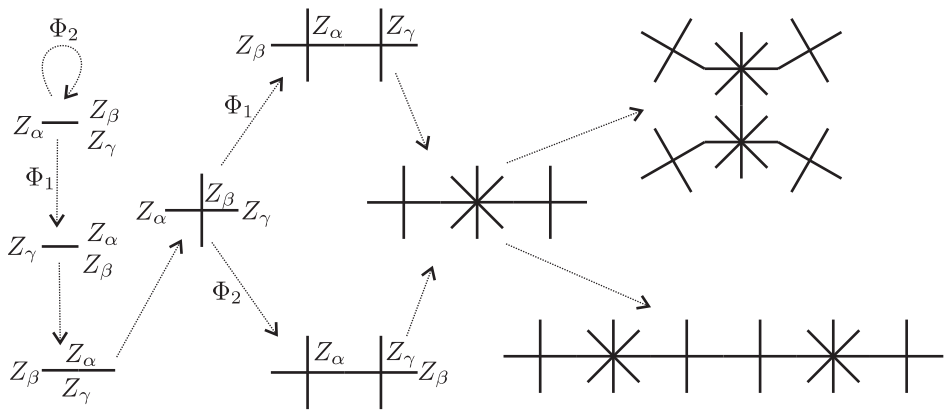}\\
  \caption{Dendrites $\rho_\infty^{-1}(\zeta)$ that are graphs}\label{fig:graphs}
\end{figure}

\begin{proposition}
The intersection $J_{w_0}$ of the Julia set of $f$ with the line
$w=w_0$ is homeomorphic to a finite tree if and only if $w_0$ is
mapped onto the fixed point $1$ by some iteration of $\wh
f(w)=(1-2/w)^2$.
\end{proposition}

\begin{proof}
Let $\zeta=(\zeta_0, \zeta_1, \ldots)\in\mathcal{S}_\infty$, where
$\zeta_{n+1}=\lambda_n^{-1}(\zeta_n)$. Let $(x_1, x_2, x_3)$ be
the lengths of legs of the tripod $\rho^{-1}(\zeta_0)$. We have to
show that $\rho_\infty^{-1}(\zeta)$ is a finite tree if and only
if the point $\zeta$ is mapped onto $A'$ by some iteration of the
map $q\circ\lambda^{-1}$, i.e., that $F^n(x_1, x_2, x_3)=(1, 0,
0)$ for some $n$.

By Proposition~\ref{pr:fiberinvlim}, the tree
$\rho_n^{-1}(\zeta_n)$ is obtained by pasting together $2^n$
copies of the tripod $T_n$ with the lengths of legs $F^n(x_1, x_2,
x_3)$. If $T_n$ has all legs of positive length, then
$\rho_n^{-1}(\zeta_n)$ has at least $2^n$ vertices of degree 3
(the copies of the common point of the legs). If $T_n$ has one leg
of length zero, but all three feet are different, then the Hubbard
tripod of
$\Phi_{k_{n-2}}\circ\Phi_{k_{n-1}}\circ\Phi_{k_n}(F^n(T))$ has all
legs of positive lengths for any triple of indices $k_n, k_{n-1},
k_{n-2}$, hence the tree $\rho_n^{-1}(\zeta_n)$ has at least
$2^{n-3}$ vertices of degree 3.

In any case, the number of vertices of $\rho_n^{-1}(\zeta_n)$ goes
to infinity and $\rho_\infty^{-1}(\zeta)$ can not be a finite
graph, if $F^n(x_1, x_2, x_3)$ never has less than two non-zero
coordinates. Note that $F(0, 1, 0)=(0, 0, 1)$ and $F(0, 0, 1)=(1,
0, 0)$, which implies that if $F^n(x_1, x_2, x_3)$ has only one
non-zero coordinate, then $F^{n+2}(x_1, x_2, x_3)=(1, 0, 0)$.

If, on the other hand, $F^n(x_1, x_2, x_3)=(1, 0, 0)$, then
$F^m(x_1, x_2, x_3)=(1, 0, 0)$ for all $m\ge n$, and all graphs
$\rho_m^{-1}(\rho_m)$ are isomorphic for $m\ge n$ (with respect to
the homeomorphisms
$\iota_m:\rho_{m+1}^{-1}(\rho_{m+1})\arr\rho_m^{-1}(\rho_m)$) and
are obtained by gluing together a finite number of segments, i.e.,
are finite simplicial graphs.
\end{proof}

See the first six generations of the graphs
$\Phi_{i_1}\circ\cdots\circ\Phi_{i_k}(\rho^{-1}(A'))$ on
Figure~\ref{fig:graphs}. The corresponding slices of the Julia set
of the rational map $f$ are shown on Figure~\ref{fig:smoothj}.

\begin{figure}
\centering
  \includegraphics{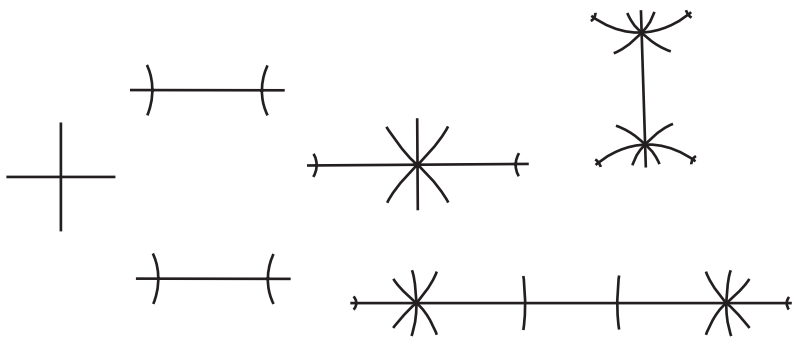}\\
  \caption{Piecewise smooth slices of the Julia set}\label{fig:smoothj}
\end{figure}

\providecommand{\bysame}{\leavevmode\hbox
to3em{\hrulefill}\thinspace}

\end{document}